\documentclass[leqno]{article}
\usepackage[all]{xy}
\usepackage{amsthm}
\usepackage{amsmath}
\usepackage{amssymb}
\usepackage{amsfonts}
\usepackage{graphicx}
\usepackage{amscd}
\usepackage{mathrsfs}
\usepackage{color}
\usepackage[top=30truemm,bottom=30truemm,left=35truemm,right=35truemm]{geometry}
\theoremstyle{Definition}
\newtheorem{Remark}{Remark}[subsection]
\newtheorem*{Remark*}{Remark}
\newtheorem{Theorem}[Remark]{Theorem}
\newtheorem*{Main Theorem*}{Main Theorem}
\newtheorem*{Proof of Main Theorem*}{Proof of Main Theorem}
\newtheorem{Definition}[Remark]{Definition}
\newtheorem*{Definition*}{Definition}
\newtheorem{Proposition}[Remark]{Proposition}
\newtheorem*{Proposition*}{Lemma}
\newtheorem{Lemma}[Remark]{Lemma}
\newtheorem*{Lemma*}{Lemma}
\newtheorem{Corollary}[Remark]{Corollary}
\newtheorem{Hypothesis}[Remark]{Hypothesis}
\newtheorem*{Hypothesis*}{Hypothesis}

\newcommand{\plim}[1][]{\mathop{\varprojlim}\limits_{#1}}

\newcommand{\jump}[1]{\ensuremath{[\![#1]\!]} }

\begin{document}
\title{Triple product $p$-adic $L$-function attached to $p$-adic families of modular forms}
\author{Kengo Fukunaga}
\date{}
\maketitle
\begin{abstract}
Ming-Lun Hsieh constructed three-variable triple product $p$-adic $L$-functions attached to triples of primitive Hida families and proved interpolation formulas. We generalize his result in the unbalanced case and construct a three-variable triple product $p$-adic $L$-function attached to a primitive Hida family and two more general $p$-adic families of modular forms.
\end{abstract}
\setcounter{tocdepth}{1}
\tableofcontents
\section{Introduction.}
Let $p$ be an odd prime. We construct a three-variable triple product $p$-adic $L$-function which interpolates central critical values of triple product $L$-functions in the unbalanced case. Let $K$ be a finite extension of the field  of $p$-adic numbers $\mathbb{Q}_{p}$ and $\mathcal{O}_{K}$ the ring of integers of $K$. Let $\mathbf{I}_{i}$ be a normal finite flat extension of the Iwasawa algebra $\Lambda:=\mathcal{O}_{K}\jump{\Gamma}$ of the topological group $\Gamma:=1+p\mathbb{Z}_{p}$ for $i=1,2,3$. We denote by $\mathbf{m}_{i}$ the unique maximal ideal of $\mathbf{I}_{i}$. Let $\overline{\mathbb{Q}}_{p}$ be the algebraic closure of $\mathbb{Q}_{p}$. We fix an isomorphism $i_{p}:\overline{\mathbb{Q}}_{p}\cong \mathbb{C}$ over $\overline{\mathbb{Q}}$ throughout the paper. Let $\omega_{p}$ be the Teichm$\ddot{\mathrm{u}}$ler character mod $p$. Let $(N_{1},N_{2},N_{3})$ be a triple of positive integers which are prime to $p$ and $(\psi_{1},\psi_{2},\psi_{3})$ a triple of Dirichlet characters of modulo $(N_{1}p,N_{2}p,N_{3}p)$ which satisfies the following hypothesis.\footnote[0]{2010 {\it Mathematics Subject Classification.} Primary 11F67; Secondary 11F33.} 
\begin{Hypothesis*}[1]
There exists an integer $a\in\mathbb{Z}$ such that $\psi_{1}\psi_{2}\psi_{3}=\omega_{p}^{2a}$.
\end{Hypothesis*}
For $i=1,2,3$, we fix a set of non-zero continuous $\mathcal{O}_{K}$-algebraic homomorphisms 
$$\mathfrak{X}^{(i)}:=\{Q_{m}^{(i)}:\mathbf{I}_{i}\rightarrow \overline{\mathbb{Q}}_{p}\}_{m\in\mathbb{Z}_{\geq 1}}.$$
Then, we take a formal series $G^{(i)}\in \mathbf{I}_{i}\jump{q}$ such that for each $m\in\mathbb{Z}_{\geq 1}$, the specialization
$$G^{(i)}(m):=\sum Q^{(i)}_{m}(a(n,G^{(i)}))q^{n}\in\overline{\mathbb{Q}}_{p}\jump{q}$$
is a Fourier expansion of a normalized cuspidal Hecke eigenform of weight $k^{(i)}(m)$, level $N_{i}p^{e^{(i)}(m)}$ and Nebentypus $\psi_{i}\omega_{p}^{-k^{(i)}(m)}\epsilon^{(i)}_{m}$ which is primitive outside of $p$. Here, $k^{(i)}(m)$ and $e^{(i)}(m)$ are positive integers and $\epsilon^{(i)}_{m}$ is a finite character of $\Gamma$. For $i=1$, we take $G^{(1)}$ to be a primitive Hida family $F$ and we take $\mathfrak{X}^{(1)}$ to be the set $\mathfrak{X}_{\mathbf{I}_{1}}$ of arithmetic points $Q$ with weight $k_{Q}\geq 2$ and a finite part $\epsilon_{Q}$. Throughout this section, we fix the triple $(F,G^{(2)},G^{(3)})$ and $(\mathfrak{X}_{\mathbf{I}_{1}},\mathfrak{X}^{(2)},\mathfrak{X}^{(3)})$. Let $R:=\mathbf{I}_{1}\widehat{\otimes}_{\mathcal{O}_{K}}\mathbf{I}_{2}\widehat{\otimes}_{\mathcal{O}_{K}}\mathbf{I}_{3}$ be the complete tensor product of $\mathbf{I}_{1},\mathbf{I}_{2}$ and $\mathbf{I}_{3}$. We define an unbalanced domain of interpolation points of $R$ to be
$$\mathfrak{X}_{R}^{F}:=\left\{\underline{Q}=\left(Q_{1},Q^{(2)}_{m_{2}},Q^{(3)}_{m_{3}}\right)\in\mathfrak{X}_{\mathbf{I}_{1}}\times\mathfrak{X}^{(2)}\times\mathfrak{X}^{(3)}\middle| \begin{tabular}{l} $k_{Q_{1}}+k^{(2)}(m_{2})+k^{(3)}(m_{3})\equiv 0 \ (\mathrm{mod}\ 2)$,\\ $k_{Q_{1}}\geq k^{(2)}(m_{2})+k^{(3)}(m_{3})$.\end{tabular}\right\}.$$
To state our result precisely, we prepare some of notation. Let $\Pi_{\underline{Q}}:=\Pi_{\underline{Q}}^{\prime}\otimes(\chi_{\underline{Q}})_{\mathbb{A}}$ be an automorphic representation of $\mathrm{GL}_{2}(\mathbb{A})$ for each $\underline{Q}=\left(Q_{1},Q^{(2)}_{m_{2}},Q^{(3)}_{m_{3}}\right)\in\mathfrak{X}_{R}^{F}$. Here, $\Pi_{\underline{Q}}^{\prime}$ is the product of the triple of automorphic representations attached to the triple $(F,G^{(2)},G^{(3)})(\underline{Q})$ and $(\chi_{\underline{Q}})_{\mathbb{A}}$ is the adelization of the following Dirichlet character 
$$\chi_{\underline{Q}}:=\omega_{p}^{\frac{1}{2}(2a-k_{Q_{1}}-k^{(2)}(m_{2})-k^{(3)}(m_{3}))}(\epsilon_{Q_{1}}\epsilon^{(2)}_{m_{2}}\epsilon^{(3)}_{m_{3}})^{\frac{1}{2}}.$$
Let $\epsilon(s,\Pi_{\underline{Q},l})$ be the epsilon factor of $\Pi_{\underline{Q},l}$ defined in \cite{Ike92} for each finite prime $l$.

We set $N=N_{1}N_{2}N_{3}$. We summarize some hypotheses to state Main Theorem.
\begin{Hypothesis*}[2]
The residual Galois representation $\overline{\rho}_{F}:=\rho_{F}\ \mathrm{mod}\ \mathbf{m}_{1}:\mathrm{Gal}(\overline{\mathbb{Q}}\slash \mathbb{Q})\rightarrow \mathrm{GL}_{2}(\overline{\mathbb{F}}_{p})$ attached to $F$ is absolutely irreducible as $\mathrm{Gal}(\overline{\mathbb{Q}}\slash\mathbb{Q})$-module and $p$-distinguished in the sense that the semi-simplification of $\overline{\rho}_{F}$ restricted to $\mathrm{Gal}(\overline{\mathbb{Q}}_{p}\slash\mathbb{Q}_{p})$-module is a sum of two different characters.
\end{Hypothesis*}
\begin{Hypothesis*}[3]
The number $\mathrm{gcd}(N_{1},N_{2},N_{3})$ is square free.
\end{Hypothesis*}
\begin{Hypothesis*}[4]
For each $\underline{Q}\in\mathfrak{X}_{R}^{F}$ and for each prime $l\vert N$, we have $\epsilon(1/2,\Pi_{\underline{Q},l})=1$.
\end{Hypothesis*}
\begin{Hypothesis*}[5]
For each $n\in\mathbb{Z}_{\geq 1}$ which is prime to $p$ and for  $i=2,3$, there exits an element $\langle n\rangle^{(i)}\in \mathbf{I}_{i}$  which satisfies
$$Q^{(i)}_{m}(\langle n\rangle^{(i)})=\epsilon^{(i)}_{m}(n)(n\omega_{p}^{-1}(n))^{k^{(i)}(m)}$$
for any $m\in\mathbb{Z}_{\geq 1}$.
\end{Hypothesis*}
\begin{Hypothesis*}[6]
For $i=2,3$ and for each $m\in\mathbb{Z}_{\geq 1}$, we have $a(p,G^{(i)}(m))\neq0$ or $G^{(i)}(m)$ is primitive. 
\end{Hypothesis*}
\begin{Hypothesis*}[7]
For each prime $l\vert N$, the $l$-th Fourier coefficients of $F,G^{(2)}$ and $G^{(3)}$ are non-zero. 
\end{Hypothesis*}
 Let $L(s,\Pi_{\underline{Q}})$ be the triple product $L$-function attached to $\Pi_{\underline{Q}}$ defined in \S3.2. Let $\Omega_{F_{Q_{1}}}$ be the canonical period defined in Definition 3.3.4 and $\mathcal{E}_{F_{Q_{1}}}(\Pi_{\underline{Q},p})$ be the modified $p$-Euler factor defined in (3.4.1). Our main theorem is as follows.
\begin{Main Theorem*}
Let us assume Hypotheses (1)`(7). Then there exists a unique element $\mathcal{L}^{F}_{G^{(2)},G^{(3)}}\in R$ such that we have the interpolation property :
$$(\mathcal{L}^{F}_{G^{(2)},G^{(3)}}(\underline{Q}))^{2}=\mathcal{E}_{F_{Q_{1}}}(\Pi_{\underline{Q},p})\cdot\frac{L(\frac{1}{2},\Pi_{\underline{Q}})}{(\sqrt{-1})^{2k_{Q_{1}}}\Omega_{F_{Q_{1}}}^{2}}\leqno(\mathrm{MT})$$
for every $\underline{Q}=(Q_{1},Q^{(2)}_{m_{2}},Q^{(3)}_{m_{3}})\in \mathfrak{X}^{F}_{R}.$
\end{Main Theorem*}
Our main theorem is an axiomatic generalization of the result of the unbalanced case in \cite{Hsi17} where he takes $G^{(2)},G^{(3)}$ to be primitive Hida families. Hence, we can take more general families for $G^{(2)},G^{(3)}$ such as Coleman families or CM-families.

We give the outline of the proof of Main Theorem. We use the Hida's $p$-adic Rankin-Selberg convolution to construct a triple product $p$-adic $L$-function. Let $ \mathbf{S}^{\mathrm{ord}}(N_{1},\psi_{1,(p)}\overline{\psi}_{1}^{(p)},\mathbf{I}_{1})$ be the $\mathbf{I}_{1}$-module consisting of $\mathbf{I}_{1}$-adic ordinary cusp forms of tame level $N_{1}$ and Nebentypus $\psi_{1,(p)}\overline{\psi}_{1}^{(p)}$, where $\psi_{1,(p)}$ (resp. $\psi_{1}^{(p)}$) is the $p$-part (resp. $N_{1}$-part) of the Dirichlet character $\psi_{1}$. We denote by $\mathbf{T}^{\mathrm{ord}}(N_{1},\psi_{1,(p)}\overline{\psi}_{1}^{(p)},\mathbf{I}_{1})$ the Hecke algebra of $\mathbf{S}^{\mathrm{ord}}(N_{1},\psi_{1,(p)}\overline{\psi}_{1}^{(p)},\mathbf{I}_{1})$. Let $\breve{F}\in\mathbf{S}^{\mathrm{ord}}(N_{1},\psi_{1,(p)}\overline{\psi}_{1}^{(p)},\mathbf{I}_{1})$ be the primitive Hida family attached to the twisted Hida family $F\vert[\overline{\psi}_{1}^{(p)}]$. We denote by $1_{\breve{F}}\in\mathbf{T}^{\mathrm{ord}}(N_{1},\psi_{1,(p)}\overline{\psi}_{1}^{(p)},\mathbf{I}_{1})\otimes_{\mathbf{I}_{1}}\mathrm{Frac}\mathbf{I}_{1}$ and $\eta_{\breve{F}}\in \mathbf{I}_{1}$ the idempotent element attached to $\breve{F}$ and the congruence number respectively. By Hypothesis (7), we have the rigidity of automorphic types in $G^{(i)}$ and then we can define a formal series $G^{(i),*}$ in \S4.1. Further, we construct an $R$-adic form $H^{\mathrm{aux}}\in\mathbf{S}^{\mathrm{ord}}(N,\psi_{1,(p)}\overline{\psi}_{1}^{(p)},\mathbf{I}_{1})\otimes_{\mathbf{I}_{1}}R$ attached to the pair $(G^{(2),*},G^{(3),*})$. The triple product $p$-adic $L$-function $L^{F}_{G^{(2)},G^{(3)}}$ is given by
$$L_{G^{(2)},G^{(3)}}^{F}:=\mathrm{the\ first\ Fourier\ coefficient\ of\ the\ series}\ \eta_{\breve{F}}1_{\breve{F}}\mathrm{Tr}_{N\slash N_{1}}(H^{\mathrm{aux}}),$$
 where $\mathrm{Tr}_{N\slash N_{1}}:\mathbf{S}^{\mathrm{ord}}(N,\psi_{1,(p)}\overline{\psi}_{1}^{(p)},\mathbf{I}_{1})\rightarrow\mathbf{S}^{\mathrm{ord}}(N_{1},\psi_{1,(p)}\overline{\psi}_{1}^{(p)},\mathbf{I}_{1})$ is the trace map.

To prove the interpolation formula, we relate the value $L_{G^{(2)},G^{(3)}}^{F}(\underline{Q})$ at $\underline{Q}$ to the global trilinear period integral $I(\rho(\mathbf{t}_{n})\phi_{\gamma})$ defined in page 18. By \cite[Proposition 3.10]{Hsi17}, we decompose the square of the integral $I(\rho(\mathbf{t}_{n})\phi_{\gamma})$ locally in \S3.3. We calculate the local integrals in \S3.4. We check that the value of the local zeta integral at each $q\mid N$ can be interpolated by a fudge factor $\mathfrak{f}_{q}\in R^{\times}$ in \S5.1. Then, if we put $\mathfrak{f}:=\prod_{q\mid N}\mathfrak{f}_{q}$, the element $\mathcal{L}^{F}_{G^{(2)},G^{(3)}}:=L_{G^{(2)},G^{(3)}}^{F}\slash(-\psi_{1,(p)}(-1))^{\frac{1}{2}}\mathfrak{f}^{\frac{1}{2}}\in R$ satisfies the desired interpolation formula.
\\
\\
{\bf Organization of the paper.} This paper is organized as follows. In \S2, we prepare some of notation and some known results on automorphic forms, irreducible representations of $\mathrm{GL}_{2}$ and $p$-adic families of modular forms.

In \S3, we define triple product $L$-functions, global trilinear period integrals and local zeta integrals. Then, we decompose the global trilinear period integral to the product of the triple product $L$-function and local zeta integrals by \cite[Proposition 3.10]{Hsi17}. Further we calculate the local zeta integrals.

In \S4, we prove the rigidity of automorhpic types in $G^{(i)}$, and construct a triple product $p$-adic $L$-function attached to the triple $(F,G^{(2)},G^{(3)})$.

In \S5, we prove Main Theorem and give some examples of $G^{(i)}$.

In Appendix, we prove the rigidity of conductors in a Coleman family.
\begin{Remark}
Hypotheses (1)`(4) are the same as in \cite{Hsi17}. On the other hand, other hypotheses are added to generalize the result of \cite{Hsi17}. We need Hypothesis (5) to construct a triple product $p$-adic $L$-function in \S4. If $G^{(i)}$ is a $\Lambda$-adic Hida family, the element $\langle n\rangle^{(i)}\in\mathbf{I}_{i}^{\times}$ in Hypothesis (5) is given by the element $[n\omega_{p}^{-1}(n)]\in\Lambda^{\times}$ defined in page 13. We need Hypothesis (6) to calculate a local zeta integral at $p$ in Proposition 3.4.7. We need Hypothesis (7) so that Hypothesis 3.4.2 and the rigidity of the automorphic types in $(F,G^{(2)},G^{(3)})$ hold. 

\end{Remark}
\noindent{\bf Notation.} We prepare some notation. Let $\mathbb{Q}_{v}$ be the completion field of $\mathbb{Q}$ at each place $v$. We denote by $\mathbb{A}$ (resp. $\mathbb{A}_{\mathrm{fin}}$) the adele (resp. finite adele) of $\mathbb{Q}$. For each finite prime $l$, we denote by $v_{l}$ the additive valuation of $\mathbb{Q}_{l}$ with $v_{l}(l)=1$ and by $|\cdot|_{l}$ the normalized norm of $\mathbb{Q}_{l}$ such that $|z|_{l}=l^{-v_{l}(z)}$ for every $z\in\mathbb{Q}_{l}$. Further, we denote by $|\cdot|_{\infty}$ the usual norm of $\mathbb{R}$ such that $|x|_{\infty}=\mathrm{max}\{x,-x\}$ for every $x\in \mathbb{R}$. By the product of local norms, we define the global norm $|\cdot|_{\mathbb{A}}:=\prod_{v}|\cdot|_{v}$ on $\mathbb{A}$. We set $\mathbb{R}_{+}:=\{x\in \mathbb{R}\vert x>0\}$ and $\widehat{\mathbb{Z}}=\displaystyle{\prod_{l:\mathrm{finite\ prime}}}\mathbb{Z}_{l}$. For each prime $l$, we denote by $j_{l}:\mathbb{Q}_{l}^{\times}\rightarrow \mathbb{A}^{\times}$ the natural embedding.

Let $\zeta_{v}(s)$ be the local zeta function for each place $v$ defined by
$$\zeta_{v}(s):= \begin{cases}
   \pi^{-\frac{s}{2}}\Gamma(\frac{s}{2})  & \mathrm{if}\ v=\infty, \\
    (1-v^{-s})^{-1} & \mathrm{if}\ v\neq \infty.
  \end{cases}$$
We define the global zeta function to be $\zeta_{\mathbb{A}}(s):=\prod_{v}\zeta_{v}(s)$ for $\mathrm{Re}(s)>1$. It is well-known that $\zeta_{\mathbb{A}}(s)$ has a meromorphic continuation on $\mathbb{C}$ and a functional equation. 

Let $\psi_{\mathbb{A}}=\prod_{v}\psi_{\mathbb{Q}_{v}}:\mathbb{A}\slash\mathbb{Q}\rightarrow \mathbb{C}^{\times}$ be an additive unitary character such that $\psi_{\mathbb{R}}=\mathrm{exp}(2\pi\sqrt{-1}x)$ for every $x\in \mathbb{R}$ and the conductor of $\psi_{\mathbb{Q}_{l}}$ equals to $\mathbb{Z}_{l}$ for each finite prime $l$.

Let $A$ be a commutative ring. We denote by $\rho$ the right regular representation of $\mathrm{GL}_{2}(A)$ on the space of functions $f: \mathrm{GL}_{2}(A)\rightarrow \mathbb{C}$.  For each function $f:\mathrm{GL}_{2}(A)\rightarrow \mathbb{C}$ and each character $\chi:A^{\times}\rightarrow \mathbb{C}^{\times}$, we denote by $f\otimes\chi$ the twisted function on $\mathrm{GL}_{2}(A)$ such that $f\otimes\chi (g):=f(g)\chi(\mathrm{det}(g))$ for each $g\in \mathrm{GL}_{2}(A)$.
 
Let $\chi$ be a Dirichlet character modulo $M\in\mathbb{Z}_{\geq 1}$. We  denote by $m_{\chi}$ the conductor of $\chi$.  By the class field theory, there exists a unique Hecke character $\chi_{\mathbb{A}}:\mathbb{A}^{\times}\slash\mathbb{Q}^{\times}\mathbb{R}_{+}(1+M\widehat{\mathbb{Z}})\rightarrow\mathbb{C}^{\times}$ such that $\chi_{\mathbb{A}}(j_{l}(l))=\chi(l)^{-1}$ for each $l\nmid M$. We call the Hecke character $\chi_{\mathbb{A}}$ the adelization of the character $\chi$. For each prime $l$, we write $\chi=\chi_{(l)}\chi^{(l)}$, where $\chi_{(l)}$ is a Dirichlet character modulo $l^{v_{l}(M)}$ and $\chi^{(l)}$ is a Dirichlet character modulo $Ml^{-v_{l}(M)}$.

Let $G$ be a locally compact Hausdorff group and $\mathrm{d}g$ a Haar measure on $G$. For each measurable subset $U$ of $G$, we denote by $\mathrm{vol}(U,\mathrm{d}g)$ the volume of $U$ attached to $\mathrm{d}g$. We denote by $\mathrm{d}_{\infty}x$ or $\mathrm{d}_{\infty}y$ the Haar measure on $\mathbb{R}$ such that $\mathrm{vol}([0,1],\mathrm{d}_{\infty}x)=\mathrm{vol}([0,1],\mathrm{d}_{\infty}y)=1$. We define the Haar measure on $\mathbb{R}^{\times}$ to be $\mathrm{d}^{\times}_{\infty}y:=|y| ^{-1}_{\infty}\mathrm{d}_{\infty}y$. For each prime $l$, we denote by $\mathrm{d}_{l}x$ and $\mathrm{d}^{\times}_{l}y$ the Haar measures on $\mathbb{Q}_{l}$ and $\mathbb{Q}_{l}^{\times}$ with $\mathrm{vol}(\mathbb{Z}_{l},\mathrm{d}_{l}x)=1$ and $\mathrm{vol}(\mathbb{Z}_{l}^{\times},\mathrm{d}_{l}^{\times}y)=1$ respectively. By the product of local Haar measures, we define the global Haar measures $\mathrm{d}_{\mathbb{A}}x:=\prod_{v}\mathrm{d}_{v}x$ and $\mathrm{d}_{\mathbb{A}}^{\times}y: =\prod_{v}\mathrm{d}_{v}^{\times}y$ on $\mathbb{A}\slash\mathbb{Q}$ and $\mathbb{A}^{\times}\slash\mathbb{Q}^{\times}$ respectively. Let $\mathbf{K}_{v}$ be the subgroup of $\mathrm{GL}_{2}(\mathbb{Q}_{v})$ defined by
$$\mathbf{K}_{v}:= \begin{cases}
   \mathrm{O}(2,\mathbb{R})  & \mathrm{if}\ v=\infty, \\
    \mathrm{GL}_{2}(\mathbb{Z}_{v}) & \mathrm{if}\ v\neq \infty.
  \end{cases}$$
Let $\mathrm{d}_{v}k$ be the Haar measure on $\mathbf{K}_{v}$ with $\mathrm{vol}(\mathbf{K}_{v},\mathrm{d}_{v}k)=1$. We define the Haar measure $\mathrm{d}_{v}g$ on $\mathrm{PGL}_{2}(\mathbb{Q}_{v})$ to be 
$$\mathrm{d}_{v}g:=|y|_{v}^{-1}\mathrm{d}_{v}x\mathrm{d}^{\times}_{v}y\mathrm{d}_{v}k$$
for $g=\left(
    \begin{array}{cc}
      y&x \\
      0&1 \\
    \end{array}
  \right)k$, where $y\in \mathbb{Q}_{v}^{\times}, x\in \mathbb{Q}_{v}$ and $k\in \mathbf{K}_{v}$.
\section{Preliminaries.}
In \S2, we summarize some of notation and results of \cite{Hsi17} on automorphic forms, irreducible representations of $\mathrm{GL}_{2}$ and $p$-adic families of modular forms. Further, we define  primitive $p$-adic families of modular forms in Definition 2.5.6.
\subsection{Automorphic forms on the upper half plane.}
Let $\mathfrak{H}:=\{z\in\mathbb{C}\mid\mathrm{Im}(z)>0\}$ be the upper half plane. In this subsection, we recall definitions of automorphic forms on $\mathfrak{H}$ and the Petersson inner product. Further, we introduce some operators on the automorphic forms and recall definitions of Hecke eigenforms, primitive forms and $p$-stabilized newforms. We denote by $C^{\infty}(\mathfrak{H})$ the space of $\mathbb{C}$-valued infinitely differentiable functions on $\mathfrak{H}$. We denote by $\mathrm{GL}_{2}^{+}(\mathbb{R})$ the subgroup of $\mathrm{GL}_{2}(\mathbb{R})$ consisting of matrices of positive determinant. Let $k$ be an integer and $f$ be a $\mathbb{C}$-valued function on $\mathfrak{H}$. For each $\gamma=\left(
    \begin{array}{cc}
      a&b \\
      c&d \\
    \end{array}
  \right)\in \mathrm{GL}_{2}^{+}(\mathbb{R})$,  we denote by $f\vert_{k}\gamma$ the $\mathbb{C}$-valued function on $\mathfrak{H}$ defined by
$$f\vert_{k}\gamma(z):=f(\gamma(z))(cz+d)^{-k}(\mathrm{det}\gamma)^{\tfrac{k}{2}}.\leqno(2.1.1)$$
We define the Maass-Shimura differential operators on $C^{\infty}(\mathfrak{H})$ ($cf$. \cite[page 310]{Hid93}) to be
\begin{align*}
\delta_{k}&:=\frac{1}{2\pi\sqrt{-1}}\left(\frac{\partial}{\partial z}+\frac{k}{2\sqrt{-1}\mathrm{Im}(z)}\right),\\
\epsilon&:=-\frac{1}{2\pi\sqrt{-1}}\mathrm{Im}(z)^{2}\frac{\partial}{\partial \overline{z}},\tag{2.1.2}\\
d&:=\frac{1}{2\pi\sqrt{-1}}\frac{\partial}{\partial z},
\end{align*}
and we write $\delta_{k}^{m}:=\delta_{k+2m-2}\ldots\delta_{k+2}\delta_{k}$ for each $m\in\mathbb{Z}_{\geq0}$. For each positive integer $M\in \mathbb{Z}_{\geq 1}$, we denote by $\Gamma_{0}(M)$ the congruence subgroup of $\mathrm{SL}_{2}(\mathbb{Z})$ consisting of all matrices $\left(
    \begin{array}{cc}
      a&b \\
      c&d \\
    \end{array}
  \right)$\ with $c\equiv 0\ \mathrm{mod}\ M$.

Next, we recall the definition of nearly holomorphic modular forms referring to \cite[page 314]{Hid93}.
\begin{Definition}
Let $k,r$ be positive integers and $\chi$ be a Dirichlet character modulo $M\in\mathbb{Z}_{\geq 1}$. We call a function $f\in C^{\infty}(\mathfrak{H})$ a nearly holomorphic modular form of weight $k$, level $M$, Nebentypus $\chi$ and order $\leq r$ if $f$ satisfies the following conditions.
\begin{enumerate}
\item $f \vert_{k}\gamma=\chi(d)f$ for any $\gamma=\left(
    \begin{array}{cc}
      a&b \\
      c&d \\
    \end{array}
  \right)\in \Gamma_{0}(M)$,\\
\item $\epsilon^{r+1}(f)=0$,\\
\item $f$ is slowly increasing, that is, for any $g\in\mathrm{GL}_{2}^{+}(\mathbb{Q})$, there exists $C,M_{1},M_{2}\in \mathbb{R}_{+}$ such that $\vert f \vert_{k}g(z)\vert_{\infty}<C\vert \mathrm{Im}(z)\vert_{\infty}^{M_{1}}$ at $\mathrm{Im}(z)>M_{2}$.
\end{enumerate}
In addition, if $f$ is rapid decreasing, that is, for any $g\in\mathrm{GL}_{2}^{+}(\mathbb{Q})$ and $M_{1}\in \mathbb{R}_{+}$, $\vert (f\vert_{k}g(z))\cdot\mathrm{Im}(z)^{M_{1}}\vert_{\infty}\mapsto 0$ at $\mathrm{Im}(z)\mapsto \infty$, we call $f$ a nearly holomorphic cusp form.
\end{Definition}
We denote by $\mathcal{N}_{k}^{[r]}(M,\chi)$ the space of nearly holomorphic modular forms of weight $k$, level $M$, Nebentypus $\chi$ and order $\leq r$. In particular, we call an element $f\in\mathcal{M}_{k}(M,\chi):=\mathcal{N}_{k}^{[0]}(M,\chi)$ a holomorphic modular form. It is easy to check that the Maass-Shimura operator $\delta_{k}$ induces the map $\delta_{k}:\mathcal{N}_{k}^{[r]}(M,\chi)\rightarrow \mathcal{N}_{k+2}^{[r+1]}(M,\chi)$ ($cf$. \cite[page 312]{Hid93}). We give the Fourier expansion of $f\in \mathcal{M}_{k}(M,\chi)$ at $\infty$ by
$$f(z)=\displaystyle{\sum_{n\geq 0}}a(n,f)q^{n},$$
where
 $q=\mathrm{exp}(2\pi\sqrt{-1}z)$. Let $\mathcal{S}_{k}(M,\chi)$ be the subspace of $\mathcal{M}_{k}(M,\chi)$ consisting of cusp forms. We define a $\mathbf{Z}[\chi]$-module $S_{k}(M,\chi,\mathbf{Z}[\chi])$ to be
$$S_{k}(M,\chi,\mathbf{Z}[\chi]):=\{f\in S_{k}(M,\chi)\vert a(n,f)\in \mathbb{Z}[\chi],\ \forall n\in\mathbb{Z}_{\geq1}\}.$$
Further, for any $\mathbb{Z}[\chi]$-algebra $A$, we set $S_{k}(M,\chi,A):=S_{k}(M,\chi,\mathbf{Z}[\chi])\otimes_{\mathbb{Z}[\chi]}A$.

Next, we recall the definition of the Petersson inner product. For each $f,g\in \mathcal{S}_{k}(M,\chi)$, the Petersson inner product $\langle f,g\rangle_{\Gamma_{0}(M)}$ is defined by 

$$\langle f,g\rangle_{\Gamma_{0}(M)}:=\int_{\Gamma_{0}(M)\backslash\mathfrak{H}} f(x+\sqrt{-1}y)\overline{g(x+\sqrt{-1}y)}y^{k}\frac{\mathrm{d}_{\infty}x\mathrm{d}_{\infty}y}{y^{2}}.$$
In particular, for each $f\in \mathcal{S}_{k}(M,\chi)$, we define the Petterson norm of $f$ to be $\|f\|_{\Gamma_{0}(M)}^{2}=\langle f,f\rangle_{\Gamma_{0}(M)}$. 

Next, we recall the definitions of the Hecke operators on $\mathcal{N}_{k}^{[r]}(M,\chi)$ at each finite prime $l$. For each $d\in\mathbb{Z}_{\geq 1}$, let $V_{d}$ and $U_{d}$ be operators on $C^{\infty}(\mathfrak{H})$ defined by
$$V_{d}f(z)=df(dz);\ \ U_{d}=\frac{1}{d}\displaystyle{\sum_{j=0}^{d-1}}f(\frac{z+j}{d}).\leqno(2.1.3)$$
It is easy to check that $V_{d}(f)\in \mathcal{N}_{k}^{[r]}(Md,\chi)$ for each $f\in \mathcal{N}_{k}^{[r]}(M,\chi)$. The Hecke operator $T_{l}$ on $\mathcal{N}_{k}^{[r]}(M,\chi)$ at $l$ is defined by
$$T_{l}f:=\begin{cases}
U_{l}f\ &\mathrm{if}\ l\vert M,\\
U_{l}f+\chi(l)l^{k-2}V_{l}f\ &\mathrm{if}\ l\nmid M.
\end{cases}\leqno(2.1.4)$$
Next, we recall the definition of the twisting operator for each Dirichlet character . Let $\kappa$ be a Dirichlet character modulo $L\in\mathbb{Z}_{\geq 1}$. For each modular form $f\in\mathcal{M}_{k}(M,\chi)$, the twisted modular form $f\mid[\kappa]$ is defined by
$$f\vert[\kappa](z)=\sum_{n\geq0}\kappa(n)a(n,f)q^{n}.\leqno(2.1.5)$$
It is known that $f\vert[\kappa]\in\mathcal{M}_{k}(ML^{2},\kappa\chi)$.

We say that $f\in\mathcal{M}_{k}(M,\chi)$ is a Hecke eigenform if $f$ is an eigenfunction of $T_{l}$ for any prime $l$. In particular, if $a(1,f)=1$, we call $f$ a normalized Hecke eigenform. We denote by $S^{1}_{k}(M,\chi)$ the subspace of $S_{k}(M,\chi)$ generated by the elements
$$\displaystyle{\bigcup_{L}\bigcup_{P}}\{V_{P}f \mid f\in S_{k}(L,\chi)\}.$$
Here, $L$ (resp. $P$) runs through all positive integers such that $m_{\chi}\vert L$, $L\vert M$ and $L\neq M$ (resp. $P\vert M\slash L$). We say that a normalized Hecke eigenform $f\in\mathcal{S}_{k}(M,\chi)$ is primitive if $f$ is contained in the orthogonal space $S^{0}(M,\chi):=S^{1}(M,\chi)^{\perp}$ with respect to the Petersson inner product. It is known that for each normalized Hecke eigenform $f\in\mathcal{S}_{k}(M,\chi)$, there exists a unique primitive form $f^{\mathrm{prim}}\in S^{0}(m_{f},\chi)$ with $m_{f}\vert M$ such that $a(l,f^{\mathrm{prim}})=a(l,f)$ for all but a finite number of primes $l$. We call $f^{\mathrm{prim}}$ and $m_{f}$ the primitive form attached to $f$ and the conductor of $f$ respectively.

Next, we recall the definition of $p$-stabilized newforms.
\begin{Definition}
We assume $p\nmid M$. Let $f\in \mathcal{S}_{k}(Mp^{n},\chi)$ be a normalized Hecke eigenform with $n\in\mathbb{Z}_{\geq 1}$. We fix an element $\alpha\in\mathbb{Q}_{\geq 0}\cup\{\infty\}$. We call $f$ a slope $\alpha$ $p$-stabilized newform of tame level $M$ if $M$ divides the conductor of $f$ and $\vert a(p,f)\vert_{p}=p^{-\alpha}$. In particular, if $\alpha=0$, we call $f$ an ordinary $p$-stabilized newform.  
\end{Definition}

\subsection{Automorphic forms on $\mathrm{GL}_{2}(\mathbb{A}$).}
In this subsection, we recall some of definitions and properties on automorphic forms on $\mathrm{GL}_{2}(\mathbb{A})$. Further, at the end of this subsection, we recall the definition of the adelic lifts of classical forms and give the relations between the operators defined in the previous subsection and in this subsection. For each element $g\in\mathrm{GL}_{2}(\mathbb{A})$, we denote by $g_{\infty}\in\mathrm{GL}_{2}(\mathbb{R})$ (resp. $g_{\mathrm{fin}}\in\mathrm{GL}_{2}(\mathbb{A}_\mathrm{fin})$) the infinite part (resp. finite pair) of $g$. Let $M$ be a positive integer. We define open compact subgroups $U_{0}(M),U_{1}(M)$ of $\mathrm{GL}_{2}(\widehat{\mathbb{Z}})$ to be
\begin{center}
$U_{0}(M)=\left\{ g\in \mathrm{GL}_{2}(\widehat{\mathbb{Z}}) \mid g\equiv \left(
    \begin{array}{cc}
      *&* \\
      0&* \\
    \end{array}
  \right) (\mathrm{mod} M\widehat{\mathbb{Z}})\right\},$

$U_{1}(M)=\left\{ g\in \mathrm{GL}_{2}(\widehat{\mathbb{Z}}) \mid g\equiv \left(
    \begin{array}{cc}
      *&* \\
      0&1 \\
    \end{array}
  \right) (\mathrm{mod} M\widehat{\mathbb{Z}})\right\}.$
\end{center}
Let $\xi=\prod_{v}\xi_{v}: \mathbb{Q}^{\times}\backslash\mathbb{A}^{\times}\rightarrow \mathbb{C}^{\times}$ be a finite order Hecke character of level $M$. We can regard $\xi$ as a character of $U_{0}(M)$ via 
$$\xi\left(\left(
    \begin{array}{cc}
      a&b \\
      c&d \\
    \end{array}
  \right)\right)=\prod_{l\vert M}\xi_{l}(d_{l})$$
 for each $\left(
    \begin{array}{cc}
      a&b \\
      c&d \\
    \end{array}
  \right) \in U_{0}(M)$. Let $\mathbf{K}:=\mathrm{O}_{2}(\mathbb{R})\mathrm{GL}_{2}(\widehat{\mathbb{Z}})$ be the maximal compact subgroup of $\mathrm{GL}_{2}(\mathbb{A})$. We denote by $\mathfrak{g}$ (resp. $\mathfrak{g}_{\mathbb{C}}$) the Lie algebra (resp. the complex Lie algebra) of $\mathrm{GL}_{2}(\mathbb{R})$. We recall the definition of automorphic forms referring to \cite[page 299]{Bum97}.
\begin{Definition}
We say that a function $f:\mathrm{GL}_{2}(\mathbb{A})\rightarrow \mathbb{C}$ is smooth if for each $g\in\mathrm{GL}_{2}(\mathbb{A})$, there exists a $C^{\infty}$-function $f_{g}: \mathrm{GL}_{2}(\mathbb{R})\rightarrow\mathbb{C}$ and a neighborhood $N$ of $g$ in $\mathrm{GL}_{2}(\mathbb{A})$ such that
$f(h)=f_{g}(h_{\infty})$ for every $h=h_{\infty}h_{\mathrm{fin}}\in N$.
\end{Definition}
On the space of smooth functions $f:\mathrm{GL}_{2}(\mathbb{A})\rightarrow\mathbb{C}$, we define the action of $\mathfrak{g}$ to be
$$Xf(g)=\frac{d}{dt}f(g_{\infty}\mathrm{exp}(tX)g_{\mathrm{fin}})\vert_{t=0}$$
for each $X\in\mathfrak{g}$ and $g=g_{\infty}g_{\mathrm{fin}}\in\mathrm{GL}_{2}(\mathbb{A})=\mathrm{GL}_{2}(\mathbb{R})\mathrm{GL}_{2}(\mathbb{A}_{\mathrm{fin}}).$ The above action of $\mathfrak{g}$ can be extend to an action of the universal enveloping algebra $U(\mathfrak{g}_{\mathbb{C}})$ of $\mathfrak{g}_{\mathbb{C}}$. Let $\mathcal{Z}$ be the center of $U(\mathfrak{g}_{\mathbb{C}})$.
\begin{Definition}
Let $\xi$ be a finite order Hecke character of level $M$. We call a smooth function $\varphi:\mathrm{GL}_{2}(\mathbb{Q})\backslash\mathrm{GL}_{2}(\mathbb{A})\rightarrow\mathbb{C}$ an automorphic form of $\mathrm{GL}_{2}(\mathbb{A})$ with the central character $\xi$ if $\varphi$ satisfies the following conditions.
\begin{enumerate}
\item The function $\varphi$ is moderate growth , that is, there exists positive constants $C$ and $n$ such that $\vert\varphi(g)\vert_{\infty}<C\displaystyle{\max_{1\leq i,j\leq2}}\{\vert g_{i,j}\vert_{\mathbb{A}},\vert \mathrm{det}(g)\vert_{\mathbb{A}}^{-1}\}^{n}$ for all $g=\left(
    \begin{array}{cc}
      g_{1,1}&g_{1,2} \\
      g_{2,1}&g_{2,2}
    \end{array}
  \right)\in\mathrm{GL}_{2}(\mathbb{A})$.\\
\item We have $\mathrm{dim}_{\mathbb{C}}\mathrm{Span}\mathbb{C}\{X\varphi\mid X\in\mathcal{Z}\}<\infty$ and $\mathrm{dim}_{\mathbb{C}}\mathrm{Span}\mathbb{C}\{\rho(k)f\mid k\in\mathbf{K}\}<\infty$.\\
\item We have $\varphi(\mathrm{diag}\{z,z\}g)=\xi(z)\varphi(g)\ \mathrm{for\ every}\ z\in\mathbb{A}\ \mathrm{and}\ g\in\mathrm{GL}_{2}(\mathbb{A})$.
\end{enumerate}
In particular, we say that $\varphi$ is cuspidal if $\varphi$ satisfies
$$\int_{\mathbb{A}\slash\mathbb{Q}}\varphi\left(\left(
    \begin{array}{cc}
      1&x \\
      0&1
    \end{array}
  \right)g\right)\mathrm{d}_{\mathbb{A}}x=0$$
for almost all $g\in\mathrm{GL}_{2}(\mathbb{A})$. Further, we say that $\varphi$ is of weight $k\in\mathbb{Z}$, level $M$ and character $\xi$ if $\varphi$ satisfies
 $$\varphi(\alpha gu_{\infty}u_{\mathrm{f}})=\varphi(g)e^{\sqrt{-1}k\theta}\xi(u_{\mathrm{f}})$$
for $\alpha\in \mathrm{GL}_{2}(\mathbb{Q}), u_{\infty}=\left(
    \begin{array}{cc}
      \mathrm{cos}\theta&\mathrm{sin}\theta \\
      -\mathrm{sin}\theta&\mathrm{cos}\theta \\
    \end{array}
  \right)$ and $u_{\mathrm{f}}\in U_{0}(M)$. 
\end{Definition}
We denote by $\mathcal{A}(\xi)$ (resp. $\mathcal{A}^{0}(\xi)$) the space of automorphic forms (resp. cuspidal automorphic forms) of $\mathrm{GL}_{2}(\mathbb{A})$ with the central character $\xi$. Further, we denote by $\mathcal{A}_{k}(M,\xi)$ (resp. $\mathcal{A}^{0}_{k}(M,\xi)$) the space of automorphic forms (resp. cuspidal automorphic forms) of weight $k$, level $M$ and character $\xi$.

Next, we recall the definitions of the weight raising\slash lowering operators, level rising operators, Hecke operators and twisting operators on $\mathcal{A}_{k}(M,\xi)$. The weight raising\slash lowering operators $\mathbf{V}_{\pm}:\mathcal{A}_{k}(M,\xi)\rightarrow \mathcal{A}_{k\pm2}(M,\xi)$ in \cite[page 165]{JL70} are defined by
$$\mathbf{V}_{\pm}=\frac{1}{(-8\pi)}\left(\left(
    \begin{array}{cc}
      1&0 \\
      0&-1 \\
    \end{array}
  \right)\otimes 1 \pm \left(
    \begin{array}{cc}
      0&1 \\
      1&0 \\
    \end{array}
  \right)\otimes \sqrt{-1}\right)\in \mathfrak{g}_{\mathbb{C}}.\leqno(2.2.1)$$
The level-raising operator $\mathbf{V}_{d}:\mathcal{A}_{k}(M,\xi)\rightarrow\mathcal{A}_{k}(Md,\xi)$ for $d\in\mathbb{Z}_{\geq 1}$ is defined by
$$\mathbf{V}_{d}\varphi(g):=\rho(\left(
    \begin{array}{cc}
      j_{d}^{-1}&0 \\
      0&1 \\
    \end{array}
  \right))\varphi,\leqno(2.2.2)$$
where $j_{d}=\prod_{l\vert d}j_{l}(l)\in\mathbb{A}^{\times}$.

Let $l$ be a finite prime and $U_{l}$ be the operator on the space of functions $\varphi: \mathrm{GL}_{2}(\mathbb{A})\rightarrow \mathbb{C}$ defined by
$$\mathbf{U}_{l}\varphi=\displaystyle{\sum_{x=1}^{l}}\rho(\left(
    \begin{array}{cc}
      j_{l}(l)&j_{l}(x) \\
      0&1 \\
    \end{array}
  \right))\varphi.\leqno(2.2.3)$$
The Hecke operator at $l$ on $\mathcal{A}_{k}(M,\xi)$ is defined by
$$\mathbf{T}_{l}\varphi:=\begin{cases}
\mathbf{U}_{l}\varphi\ &\mathrm{if}\ l\vert M,\\
\mathbf{U}_{l}\varphi+\xi(j_{l}(l))\mathbf{V}_{l}\varphi\ &\mathrm{if}\ l\nmid M.
\end{cases}\leqno(2.2.4)$$
Next, we recall the twisting operator. Let $\kappa$ be a Dirichlet character modulo $l^{m}$, where $l$ is a finite prime and $m$ is a non-negative integer. If $m_{\kappa}>0$, the Gauss sum $\mathfrak{g}(\kappa)$ attached to $\kappa$ is defined by
$$\mathfrak{g}(\kappa)=\displaystyle{\sum}_{x\in (\mathbb{Z}\slash m_{\kappa}\mathbb{Z})^{\times}}\kappa^{-1}(x)e^{\frac{-2\pi\sqrt{-1}x}{m_{\kappa}}}.\leqno(2.2.5)$$
For each $\varphi\in\mathcal{A}_{k}(M,\xi)$, the twisted automorphic form $\theta^{\kappa}_{l}\varphi: \mathrm{GL}_{2}(\mathbb{A})\rightarrow \mathbb{C}$ is defined by
$$ \theta^{\kappa}_{l}\varphi=\begin{cases}
    \varphi-l^{-1}\mathbf{V}_{l}\mathbf{U}_{l}\varphi & \mathrm{if}\ m_{\chi}=1,\\
    \mathfrak{g}(\kappa)^{-1}\displaystyle{\sum_{x\in(\mathbb{Z}\slash l^{n}\mathbb{Z})^{\times}}}\kappa^{-1}(x)\rho(\left(
    \begin{array}{cc}
      1&j_{l}(x\slash l^{n}) \\
      0&1 \\
    \end{array}
  \right))\varphi & \mathrm{if}\ m_{\chi}>1.
  \end{cases}\leqno(2.2.6)$$
Next, we recall the definition of the $\mathrm{GL}_{2}(\mathbb{A})$-equivariant bilinear form $\langle\ ,\ \rangle:\mathcal{A}^{0}(\xi)\times\mathcal{A}^{0}(\xi^{-1})\rightarrow \mathbb{C}$ defined in  \cite[\S2.7]{Hsi17}. Let $\mathrm{d}^{\tau}_{\mathbb{A}}g$ be the Tamagawa measure of $\mathrm{PGL}_{2}(\mathbb{A})$.
\begin{Definition}
We define a $\mathrm{GL}_{2}(\mathbb{A})$-equivariant bilinear form $\langle\ ,\ \rangle:\mathcal{A}^{0}(\xi)\times\mathcal{A}^{0}(\xi^{-1})\rightarrow \mathbb{C}$ to be
$$\langle\varphi,\varphi^{\prime} \rangle:=\int_{\mathbb{A}^{\times}\mathrm{GL}_{2}(\mathbb{Q})\backslash\mathrm{GL}_{2}(\mathbb{A})}\varphi(g)\varphi^{\prime}(g)\mathrm{d}_{v}^{\tau}g$$
for each $\varphi\in\mathcal{A}^{0}(\xi),\varphi^{\prime}\in\mathcal{A}^{0}(\xi^{-1}).$
\end{Definition}
We have the following lemma which is proved in \cite[page 217]{Wal85}.
\begin{Lemma}
Let $M,k$ be positive integers. If we take elements  $\varphi\in\mathcal{A}^{0}_{k}(M,\xi)$ and $\varphi^{\prime}\in\mathcal{A}^{0}_{-k}(M,\xi^{-1})$, we have 
\begin{align*}
\langle X\varphi,\varphi^{\prime}\rangle&=-\langle \varphi,X\varphi^{\prime}\rangle\ \ \mathrm{for\ each}\ X\in\mathfrak{g},\\
\langle\varphi,\mathbf{T}_{l}\varphi^{\prime}\rangle&=l\langle\mathbf{V}_{l}\varphi,\varphi^{\prime}\rangle\ \ \mathrm{for\ each}\ l\vert M,\\
\langle\mathbf{T}_{l}\varphi,\varphi^{\prime}\rangle&=\xi(l)\langle\varphi,\mathbf{T}_{l}\varphi^{\prime}\rangle\ \ \mathrm{for\ each}\ l\nmid M.
\end{align*}
\end{Lemma}
Next, we recall the definition of the adelic lifts of classical forms on $\mathfrak{H}$ and give the relation between the classical forms and the automorphic forms on $\mathrm{GL}_{2}(\mathbb{A})$.
\begin{Definition}
For each $f\in \mathcal{N}_{k}^{[r]}(M,\chi)$, we define an element $\Phi(f)\in\mathcal{A}_{k}(M,\chi_{\mathbb{A}}^{-1})$ to be
$$\Phi(f)(g): = (f\vert_{k}g_{\infty})(\sqrt{-1})\cdot \chi_{\mathbb{A}}^{-1}(u)$$
for each $g=\alpha g_{\infty}u\in\mathrm{GL}_{2}(\mathbb{A})$, where $\alpha\in\mathrm{GL}_{2}(\mathbb{Q})$, $g_{\infty}\in\mathrm{GL}_{2}^{+}(\mathbb{R})$ and $u\in U_{0}(M)$ ($cf$. \cite[\S3]{Cas73}. We call $\Phi(f)$ the adelic lift of $f$.
\end{Definition}
We recall the compatibility of some operators via the adelic lift. Let $\delta_{k}^{m},\epsilon, V_{d},T_{l},\vert_{\kappa}$ be the operators on $\mathcal{N}_{k}^{[r]}(M,\chi)$ defined in (2.1.2)`(2.1.5) and $\mathbf{V}_{\pm},\mathbf{T}_{l},\theta^{\kappa}_{l}$ be the operators on $\mathcal{A}(M,\chi_{\mathbb{A}}^{-1})$ defined in (2.2.1),(2.2.4) and (2.2.6). Then, we have the following equations
\begin{align*}
\Phi(\delta_{k}^{m}f)&=\mathbf{V}_{+}^{m}\Phi(f),\\
\Phi(\epsilon f)&=\mathbf{V}_{-}\Phi(f),\\
\Phi(V_{d}f)&=d^{1-\frac{k}{2}}\mathbf{V}_{d}\Phi(f),\tag{2.2.7}\\
\Phi(T_{l}f)&=l^{\frac{k}{2}-1}\mathbf{T}_{l}\Phi(f),\\
\Phi(f\vert[\kappa])&=\theta^{\kappa}_{l}\Phi(f)\otimes\kappa_{\mathbb{A}}^{-1}
\end{align*}
for each $f\in \mathcal{N}_{k}^{[r]}(M,\chi)$.
\subsection{Preliminaries on irreducible representations of $\mathrm{GL}_{2}(\mathbb{Q}_{v})$.}
Let $v$ be a place of $\mathbb{Q}$. In this subsection, we recall some of notation of irreducible representations of $\mathrm{GL}_{2}(\mathbb{Q}_{v})$ and introduce definitions of new lines and ordinary lines which are defined in \cite{Cas73} and \cite{Hsi17}. We also recall the definitions of Whittaker models, normalized local Whittaker newforms and normalized ordinary Whittaker functions defined in \cite{Hsi17}.

Let $l$ be a finite prime. Let $\pi(\alpha,\beta)$ be the standard induced representation of $\mathrm{GL}_{2}(\mathbb{Q}_{l})$ with characters $\alpha,\beta$ of $\mathbb{Q}_{l}^{\times}$. That is, the space $\pi(\alpha,\beta)$ consists of all locally constant functions $f:\mathrm{GL}_{2}(\mathbb{Q}_{l})\rightarrow\mathbb{C}$ with the transformation property
$$f\left(\left(
    \begin{array}{cc}
      a&b\\
      0&d \\
    \end{array}
  \right)g\right)=\alpha(a)\beta(d)\left|ad^{-1}\right|_{l}^{\frac{1}{2}}f(g)$$
for all $g\in\mathrm{GL}_{2}(\mathbb{Q}_{l}),a,d\in\mathbb{Q}_{l}^{\times},b\in\mathbb{Q}_{l}$. The action of $\mathrm{GL}_{2}(\mathbb{Q}_{l})$ on $\pi(\alpha,\beta)$ is given by right translation. It is well-known that if $\alpha\beta^{-1}\neq\vert\cdot\vert_{l}^{\pm}$, the induced representation is irreducible ($cf$. \cite[Theorem 4.5.1]{Bum97}), and it is called the principal series. On the other hand, the induced representation $\pi(\chi\vert\cdot\vert_{l}^{\frac{1}{2}},\chi\vert\cdot\vert_{l}^{-\frac{1}{2}})$ with a character $\chi$ of $\mathbb{Q}_{l}^{\times}$ is not irreducible and there exists a unique irreducible subrepresentation $\chi\mathrm{St}$ ($cf$. \cite[Theorem 4.5.1]{Bum97}). It is called the special representation or the Steinberg representation. Further, if an infinite dimensional irreducible representation $\pi_{l}$ of $\mathrm{GL}_{2}(\mathbb{Q}_{l})$ is neither a principal series nor a special representation, we call $\pi_{l}$ the supercuspidal representation.

If $v=\infty$, an irreducible representation $\pi_{\infty}$ of $\mathrm{GL}_{2}(\mathbb{R})$ is called the discrete series (resp. the limit of discrete series) when the lowest weight of $\pi_{\infty}$ by the action of $\mathrm{SO}_{2}(\mathbb{R})$ is $k\geq 2$ (resp. $k=1$). In particular, we denote by $\mathcal{D}_{0}(k)$ the (limit of) discrete series of lowest weight $k$ with the central character $\mathrm{sgn}^{k}$.

Let $l$ be a finite prime. Next, we recall the definition of new lines of an irreducible admissible infinite dimensional representations of $\mathrm{GL}_{2}(\mathbb{Q}_{l})$. For each $n\in\mathbb{Z}_{\geq 0}$, let $\mathcal{U}_{1}(l^{n})$ be a subgroup of $\mathrm{GL}_{2}(\mathbb{Z}_{l})$ defined by
$$\mathcal{U}_{l}(l^{n}):=\mathrm{GL}_{2}(\mathbb{Z}_{l})\cap\left(
    \begin{array}{cc}
      \mathbb{Z}_{l}&\mathbb{Z}_{l} \\
      l^{n}\mathbb{Z}_{l}&1+l^{n}\mathbb{Z}_{l} \\
    \end{array}
  \right).$$
Let $\pi_{l}$ be an irreducible admissible infinite dimensional representation of $\mathrm{GL}_{2}(\mathbb{Q}_{l})$ with a central character $\omega_{\pi_{l}}$. We denote by $\mathcal{V}_{\pi_{l}}$ the space of the representation $\pi_{l}$. Because $\pi_{l}$ is admissible, we can define 
$$c(\pi_{l}):=\mathrm{min}\{n\in\mathbb{Z}_{\geq0}\mid {\mathcal{V}_{\pi_{l}}}^{\mathcal{U}_{1}(l^{n})}\neq\{0\} \}.\leqno(2.3.1)$$
\begin{Definition}
  We call the space
$$\mathcal{V}_{\pi_{l}}^{\mathrm{new}}=\left\{\xi\in \mathcal{V}_{\pi_{l}}\middle|
 \pi_{l}\left(g\right)\xi=\xi,\ {}^\forall g\in \mathcal{U}_{1}(l^{c(\pi_{l})})\right\}$$
the new line of $\pi_{l}$.
\end{Definition}
The following proposition is proved in \cite[Theorem 1]{Cas73}.
\begin{Proposition}
We have
$$\mathrm{dim}_{\mathbb{C}}{\mathcal{V}_{\pi_{l}}}^{\mathrm{new}}=1.$$
\end{Proposition}
Next, we introduce ordinary lines defined in \cite[\S2.5]{Hsi17}. We denote by $N(\mathbb{Z}_{l})$ the nilpotent radical of the subgroup of upper triangle elements of $\mathrm{GL}_{2}(\mathbb{Z}_{l})$. Let $U_{l}\in\mathrm{End}_{\mathbb{C}}(V_{\pi_{l}}^{N(\mathbb{Z}_{l})})$ be a local operator defined by
$$U_{l}\xi:=\displaystyle{\sum_{i\in\mathbb{Z}\slash l\mathbb{Z}}}\pi\left(
    \begin{array}{cc}
      l&i \\
      0&1 \\
    \end{array}
  \right)\xi,$$
where $\xi\in V_{\pi_{l}}^{N(\mathbb{Z}_{l})}$.
\begin{Definition}
 Let $\alpha :\mathbb{Q}_{l}^{\times}\rightarrow \mathbb{C}^{\times}$ be a character. We call the space
$$V_{\pi_{l}}^{\mathrm{ord}}(\alpha):=\left\{\xi\in V_{\pi_{l}}^{N(\mathbb{Z}_{l})}\middle| U_{l}\xi=\alpha\vert\cdot\vert_{l}^{-\frac{1}{2}}(l)\xi, \pi_{l}\left(\mathrm{diag}\{t,1\}\right)\xi=\alpha(t)\xi, {}^{\forall}t\in\mathbb{Z}_{l}^{\times}\right\}$$
the ordinary line of $\pi_{l}$ attached to $\alpha$.
\end{Definition}
Ming-Lun Hsieh proved the following proposition in \cite[Proposition 2.2]{Hsi17}.
\begin{Proposition}
We have $\mathrm{dim}_{\mathbb{C}}V_{\pi_{l}}^{\mathrm{ord}}(\alpha)\leq 1$. Further, the space $V_{\pi_{l}}^{\mathrm{ord}}(\alpha)$ is non-zero if and only if $\pi_{l}$ is isomorphic to the principal series $\pi(\alpha,\alpha^{-1}\omega_{\pi_{l}})$ or the special representation $\alpha\vert\cdot\vert^{-\frac{1}{2}}_{l}\mathrm{St}$.
\end{Proposition}

Next, we recall the definition of Whittaker models and introduce a non-degenerate $\mathrm{GL}_{2}(\mathbb{Q}_{v})$-equivalent pairing of Whittaker models for each place $v$. For any infinite dimensional irreducible representation $\pi_{v}$ of $\mathrm{GL}_{2}(\mathbb{Q}_{v})$ with any place $v$, it is known that there exists a unique realization $\mathcal{W}(\pi_{v})=\mathcal{W}(\pi_{v},\psi_{v})$ of $\pi_{v}$ which is called the Whittaker model of $\pi_{v}$. The space of $\mathcal{W}(\pi_{v})$ is a subspace of functions $W:\mathrm{GL}_{2}(\mathbb{Q}_{v})\rightarrow \mathbb{C}$ which satisfies the following conditions. 
\begin{enumerate}
\item We have $W\left(\left(
    \begin{array}{cc}
      1&x \\
      0&1 \\
    \end{array}
  \right)g\right)$=$\psi_{v}(x)W(g)$ for any $x\in\mathbb{Q}_{v}$ and $g\in\mathrm{GL}_{2}(\mathbb{Q}_{v})$.
\item If $v=l$ is a finite prime, $W$ is locally constant.
\item If $v=\infty$, $W$ is a $C^{\infty}$-function and $W\left(\mathrm{diag}\{y,1\}\right)$ is moderate growth as $\vert y\vert_{\infty}\mapsto\infty$. 
\end{enumerate}
The action of $\mathrm{GL}_{2}(\mathbb{Q}_{v})$ on the space of $\mathcal{W}(\pi_{v})$ is given by right translation.

Next, we recall the definition of the non-degenerate $\mathrm{GL}_{2}(\mathbb{Q}_{v})$-equivalent paring on local Whittaker models defined in \cite[(2.14)]{Hsi17}.
\begin{Definition}
Let $\tilde{\pi}_{v}$ be the contragredient representation of $\pi_{v}$. We define a non-degenerate $\mathrm{GL}_{2}(\mathbb{Q}_{v})$-equivalent paring $\langle\ ,\ \rangle_{v}:\mathcal{W}(\pi_{v})\times\mathcal{W}(\tilde{\pi}_{v})\rightarrow\mathbb{C}$ to be
$$\langle W_{v},W^{\prime}_{v}\rangle_{v}:=\int_{\mathbb{Q}_{v}^{\times}}W(\mathrm{diag}\{y,1\})W^{\prime}(\mathrm{diag}\{-y,1\})\mathrm{d}^{\times}_{v}y$$
for the $W_{v}\in\mathcal{W}(\pi_{v})$ and $W^{\prime}_{v}\in\mathcal{W}(\tilde{\pi}_{v})$. This integral $\langle W_{v},W^{\prime}_{v}\rangle$ converges absolutely.
\end{Definition}
Next, we introduce definitions of normalized local Whittaker newforms and normalized ordinary Whittaker functions which are defined in \cite[\S2.4]{Hsi17}.
\begin{Definition}
\begin{enumerate}
\item If $v=l$ is a finite prime, we call the unique element $W_{\pi_{v}}^{\mathrm{new}}\in\mathcal{W}(\pi_{v})^{\mathrm{new}}$ with $W_{\pi_{v}}^{\mathrm{new}}(1)=1$ the normalized local Whittaker newform.\\

\item In the case $v=\infty$ and $\pi=\mathcal{D}_{0}(k)$, we define the normalized local Whittaker newform $W_{\pi_{v}}^{\mathrm{new}}\in \mathcal{W}(\pi_{v})$ to be
$$W_{\pi_{v}}^{\mathrm{new}}\left(z\left(
    \begin{array}{cc}
      y&x \\
      0&1 \\
    \end{array}
  \right)\left(
    \begin{array}{cc}
      \mathrm{cos}\theta&\mathrm{sin}\theta \\
      -\mathrm{sin}\theta&\mathrm{cos}\theta \\
    \end{array}
  \right)\right)=\mathbf{I}_{\mathbb{R}_{+}}(y)\cdot y^{\frac{k}{2}}e^{-2\pi y}\mathrm{sgn}(z)^{k}\psi_{\mathbb{R}}(x)e^{\sqrt{-1}k\theta}$$
where $y,z\in\mathbb{R}^{\times},\ x,\theta\in \mathbb{R}$ and $\mathbf{I}_{\mathbb{R}_{+}}$ is the characteristic function of the set $\mathbb{R}_{+}$.
\end{enumerate}
\end{Definition}
Next, we introduce normalized ordinary Whittaker functions.
\begin{Definition}
Let $l$ be a finite prime and $\alpha:\mathbb{Q}_{l}^{\times}\rightarrow\mathbb{C}^{\times}$ be a character. We assume $V_{\pi_{l}}^{\mathrm{ord}}(\alpha)\neq\{0\}$. We call the unique element $W^{\mathrm{ord}}_{\pi_{l}}\in\mathcal{W}(\pi_{l})^{\mathrm{ord}}(\alpha)$ with $W^{\mathrm{ord}}_{\pi_{l}}(1)=1$ the normalized ordinary Whittaker function.
\end{Definition}
In \cite[Corollary 2.3]{Hsi17}, the normalized ordinary Whittaker functions are characterized as follows.
\begin{Corollary}
Let $\alpha:\mathbb{Q}_{l}^{\times}\rightarrow\mathbb{C}^{\times}$ be a character. We assume $V_{\pi_{l}}^{\mathrm{ord}}(\alpha)\neq\{0\}$. Then, the normalized ordinary Whittaker function $W^{\mathrm{ord}}$ is characterized by
$$W^{\mathrm{ord}}(\left(
    \begin{array}{cc}
      y&0 \\
      0&1 \\
    \end{array}
  \right))=\alpha\vert\cdot\vert_{l}^{\frac{1}{2}}(y)\mathbf{I}_{\mathbb{Z}_{l}}(y),$$
where $y\in\mathbb{Q}_{l}^{\times}$ and $\mathbf{I}_{\mathbb{Z}_{l}}$ is the characteristic function on $\mathbb{Z}_{l}$.
\end{Corollary}
\subsection{Preliminaries on global irreducible representations on $\mathrm{GL}_{2}(\mathbb{A})$.}
In this subsection, we recall some of notation on representations of $\mathrm{GL}_{2}(\mathbb{A})$. In particular, we introduce global Whittaker functions attached to automorphic representations of $\mathrm{GL}_{2}(\mathbb{A})$.

Let $f\in \mathcal{S}_{k}(M,\chi)$ be a normalized Hecke eigenform. We denote by $\pi_{f}$ the subrepresentation of $\mathcal{A}^{0}(M,\chi_{\mathbb{A}}^{-1})$ generated by the adelic lift $\Phi(f)$. It is known that this representation $\pi_{f}$ is irreducible ($cf$. \cite[Theorem 3.6.1]{Bum97}) and $\pi_{f}$ decomposes $\pi_{f}:=\otimes^{\prime}_{v}\pi_{f,v}$ locally ($cf$. \cite[Theorem 3.3.3 and p.320]{Bum97}), where the symbol $\otimes_{v}^{\prime}$ represents the restricted tensor product. For each $\varphi\in \pi_{f}$, we define a global Whittaker function $W_{\varphi}: \mathrm{GL}_{2}(\mathbb{A})\rightarrow\mathbb{C}$ attached to $\varphi$ with $\psi_{\mathbb{A}}$ to be
$$W_{\varphi}(g)=\int_{\mathbb{A}\slash\mathbb{Q}}\varphi(\left(
    \begin{array}{cc}
      1&x \\
      0&1 \\
    \end{array}
  \right)g)\psi_{\mathbb{A}}(-x)\mathrm{d}_{\mathbb{A}}x\ \ (g\in\mathrm{GL}_{2}(\mathbb{A})).\leqno(2.4.1)$$

We denote by $\mathcal{W}(\pi_{f})$ the representation generated by $W_{\varphi}$ for every $\varphi\in\pi_{f}$. The action of $\mathrm{GL}_{2}(\mathbb{A})$ on $\mathcal{W}(\pi_{f})$ is given by right translation. The representation $\mathcal{W}(\pi_{f})$ is called the Whittaker model of $\pi_{f}$. We have the following proposition ($cf$. \cite[Theorem 3.5.2]{Bum97}).
\begin{Proposition}
For each pure tensor element $\varphi=\otimes^{\prime}\varphi_{v}\in \pi_{f}$, we have the local decomposition
$$W_{\varphi}(g)=\prod_{v}W_{\varphi_{v}}(g_{v})$$
for $g=(g_{v})_{v}\in\mathrm{GL}_{2}(\mathbb{A})$, where $W_{\varphi_{v}}(g_{v})\in\mathcal{W}(\pi_{v})$.
\end{Proposition}
\begin{Remark}
We fix an element $\alpha\in\mathbb{Q}_{\geq 0}^{\times}\cup{\infty}$. We assume that $f$ is a slope $\alpha$ $p$-stabilized newform. Further, if $\alpha=\infty$, we assume that $f$ is primitive. Then, the Whittaker function $W_{\Phi(f)}$ of the adelic lift $\Phi(f)$ decomposes as follows:
$$W_{\Phi(f)}(g)=\begin{cases}
W^{\mathrm{ord}}_{\pi_{f,p}}(g_{p})\displaystyle{\prod_{v\neq p}}W_{\pi_{f,v}}^{\mathrm{new}}(g_{v})\ &\mathrm{if}\ \alpha\neq\infty,\\
\displaystyle{\prod_{v}}W_{\pi_{f,v}}^{\mathrm{new}}(g_{v})\ &\mathrm{otherwise}\ ,
\end{cases}$$
where $W_{\pi_{f,v}}^{\mathrm{new}}\in\mathcal{W}(\pi_{f,v})^{\mathrm{new}}$ is the normalized local Whittaker newform in Definition 2.3.6.and $W^{\mathrm{ord}}_{\pi_{f,p}}\in\mathcal{W}(\pi_{f,p})^{\mathrm{ord}}(\alpha_{f,p})$ is the ordinary element in Definition 2.3.7 with an unramified character $\alpha_{f,p}$ such that $\alpha_{f,p}(p)=a(p,f)p^{\frac{1-k}{2}}$.
\end{Remark}
Let $\langle\ ,\rangle$ be the $\mathrm{GL}_{2}(\mathbb{A})$-equivalent paring of $\mathcal{A}^{0}(M,\chi_{\mathbb{A}}^{-1})\times A^{0}(M,\chi_{\mathbb{A}})$ defined in Definition 2.2.3. It is proved in \cite[Proposition 6]{Wal85} that for each $\varphi\in\pi_{f}$ and $\varphi^{\prime}\in\tilde{\pi}_{f}$, the value $\langle\varphi,\varphi^{\prime}\rangle$ is factorized locally.
\begin{Proposition}
Let $W_{\varphi},W_{\varphi^{\prime}}$ be Whittaker functions of $\varphi\in\pi_{f},\varphi^{\prime}\in\tilde{\pi}_{f}$. We assume that $W_{\varphi}=\prod_{v}W_{v}$ and $W_{\varphi^{\prime}}=\prod_{v}W^{\prime}_{v}$ satisfy $W_{v}(1)=1$ and $W^{\prime}_{v}(1)=1$ for all but finitely many $v$. Then, we have
$$\langle\varphi,\varphi^{\prime}\rangle=\frac{2L(1,\pi_{f},\mathrm{Ad})}{\zeta_{\mathbb{Q}}(2)}\displaystyle{\prod_{v}}\frac{\zeta_{v}(2)}{\zeta_{v}(1)L(1,\pi_{f,v},\mathrm{Ad})}\langle W_{v},W^{\prime}_{v}\rangle_{v},$$
where $\langle\ ,\ \rangle_{v}$ is the non-degenerated pairing defined in Definition 2.3.5. Further, we have
$$\frac{\zeta_{v}(2)}{\zeta_{v}(1)L(1,\pi_{f,v},\mathrm{Ad})}\langle W_{v},W^{\prime}_{v}\rangle_{v}=1$$
for all but finitely many $v$.
\end{Proposition}

We assume that $f$ is primitive. Then, by Proposition 2.4.3, we can decompose the value $\|f\|^{2}_{\Gamma_{0}(M)}$ locally. If we put $f_{c}(z)=\overline{f(-\overline{z})}$, it is known that $f_{c}\in \mathcal{S}_{k}(M,\overline{\chi})$ ($cf$. \cite[Lemma 4.3.2]{Miy06}) and we have 
$$f\vert_{k}\left(
    \begin{array}{cc}
      0&-1\\
      M&0 \\
    \end{array}
  \right)=w(f)\cdot f_{c},\ w(f)\in\mathbb{C}^{\times}$$
($cf$. \cite[Theorem 4.6.15]{Miy06}). We call $w(f)$ the root number of $f$. Then, if we set $\tau_{M}=(\tau_{M,v})_{v}\in\mathrm{GL}_{2}(\mathbb{A})$ with
$$\tau_{M,v}:=\begin{cases}
1\ &\mathrm{if}\ v=l\nmid M,\\
\footnotesize\left(
    \begin{array}{cc}
      0&-1\\
      -l^{v_{l}(M)}&0 \\
    \end{array}
  \right)\normalsize\ &\mathrm{if}\ v=l\vert M,\\
\mathrm{diag}\{-1,1\} &\mathrm{if}\ v=\infty.
\end{cases}\leqno(2.4.2)$$
By Proposition 2.4.3, we have the following equation ($cf$. \cite[(2.18)]{Hsi17}).
$$\|f\|^{2}_{\Gamma_{0}(M)}=\frac{[\mathrm{SL}_{2}(\mathbb{Z}):\Gamma_{0}(M)]}{2^{k}w(f)}\cdot L(1,\pi_{f},\mathrm{Ad})\cdot \displaystyle{\prod_{l\vert M}}B_{\pi_{f,l}},\leqno(2.4.3)$$
where 
$$B_{\pi_{f,v}}=\frac{\zeta_{v}(2)}{\zeta_{v}(1)L(1,\pi_{f,v},\mathrm{Ad})}\langle \rho(\tau_{M,v})W_{\pi_{f,v}}^{\mathrm{new}},W_{\pi_{f,v}}^{\mathrm{new}}\otimes\chi_{\mathbb{A},v}\rangle_{v}.\leqno(2.4.4)$$
\subsection{$p$-adic families of modular forms.}
Let $K$ be a finite extension of $\mathbb{Q}_{p}$ and $\mathcal{O}_{K}$ the ring of integers of $K$. Let $\mathbf{I}_{1}$ be a normal finite flat extension of the Iwasawa algebra $\Lambda:=\mathcal{O}_{K}\jump{\Gamma}$ of the topological group $\Gamma:=1+p\mathbb{Z}_{p}$. In this subsection, we introduce the definitions of ordinary $\mathbf{I}_{1}$-adic cusp forms, primitive Hida families and congruence numbers attached to Hida families. To define congruence numbers, we introduce the Hecke algebra of ordinary $\mathbf{I}_{1}$-adic cusp forms. At the end of this subsection, we define general $p$-adic families of modular forms.

Let $N_{1}$ be a positive integer which is prime to $p$. Throughout this section, we assume $\mathbb{Q}_{p}(\chi)\subset K$ for any Dirichlet character $\chi$ modulo $N_{1}p$.
\begin{Definition}
We call a continuous $\mathcal{O}_{K}$-algebra homomorphism $Q: \mathbf{I}_{1} \rightarrow \overline{\mathbb{Q}}_{p}$ an arithmetic point of weight $k_{Q}\geq 2$ and a finite part $\epsilon_{Q} : \Gamma \rightarrow \overline{\mathbb{Q}}_{p}^{\times}$ if the restriction $Q\vert_{\Gamma}: \Gamma\rightarrow \overline{\mathbb{Q}}_{p}^\times$ is given by $Q(x)=x^{k_{Q}}\epsilon_{Q}(x)$ for every $x\in\Gamma$. Here, $\epsilon_{Q}:\Gamma\rightarrow \overline{\mathbb{Q}}_{p}^{\times}$ is a finite character.
\end{Definition}
Let $\mathfrak{X}_{\mathbf{I}_{1}}$ be the set of arithmetic points of $\mathbf{I}_{1}$. We denote by $e$ the ordinary projection defined in \cite[(4.3)]{Hid88c}. We recall the definition of ordinary $\mathbf{I}_{1}$-adic cusp forms defined in \cite{Wil88}.
\begin{Definition}
Let $\chi$ be a Dirichlet character modulo $N_{1}p$. We call a formal power series $\mathbf{f}\in\mathbf{I}_{1}\jump{q}$ an ordinary $\mathbf{I}_{1}$-adic cusp form of tame level $N_{1}$ and Nebentypus $\chi$ if the specialization $\mathbf{f}_{Q}:=\displaystyle{\sum_{n\geq 0}}Q(a(n,\mathbf{f}))q^{n}\in Q(\mathbf{I}_{1})\jump{q}$ of $\mathbf{f}$ is a $q$-expansion of an element of $e\mathcal{S}_{k_{Q}}(N_{1}p^{e_{Q}},\chi\omega_{p}^{-k_{Q}}\epsilon_{Q},Q(\mathbf{I}_{1}))$ with $e_{Q}\geq1$ for all but a finite number of $Q\in \mathfrak{X}_{\mathbf{I}_{1}}$.
\end{Definition}
Let $\mathbf{S}^{\mathrm{ord}}(N_{1},\chi,\mathbf{I}_{1})$ be the $\mathbf{I}_{1}$-module consisting of ordinary $\mathbf{I}_{1}$-adic cusp forms of tame level $N_{1}$ and Nebentypus $\chi$. Next, we recall the definition of the Hecke algebra of $\mathbf{S}^{\mathrm{ord}}(N_{1},\chi,\mathbf{I}_{1})$. Let $\langle\ \rangle_{\Lambda}:\mathbb{Z}_{p}^{\times}\rightarrow \Lambda^{\times}$ be a group homomorphism defined by $\langle z\rangle_{\Lambda}=[z\omega_{p}^{-1}(z)]$, where $[z\omega_{p}^{-1}(z)]$ is the group-like element of $z\omega_{p}(z)^{-1}\in\Gamma$ in $\Lambda^{\times}$. We denote by 
$$\langle\ \rangle_{\mathbf{I}_{1}}:\mathbb{Z}_{p}^{\times}\rightarrow \mathbf{I}_{1}^{\times}\leqno(2.5.1)$$
the composition of $\langle\ \rangle_{\Lambda}$ and the embedding of $\Lambda^{\times}$ into $\mathbf{I}_{1}^{\times}$. For each prime $l\nmid N_{1}p$, the Hecke operator $T_{l}\in \mathrm{End}_{\mathbf{I}_{1}}(\mathbf{S}^{\mathrm{ord}}(N_{1},\chi,\mathbf{I}_{1}))$ at $l$ is defined by
$$T_{l}(f)=\displaystyle{\sum_{n\geq1}}a(n,T_{l}(f))q^{n},$$
where 
$$a(n,T_{l}(f))=\displaystyle{\sum_{b\vert (n,l)}}\langle b\rangle_{\mathbf{I}_{1}}\chi(b)b^{-1}a(ln\slash b^{2},f).$$
For each prime $l\vert N_{1}p$, the Hecke operator $T_{l}\in \mathrm{End}_{\mathbf{I}_{1}}(\mathbf{S}^{\mathrm{ord}}(N_{1},\chi,\mathbf{I}_{1}))$ at $l$ is defined by 
$$T_{l}(f)=\displaystyle{\sum_{n\geq 1}}a(ln,f)q^{n}.$$
The Hecke algebra $\mathbf{T}^{\mathrm{ord}}(N_{1},\chi,\mathbf{I}_{1})$ is defined by the sub-algebra of $\mathrm{End}_{\mathbf{I}_{1}}(\mathbf{S}^{\mathrm{ord}}(N_{1},\chi,\mathbf{I}_{1}))$ generated by $T_{l}$ for all primes $l$.

Next, we recall the definition of primitive Hida families.
\begin{Definition}
We call an element $\mathbf{f}\in \mathbf{S}^{\mathrm{ord}}(N_{1},\chi,\mathbf{I}_{1})$ a primitive Hida family of tame level $N_{1}$ and Nebentypus $\chi$ if the specialization $\mathbf{f}_{Q}$ is a $q$-expansion of an ordinary $p$-stabilized cuspidal newform for all but a finite number of $Q\in \mathfrak{X}_{\mathbf{I}_{1}}$.
\end{Definition}
Next, we recall the definition of the congruence number. From now on throughout \S2.5, we assume that $F\in \mathbf{S}^{\mathrm{ord}}(N_{1},\chi,\mathbf{I}_{1})$ is a primitive Hida family and $F$ satisfies Hypothesis (2). Let $\lambda_{F}:\mathbf{T}^{\mathrm{ord}}(N_{1},\chi,\mathbf{I}_{1})\rightarrow \mathbf{I}_{1}$ be an $\mathbf{I}_{1}$-algebra homomorphism defined by $\lambda_{F}(T)=a(1,T(F))$. Let $\mathbf{m}_{F}$ be a unique maximal ideal of $\mathbf{T}^{\mathrm{ord}}(N_{1},\chi,\mathbf{I}_{1})$ which contains $\mathrm{Ker}\lambda_{F}$. Let $\mathbf{T}^{\mathrm{ord}}(N_{1},\chi,\mathbf{I}_{1})_{\mathbf{m}_{F}}$ be the localization of $\mathbf{T}^{\mathrm{ord}}(N_{1},\chi,\mathbf{I}_{1})$ by $\mathbf{m}_{F}$. Let $\lambda_{\mathbf{m}_{F}}:=\lambda_{F}:\mathbf{T}^{\mathrm{ord}}(N_{1},\chi,\mathbf{I}_{1})_{\mathbf{m}_{F}}\rightarrow \mathbf{I}_{1}$ be the restriction of $\lambda_{F}$ to $\mathbf{T}^{\mathrm{ord}}(N_{1},\chi,\mathbf{I}_{1})_{\mathbf{m}_{F}}$. By \cite[Corollary 3.7]{Hid88b}, there exists a finite dimensional $\mathrm{Frac}\mathbf{I}_{1}$-algebra $B$ and an isomorphism
$$\lambda:\mathbf{T}^{\mathrm{ord}}(N_{1},\chi,\mathbf{I}_{1})_{\mathbf{m}_{F}}\otimes_{\mathbf{I}_{1}}\mathrm{Frac}\mathbf{I}_{1}\cong\mathrm{Frac}\mathbf{I}_{1}\oplus B$$
such that $(\mathrm{pr}_{\mathrm{Frac}\mathbf{I}_{1}}\circ\lambda)\vert_{\mathbf{T}^{\mathrm{ord}}(N_{1},\chi,\mathbf{I}_{1})_{\mathbf{m}_{F}}}=\lambda_{\mathbf{m}_{F}}$, where $\mathrm{pr}_{\mathrm{Frac}\mathbf{I}_{1}}:\mathrm{Frac}\mathbf{I}_{1}\oplus  B\rightarrow\mathrm{Frac}\mathbf{I}_{1}$ is the projection to the first part.
\begin{Definition}
Let $\mathrm{pr}_{\mathrm{Frac}\mathbf{I}_{1}}$ (resp. $\mathrm{pr}_{B}$) be the projection from $\mathrm{Frac}\mathbf{I}_{1}\oplus B$ to $\mathrm{Frac}\mathbf{I}_{1}$ (resp. $B$). We put $h(\mathrm{Frac}\mathbf{I}_{1}):=\mathrm{pr}_{\mathrm{Frac}\mathbf{I}_{1}}\circ\lambda(\mathbf{T}^{\mathrm{ord}}(N_{1},\chi,\mathbf{I}_{1})_{\mathbf{m}_{F}})$ and $h(B):=\mathrm{pr}_{B}\circ\lambda(\mathbf{T}^{\mathrm{ord}}(N_{1},\chi,\mathbf{I}_{1})_{\mathbf{m}_{F}})$. We define the module of congruence for $F$ to be
$$C(F):=h(\mathrm{Frac}\mathbf{I}_{1})\oplus h(B)\slash \lambda(\mathbf{T}^{\mathrm{ord}}(N_{1},\chi,\mathbf{I}_{1})_{\mathbf{m}_{F}}).$$
\end{Definition}
Let 
$$1_{F}\in\mathbf{T}^{\mathrm{ord}}(N_{1},\chi,\mathbf{I}_{1})_{\mathbf{m}_{F}}\otimes_{\mathbf{I}}\mathrm{Frac}\mathbf{I}\leqno(2.5.2)$$
be the idempotent element corresponded to $(1,0)\in\mathrm{Frac}\mathbf{I}\oplus B$ by $\lambda$. Let $\mathrm{Ann}(C(F)):=\{a\in\mathbf{I}_{1}\mid aC(F)=\{0\}\}$ be the annihilator of $C(F)$. By \cite[Corollary 2, page 482]{Wil95}, $\mathbf{T}^{\mathrm{ord}}(N_{1},\chi,\mathbf{I}_{1})_{\mathbf{m}_{F}}$ is a Gorenstein ring. Hence, by \cite[Theorem 4.4]{Hid88d}, the annihilator $\mathrm{Ann}(C(F))$ is generated by an element.
\begin{Definition}
We call a generator $\eta_{F}$ of $\mathrm{Ann}(C(F))$ a congruence number of $F$.
\end{Definition}
To the calculation of the zeta integral in the next section, we introduce the primitive Hida family $\breve{F}\in\mathbf{S}^{\mathrm{ord}}(N_{1},\chi_{(p)}\overline{\chi}^{(p)},\mathbf{I}_{1})$ attached to the twisted Hida family $F\vert[\overline{\chi}]$. The Fourier coefficients of $\breve{F}$ are given by
$$a(l,\breve{F})=\begin{cases}
\overline{\chi}^{(p)}(l)a(l,F)\ &\mathrm{if}\ l\nmid N_{1},\\
a(l,F)^{-1}\chi_{(p)}\omega_{p}^{2}l^{-1}\langle l\rangle_{\mathbf{I}_{1}}\ &\mathrm{if}\ l\vert N_{1}.
\end{cases}\leqno(2.5.3)$$
By \cite[Theorem 4.6.16]{Miy06}, $\breve{F}$ is a primitive Hida family. It is proved in \cite[(3.2)]{Hsi17} that $\mathbf{T}^{\mathrm{ord}}(N_{1},\chi,\mathbf{I}_{1})_{\mathbf{m}_{F}}\cong\mathbf{T}^{\mathrm{ord}}(N_{1},\chi_{p}\overline{\chi}^{(p)},\mathbf{I}_{1})_{\mathbf{m}_{\breve{F}}}$ and hence 
$$\eta(F)=\eta(\breve{F}).\leqno(2.5.4)$$

Next, we introduce general $p$-adic families of modular forms. Let $\mathbf{I}$ be a normal finite flat extension of $\Lambda$. We fix a set of non-zero continuous $\mathcal{O}_{K}$-algebraic homomorphisms 
$$\mathfrak{X}:=\{Q_{m}:\mathbf{I}\rightarrow \overline{\mathbb{Q}}_{p}\}_{m\in\mathbb{Z}_{\geq 1}}.\leqno(2.5.5)$$
Then, we define the specialization of an element $G=\displaystyle{\sum_{n\geq0}}a(n,G)q^{n}\in\mathbf{I}\jump{q}$, at $Q_{m}\in\mathfrak{X}$ to be $G_{Q_{m}}:=\displaystyle{\sum_{n\geq0}}Q_{m}(a(n,G))\jump{q}\in Q_{m}(\mathbf{I})\jump{q}$. Let $M$ be a positive integer which is prime to $p$ and $\chi$ a Dirichlet character modulo $Mp$. 
\begin{Definition}
We call an element $G\in\mathbf{I}\jump{q}$ a primitive $p$-adic families of tame level $M$ and Nebentypus $\chi$ attached to $\mathfrak{X}$ if $G_{Q_{m}}$ is a Fourier expansion of a cuspidal Hecke eigenform of weight $k_{Q_{m}}$, level $Mp^{e_{Q_{m}}}$ and Nebentypus $\chi\omega_{p}^{-k_{Q_{m}}}\epsilon_{Q_{m}}$ which is primitive outside of $p$ for each $m\in\mathbb{Z}_{\geq 1}$. Here, $k_{Q_{m}}$ and $e_{Q_{m}}$ are positive integers and $\epsilon_{Q_{m}}$ is a finite character of $\mathrm{\Gamma}$.
\end{Definition}
\section{The formulas of zeta integrals.}
For $i=1,2,3$, let $\mathbf{I}_{i}$ be a normal finite flat extension of $\Lambda$ and $N_{i}$ be a positive integer which is prime to $p$. Let $F\in \mathbf{S}^{\mathrm{ord}}(N_{1},\psi_{1},\mathbf{I}_{1})$ be a primitive Hida family. For $i=2,3$, let $G^{(i)}\in \mathbf{I}_{i}\jump{q}$ be a primitive $p$-adic family of tame level $N_{i}$ and Nebentypus $\psi_{i}$ attached to $\mathfrak{X}^{(i)}:=\{Q_{m}^{(i)}:\mathbf{I}_{i}\rightarrow \overline{\mathbb{Q}}_{p}\}_{m\in\mathbb{Z}_{\geq 1}}$ defined in Definition 2.5.6. We fix a point $\underline{Q}=(Q_{1},Q^{(2)}_{m_{2}},Q^{(3)}_{m_{3}})\in\mathfrak{X}_{\mathbf{I}_{1}}\times\mathfrak{X}^{(2)}\times\mathfrak{X}^{(3)}$ and we denote by $(f_{1},g_{2},g_{3})$ the specialization of $(F,G^{(2)},G^{(3)})$ at $\underline{Q}$. 

In \S3, we summarize some of notation and results of \cite{Hsi17} on triple product $L$-functions and zeta integrals without proofs. We define a triple product $L$-function and a global trilinear period integral attached to the triple $(f_{1},g_{2},g_{3})$ in \S3.2 and decompose the global trilinear integral to the product of the triple product $L$-function and the local zeta integrals by \cite[Proposition 3.10]{Hsi17} in \S3.3. In \S3.4, we calculate the local zeta integrals.

We denote by $(k_{1},k_{2},k_{3})$, $(N_{1}p^{e_{1}},N_{2}p^{e_{2}},N_{3}p^{e_{3}})$ and $(\chi_{1},\chi_{2},\chi_{3})$ the triples of weight, levels and Nebentypus of $(f_{1},g_{2},g_{3})$ respectively. In \S3, we assume that the triple $(F,G^{(2)},G^{(3)})$ satisfies Hypotheses (1)`(7). Further, we assume that the conditions $k_{1}\equiv k_{2}+k_{3}$ mod 2 and $k_{1}\geq k_{2}+k_{3}$ hold. We write $(\epsilon_{1},\epsilon_{2},\epsilon_{3}):=(\epsilon_{Q_{1}},\epsilon_{Q^{(2)}_{m_{2}}},\epsilon_{Q^{(3)}_{m_{3}}})$ for simplicity.
\subsection{The adjustment of levels for triple of modular forms.}
We set $N=\mathrm{lcm}(N_{1},N_{2},N_{3})$ and $N^{*}=\mathrm{gcd}(N_{1},N_{2},N_{3})$. In this subsection, we construct a triple $(f_{1}^{*},g_{2}^{*},g_{3}^{*})$ of cusp forms of level $N$ attached to the triple $(f_{1},g_{2},g_{3})$.

Let $(\pi_{1},\pi_{2},\pi_{3}$) be the triple of automorphic representations of $\mathrm{GL}_{2}(\mathbb{A})$ attached to the triple ($f_{1},g_{2},g_{3}$). We divide the set of primes $q$ except $p$ into the following two subsets.
\begin{align*}
\Sigma_{i}^{1}:&=\{l:\mathrm{prime}\mid q\neq p,\ \pi_{i,q}\ \mathrm{is\ a\ principal\ series}\},\\
\Sigma_{i}^{0}:&=\{l:\mathrm{prime}\mid q\neq p,\ \pi_{i,q}\ \mathrm{is\ not\ a\ principal\ series}\},
\end{align*}
for $i=1,2,3$.
We put $\Sigma_{1,2,3}:=\Sigma_{1}^{0}\cap\Sigma_{2}^{0}\cap\Sigma_{3}^{0}$. We put $\{i,j,t\}=\{1,2,3\}$. We introduce some sets of primes.
\begin{align*}
\Sigma_{i,j}^{(\mathrm{I})}:&=\{l\in\Sigma_{i}^{1}\cup\Sigma_{j}^{1}\cup\Sigma_{1,2,3}\mid v_{l}(N_{t})<\mathrm{min}\{v_{l}(N_{i}),v_{l}(N_{j})\}\},\\
\Sigma_{i}^{(\mathrm{I\hspace{-.1em}I}\mathrm{a})}:&=\{l\in\Sigma_{j}^{0}\cap\Sigma_{t}^{0}\mid L(s,\pi_{j,l}\otimes\pi_{t,l})\neq1,\ v_{l}(N_{i})=0\},\\
\Sigma_{i}^{(\mathrm{I\hspace{-.1em}I}\mathrm{b})}:&=\{l\in\Sigma_{j}^{0}\cap\Sigma_{t}^{0}\mid L(s,\pi_{j,l}\otimes\pi_{t,l})=1,\ l\in\Sigma_{i}^{1},\ v_{l}(N_{i})<\mathrm{min}\{v_{l}(N_{j}),v_{l}(N_{t})\}\},\\
\Sigma_{i}^{\mathrm{max}}:&=\{l\ \mathrm{is\ a\ prime\ factor\ of}\ N_{i}\mid v_{l}(N_{j})=v_{l}(N_{t})=v_{l}(N^{*})<v_{l}(N_{i})\}.
\end{align*}
Here, $L(s,\pi_{j,l}\otimes\pi_{t,l})$ is the $L$-function defined in \cite{Jac72}. Further, we put
\begin{align*}
d_{i}^{(\mathrm{I})}:&=\prod_{l\in\Sigma_{i,j}^{(\mathrm{I})}}l^{\mathrm{max}\{v_{l}(N_{i}),v_{l}(N_{j})\}-v_{l}(N_{i})}\cdot\prod_{l\in\Sigma_{i,t}^{(\mathrm{I})}}l^{\mathrm{max}\{v_{l}(N_{i}),v_{l}(N_{t})\}-v_{l}(N_{i})},\\
d_{i}^{(\mathrm{I\hspace{-.1em}I})}:&=\prod_{l\in\Sigma_{i}^{(\mathrm{I\hspace{-.1em}I}\mathrm{a})}}l^{\lceil\frac{\mathrm{max}\{v_{l}(N_{j}),v_{l}(N_{l})\}}{2}\rceil}\cdot\prod_{l\in\Sigma_{i}^{(\mathrm{I\hspace{-.1em}I}\mathrm{b})}}l^{\mathrm{max}\{v_{l}(N_{j}),v_{l}(N_{t})\}-v_{l}(N_{i})},\\
d_{i}^{\mathrm{max}}:&=\prod_{l\in\Sigma_{i}^{\mathrm{max}}}l^{v_{l}(N_{i})-v_{l}(N^{*})}.
\end{align*}  
We define a triple of integers $(d_{1},d_{2},d_{3})$ attached to the triple $(f_{1},g_{2},g_{3})$ to be
$$d_{1}=d_{1}^\mathrm{I}d_{1}^{\mathrm{I\hspace{-.1em}I}},\ d_{2}=d_{2}^{(\mathrm{I})}d_{1}^{\mathrm{max}}d_{3}^{\mathrm{max}}d_{2}^{(\mathrm{I\hspace{-.1em}I})},\ d_{3}=d_{3}^{(\mathrm{I})}d_{2}^{\mathrm{max}}d_{3}^{(\mathrm{I\hspace{-.1em}I})}.\leqno(3.1.1)$$
We set 
$$\Sigma_{i,0}^{(\mathrm{I\hspace{-.1em}I}\mathrm{b})}:=\{l\in\Sigma_{i}^{(\mathrm{I\hspace{-.1em}I}\mathrm{b})}\mid v_{l}(N_{i})=0\}.\leqno(3.1.2)$$
If necessary, we exchange $\mathbf{I}_{i}$ to a normal finite flat extension for $i=1,2,3$. Then, for each prime $l\vert N$, we fix a root $\beta_{l}(F)\in\mathbf{I}_{1}^{\times}$ (resp. $\beta_{l}(G^{(i)})\in\mathbf{I}_{i}^{\times}$ for $i=2,3$) of the Hecke polynomial $X^{2}-a(p,F)X+\psi_{1}(l)l^{-1}\langle l\rangle_{\mathbf{I}_{1}}$ (resp. $X^{2}-a(p,G^{(i)})X+\psi_{i}(l)l^{-1}\langle l\rangle^{(i)}$). We denote by $(\beta_{l,1},\beta_{l,2},\beta_{l,3})$ the specialization of $(\beta_{l}(F),\beta_{l}(G^{(2)}),\beta_{l}(G^{(3)}))$ at $\underline{Q}$. For the triple $(f_{1},g_{2},g_{3})$ with different levels, we define the adjustment $(f_{1}^{*},g_{2}^{*},g_{3}^{*})$ to be
\begin{align*}
f_{1}^{*}&=\displaystyle{\sum_{I\subset\Sigma_{1,0}^{(\mathrm{I\hspace{-.1em}I}\mathrm{b})}}}(-1)^{\vert I\vert}\beta_{I}(f_{1})^{-1}V_{d_{f_{1}\slash n_{I}}}f_{1},\tag{3.1.3}\\
g_{i}^{*}&=\displaystyle{\sum_{I\subset\Sigma_{i,0}^{(\mathrm{I\hspace{-.1em}I}\mathrm{b})}}}(-1)^{\vert I\vert}\beta_{I}(g_{i})^{-1}V_{d_{g_{i}\slash n_{I}}}g_{i},\ (i=2,3),\tag{3.1.4}
\end{align*}
where $n_{I}=\prod_{l\in I}l$, $\beta_{l}(f_{1})=\prod_{l\in I}\beta_{l,1}$ and $\beta_{l}(g_{i})=\prod_{l\in I}\beta_{l,i}$. By a simple calculation, for each finite prime $l$, if we define a polynomial to be
$$Q_{i,l}(X)=X^{v_{l}(d_{i})}\begin{cases}
1\ &\mathrm{if}\ l\not\in\Sigma_{i,0}^{(\mathrm{I\hspace{-.1em}I}\mathrm{b})},\\
(1-\beta_{l,i})^{-1}l^{\frac{k_{i}}{2}-1}X^{-1}\ &\mathrm{if}\ l\in\Sigma_{i,0}^{(\mathrm{I\hspace{-.1em}I}\mathrm{b})},
\end{cases}\leqno(3.1.5)$$
we have
\begin{align*}
\varphi_{f_{1}}^{*}:=\prod_{l}Q_{1,l}(V_{l})\Phi(f_{1})=d_{1}^{\frac{k_{1}}{2}-1}\Phi(f_{1}^{*}),\tag{3.1.6}\\
\varphi_{g_{i}}^{*}:=\prod_{l}Q_{i,l}(V_{l})\Phi(g_{i})=d_{i}^{\frac{k_{i}}{2}-1}\Phi(g_{i}^{*}),\ (i=2,3).\tag{3.1.7}
\end{align*}
\subsection{Triple product $L$-functions and Global trilinear period integrals.}
In this subsection, we introduce the triple product $L$-function attached to the triple $(f_{1},g_{2},g_{3})$. Further, we introduce the modified $p$-Euler factor $\mathcal{E}_{p}(f_{1},\mathrm{Ad})$ and the global trilinear period integral $I(\rho(\mathbf{t}_{n})\phi_{\gamma})$ defined in \cite[\S3.7]{Hsi17}. We denote by $(\chi_{\underline{Q}})_{\mathbb{A}}$ the adelization of the Dirichlet character 
$$\chi_{\underline{Q}}:=\omega_{p}^{\frac{1}{2}(2a-k_{1}-k_{2}-k_{3})}(\epsilon_{1}\epsilon_{2}\epsilon_{3})^{\frac{1}{2}}.\leqno(3.2.1)$$
We define the triple product $L$-function $L(s,\Pi_{\underline{Q}})$ attached to the representation 
$$\Pi_{\underline{Q}}=\pi_{1}\otimes(\chi_{\underline{Q}})_{\mathbb{A}}\times\pi_{2}\times\pi_{3}$$
to be
$$L(s,\Pi_{\underline{Q}})=\prod_{v: \mathrm{place}}L(s,\Pi_{\underline{Q},v}),\ \mathrm{Re}(s)>\frac{5}{2},$$
where $L(s,\Pi_{\underline{Q},v})$ is the GCD local triple product $L$-function defined in \cite{PSR87} and \cite{Ike92}. In particular, for each finite prime $l$, the local $L$-function $L(s,\Pi_{\underline{Q},l})$ at $l$ can be written by the form $1\slash P(p^{-s})$, where $P(T)\in\mathbb{C}[T]$ such that $P(0)=1$. By the result of \cite{Ike98}, the archimedean factor $L(s,\Pi_{\underline{Q},\infty})$ can be written by the form
$$L(s,\Pi_{\underline{Q},\infty}):=\Gamma_{\mathbb{C}}(s+\frac{w}{2})\prod_{i=1}^{3}\Gamma_{\mathbb{C}}(s+1-k_{i}^{*}),$$
where $w=k_{1}+k_{2}+k_{3}-2$, $k_{i}^{*}=\frac{k_{1}+k_{2}+k{3}}{2} -k_{i}$ and $\Gamma_{\mathbb{C}}(s)=2(2\pi)^{-s}\Gamma(s)$. By  \cite[Proposition 2.5]{Ike92}, the function $L(s,\Pi_{\underline{Q}})$ is continued to the entire $\mathbb{C}$-plane analytically and by \cite[Proposition 2.4]{Ike92}, the function $L(s,\Pi_{\underline{Q}})$ satisfies the functional equation
$$L(s,\Pi_{\underline{Q}})=\epsilon(s,\Pi_{\underline{Q}})L(1-s,\Pi_{\underline{Q}}),$$
where $\epsilon(s,\Pi_{\underline{Q}})$ is the global epsilon factor defined in \cite{Ike92}. The epsilon factor $\epsilon(s,\Pi_{\underline{Q}})$ can be decomposed by the product of the local epsilon factors 
$$\epsilon(s,\Pi_{\underline{Q}})=\prod_{v:\mathrm{place}}\epsilon(s,\Pi_{\underline{Q},v})$$
and it is known that $\epsilon(\frac{1}{2},\Pi_{\underline{Q},v})\in\{\pm1\}$.

Next, we introduce the modified $p$-Euler factor. Let $\alpha_{f_{1},p}$ be the unramified character defined in Remark 2.4.2. Let $\pi_{1}=\otimes^{\prime}_{v}\pi_{1,v}$ be the automorphic representation of $\mathrm{GL}_{2}(\mathbb{A})$ attached to $f_{1}$ with the central character $\omega_{1}$. We put $\beta_{f_{1},p}:=\alpha_{f_{1},p}^{-1}\omega_{1,p}$.
\begin{Definition}
Let $c(\pi_{1,p})$ be the positive integer defined in (2.3.1). We define a modified $p$-Euler factor $\mathcal{E}_{p}(f_{1},\mathrm{Ad})$ for the adjoint motive of $f_{1}$ to be
\begin{align*}
\mathcal{E}_{p}(f_{1},\mathrm{Ad})&=\epsilon(1,\beta_{f_{1},p}\alpha_{f_{1},p}^{-1})L(0,\beta_{f_{1},p}\alpha_{f_{1},p}^{-1})^{-1}L(1,\beta_{f_{1},p}\alpha_{f_{1},p}^{-1})^{-1}\\
&=a(p,f_{1})^{-c(\pi_{1,p})}\cdot p^{c(\pi_{1,p})(\frac{k_{1}}{2}-1)}\epsilon(1\slash 2,\pi_{1,p})\times\\
&\begin{cases}
(1-\alpha_{f_{1},p}^{-2}\omega_{1,p}(p))(1-\alpha_{f_{1},p}^{-2}\omega_{1,p}(p)p^{-1})\ &\mathrm{if}\ c(\pi_{1,p})=0,\\
1&\mathrm{if}\ c(\pi_{1,p})>0.
\end{cases}
\end{align*}
\end{Definition}
Here, $L(s,\beta_{f_{1},p}\alpha_{f_{1},p}^{-1})$ (resp. $\epsilon(s,\beta_{f_{1},p}\alpha_{f_{1},p}^{-1}):=\epsilon(s,\beta_{f_{1},p}\alpha_{f_{1},p}^{-1},\psi_{v})$) is the complex $L$-function (resp. $\epsilon$-factor) of $\beta_{f_{1},p}\alpha_{f_{1},p}^{-1}$ ($cf$. \cite[\S3.1]{Bum97}).

Next, we introduce the global trilinear period which is defined in \cite[\S 3.7]{Hsi17}. By Hypothesis (1), we have
$$\prod_{i=1}^{3}\chi_{i}=\chi_{\underline{Q}}^{2}.$$
Then, the product of the central characters of the triple $(\pi_{1}^{\prime},\pi_{2}^{\prime},\pi_{3}^{\prime}):=(\pi_{1}\otimes(\chi_{\underline{Q}})_{\mathbb{A}},\pi_{2},\pi_{3})$ is trivial. Then, we can define a trilinear integral map $I:\Pi_{\underline{Q}}=\pi^{\prime}_{1}\times\pi^{\prime}_{2}\times\pi^{\prime}_{3}\rightarrow \mathbb{C}$ to be
$$I(\phi):=\int_{\mathbb{A}^{\times}\mathrm{GL}_{2}(\mathbb{Q})\backslash \mathrm{GL}_{2}(\mathbb{A})}\phi(x,x,x)d^{\tau}x.\leqno(3.2.2)$$
For each $n\in\mathbb{Z}_{\geq 0}$, we set
$$t_{n}:=\left(
    \begin{array}{cc}
      0&p^{-n}\\
      p^{n}&0 \\
    \end{array}
  \right)\in\mathrm{GL}_{2}(\mathbb{Q}_{p}).\leqno(3.2.3)$$
Further, we define an element $\mathbf{t}_{n}\in \mathrm{GL}_{2}(\mathbb{A})^{3}$ to be
$$\mathbf{t}_{n,v}:=\begin{cases}\left(t_{n},1,1\right)\ &\mathrm{if}\ v=p,\\
(1,1,1)\ &\mathrm{otherwise}
\end{cases}
\leqno(3.2.4)$$
for $n\geq0$. We set $\mathcal{J}_{\infty}\in \mathrm{GL}_{2}(\mathbb{A})$, where
$$\mathcal{J}_{\infty,v}:=\begin{cases}\left(
    \begin{array}{cc}
      -1&0\\
      0&1 \\
    \end{array}
  \right)\ &\mathrm{if}\ v=\infty,\\
\ \ \ \ \ \ 1\ &\mathrm{otherwise}.\end{cases}
\leqno(3.2.5)$$
We define an automorphic form $\phi_{\gamma}\in \Pi_{\underline{Q}}$ to be
$$\phi_{\gamma}:=(\rho(\mathcal{J}_{\infty})\varphi_{f_{1}}^{*}\otimes(\chi_{\underline{Q}})_{\mathbb{A}}\boxtimes\varphi_{g_{2}}^{*}\boxtimes\mathbf{V}_{+}^{r}\theta_{p}^{\kappa}\varphi_{g_{3}}^{*}\leqno(3.2.6)$$
with $r=\frac{1}{2}(k_{1}-k_{2}-k_{3})$ and $\kappa=\chi_{1,(p)}\chi_{\underline{Q}}^{-1}$. Here, $\mathbf{V}_{+}$ (resp. $\theta_{p}^{\kappa}$) is the raising operator (resp. twisting operator) defined in (2.2.1) (resp. (2.2.6)). We define a global trilinear period to be $I(\rho(\mathbf{t}_{n})\phi_{\gamma})$.

Next, we relate the special value of the triple product $p$-adic $L$-function to the integral $I(\rho(\mathbf{t}_{n})\phi_{\gamma})$ by \cite[Proposition 3.7]{Hsi17}.  Let $\mathcal{H}$ be the holomorphic projection defined in \cite[Lemma 7]{Shi76}. We can verify that $h:=e\mathcal{H}(f_{2}^{*}\delta_{k_{3}}^{r}f_{3}^{*}\vert[\kappa])$ is an element of $eS_{k_{1}}(Np^{e_{1}},\chi_{1,(p)}\overline{\chi}_{1}^{(p)},K)$. Let $\breve{f}_{1}$ be the ordinary $p$-stabilized newform attached to the twist $f\vert[\overline{\chi}_{1}^{(p)}]$. Hence, $\breve{f}_{1}$ is the specialization of $\breve{F}$ defined in (2.5.3) at $Q_{1}$. We set
$$d_{\gamma}=\displaystyle{\prod_{i=1,2,3}}d_{i}^{\frac{k_{i}}{2}-1},\leqno(3.2.7)$$
where $d_{i}$ is the integer defined in (3.1.1). Let $1_{\breve{F}_{1}}$ be the idempotent element attached to $\breve{F}_{1}$ defined in (2.5.2). Let $1_{\breve{f}_{1}}$ be the specialization of $1_{\breve{F}_{1}}$ at $Q_{1}$. Ming-Lun Hsieh proved the following proposition in \cite[Proposition 3.7]{Hsi17}.
\begin{Proposition}
For each $n\geq \mathrm{max}\{e_{1},e_{2},e_{3}\}$, we have
$$a(1,1_{\breve{f_{1}}}\mathrm{Tr}_{N\slash N_{1}}h)=\frac{\zeta_{\mathbb{Q}}(2)[\mathrm{SL}_{2}(\mathbb{Z}):\mathrm{\Gamma}_{0}(N)]}{\|f_{1}^{\mathrm{prim}}\|^{2}_{\Gamma_{0}(N_{1}p^{e_{1}^{\prime}})}\mathcal{E}_{p}(f_{1},\mathrm{Ad})}\cdot I(\rho(\mathbf{t}_{n})\phi_{\gamma})\cdot \frac{\zeta_{p}(1)}{\omega_{1,p}^{-1}\alpha_{f_{1},p}^{2}\vert\cdot\vert_{p}(p^{n})\zeta_{p}(2)}\cdot\frac{1}{d_{\gamma}},$$
where $\mathrm{Tr}_{N\slash N_{1}}: S_{k_{1}}^{\mathrm{ord}}(Np^{e_{1}},\chi_{1,(p)}\overline{\chi}_{1}^{(p)},K)\rightarrow S_{k_{1}}^{\mathrm{ord}}(N_{1}p^{e_{1}},\chi_{1,(p)}\overline{\chi}_{1}^{(p)},K)$ is the trace operator defined in \cite[(7.5)]{Hid88c} and $f_{1}^{\mathrm{prim}}$ is the primitive form of level $N_{1}p^{e_{1}^{\prime}}$ attached to $f_{1}$.
\end{Proposition}
\subsection{The relation of zeta integrals and triple product $L$-functions.}
In this subsection, by \cite[Proposition 3.10]{Hsi17}, we relate the central critical value of the triple product $L$-function to the global trilinear period integral $I(\rho(\mathbf{t}_{n})\phi_{\gamma}^{*})$. Let us keep the notation as \S3.2. For $i=1,2,3$, we fix an isomorphism $\iota_{i}:\pi_{i}^{\prime}\cong\otimes^{\prime}\pi_{i,v}^{\prime}$ where $\pi_{i,v}^{\prime}$ is an irreducible admissible representation of $\mathrm{GL}_{2}(\mathbb{Q}_{v})$. For each place $v$, we put $\Pi_{\underline{Q},v}:=\pi_{1,v}^{\prime}\boxtimes\pi_{2,v}^{\prime}\boxtimes\pi_{3,v}^{\prime}$. Let $\boxtimes_{i}^{3}\varphi_{i,v}\in \Pi_{\underline{Q},v}$ be an element defined by
$$\boxtimes_{i=1}^{3}\iota_{i}(\Phi(f_{1})\otimes(\chi_{\underline{Q}})_{\mathbb{A}}\boxtimes\Phi(g_{2})\boxtimes\Phi(g_{3}))=\otimes_{v}^{\prime}\boxtimes_{i=1}^{3}\varphi_{i,v}.\leqno(3.3.1)$$
Further, we define an element $\phi^{*}_{v}\in\Pi_{\underline{Q},v}$ to be
$$\phi^{*}_{v}=
\begin{cases}
\pi^{\prime}_{1,\infty}(\mathcal{J}_{\infty,\infty})\varphi_{1,\infty}\otimes\varphi_{2,\infty}\otimes\mathbf{V}_{+}^{r}\varphi_{3,\infty}\ &\mathrm{if}\ v=\infty,\\
\varphi_{1,p}\otimes\varphi_{2,p}\otimes\theta_{p}^{\kappa}\varphi_{3,p}\ &\mathrm{if}\ v=p,\\
Q_{1,v}(V_{v})\varphi_{1,v}\otimes Q_{2,v}(V_{v})\varphi_{2,v}\otimes Q_{3,v}(V_{v})\varphi_{3,v}\ &\mathrm{otherwise}.
\end{cases}\leqno(3.3.2)$$
Let $\omega_{i}^{\prime}$ be the central character of $\pi_{i}^{\prime}$ for $i=1,2,3$. We put $\tilde{\phi}_{v}:=\boxtimes_{i=1}^{3}\tilde{\varphi}_{v}\in \tilde{\Pi}_{\underline{Q},v}$ with $\tilde{\varphi}_{i,v}:=\varphi_{i,v}\otimes{\omega_{i,v}^{\prime}}^{-1}$ for each place $v$. Further we define an element $\tilde{\phi}^{*}_{v}\in \tilde{\Pi}_{\underline{Q},v}$ to be
$$\tilde{\phi}^{*}_{v}=
\begin{cases}
\tilde{\pi}^{\prime}_{1,\infty}(\mathcal{J}_{\infty,\infty})\tilde{\varphi}_{1,\infty}\otimes\tilde{\varphi}_{2,\infty}\otimes\mathbf{V}_{+}^{r}\tilde{\varphi}_{3,\infty}\ &\mathrm{if}\ v=\infty,\\
\tilde{\varphi}_{1,p}\otimes\tilde{\varphi}_{2,p}\otimes\theta_{p}^{\kappa}\tilde{\varphi}_{3,p}\ &\mathrm{if}\ v=p,\\
\tilde{Q}_{1,v}(V_{v})\tilde{\varphi}_{1,v}\otimes\tilde{Q}_{2,v}(V_{v})\tilde{\varphi}_{2,v}\otimes\tilde{Q}_{3,v}(V_{v})\tilde{\varphi}_{3,v}\ &\mathrm{otherwise},
\end{cases}\leqno(3.3.3)$$
where $\tilde{Q}_{i,l}(X)=Q_{i,l}({\omega^{\prime}_{i}}^{-1}(l))X)$.
  For each place $v$, we choose a $\mathrm{GL}_{2}(\mathbb{Q}_{v})^{3}$-equivalent pairing $\langle\ ,\ \rangle_{v}:\Pi_{\underline{Q},v}\times\tilde{\Pi}_{\underline{Q},v}\rightarrow\mathbb{C}$ such that $\langle\phi_{v}^{*},\tilde{\phi}_{v}^{*}\ \rangle_{v}=1$ for almost all places $v$. We introduce the local integrals defined in \cite[\S 3.8]{Hsi17}. We set $\tau_{\underline{N}}=(\tau_{N_{1}},\tau_{N_{2}},\tau_{N_{3}})\in\mathrm{GL}_{2}(\mathbb{A})^{3}$, where $\tau_{N_{i}}$ is defined in (2.4.2). We define a local integral for each place $v$ to be
$$I_{v}(\phi^{*}_{v}\otimes\tilde{\phi}^{*}_{v}):=\frac{L(1,\Pi_{\underline{Q},v},\mathrm{Ad})}{\zeta_{v}(2)^{2}L(\frac{1}{2},\Pi_{\underline{Q},v})}\int_{\mathrm{PGL}_{2}(\mathbb{Q}_{v})}\frac{\langle\Pi_{\underline{Q},v}(g_{v})\phi^{*}_{v},\tilde{\phi}^{*}_{v}\rangle_{v}}{\langle\Pi_{\underline{Q},v}(\tau_{\underline{N},v})\phi_{v},\tilde{\phi}_{v}\rangle_{v}}\mathrm{d}_{v}g.\leqno(3.3.4)$$
We define another local zeta integral at $p$ to be
$$I_{p}^{\mathrm{ord}}(\phi^{*}_{p}\otimes\tilde{\phi}^{*}_{p},\mathbf{t}_{n}):=\frac{L(1,\Pi_{\underline{Q},p},\mathrm{Ad})}{\zeta_{p}(2)^{2}L(\frac{1}{2},\Pi_{\underline{Q},p})}\int_{\mathrm{PGL}_{2}(\mathbb{Q}_{p})}\frac{\langle\Pi_{\underline{Q},p}(g_{p}\mathbf{t}_{n})\phi^{*}_{p},\tilde{\Pi}_{\underline{Q},p}(\mathbf{t}_{n})\tilde{\phi}^{*}_{p}\rangle_{p}}{\langle\Pi_{\underline{Q},p}(t_{n}\phi_{p}),\tilde{\phi}_{p}\rangle_{p}}\mathrm{d}_{p}g.\leqno(3.3.5)$$
Ming-Lun Hsieh proved the following proposition in \cite[Proposition 3.10]{Hsi17}.
\begin{Proposition}
We have
$$\frac{I(\rho(\mathbf{t}_{n})\phi_{\gamma})^{2}}{\displaystyle{\prod_{i=1}^{3}}\langle \rho(\tau_{\underline{N}}\mathbf{t}_{n})\phi_{i},\tilde{\phi}_{i}\rangle}=\frac{(-1)^{k_{1}}\zeta_{\mathbb{Q}}(2)}{8L(1,\Pi_{\underline{Q}},\mathrm{Ad})}\cdot L(\frac{1}{2},\Pi_{\underline{Q}})\cdot I_{p}^{\mathrm{ord}}(\phi^{*}_{p}\otimes\tilde{\phi}^{*}_{p},\mathbf{t}_{n})\cdot\displaystyle{\prod_{v\neq p}}I_{v}(\phi^{*}_{v}\otimes\tilde{\phi}^{*}_{v})(\chi_{\underline{Q}})_{\mathbb{A},v}(d_{1}).$$
\end{Proposition}
Next, we calculate the local integrals $I_{q}(\varphi_{q}\otimes\tilde{\varphi}_{q})$ for each finite primes $q\nmid N$ and $I_{\infty}(\varphi_{\infty}\otimes\tilde{\varphi}_{\infty})$. Ming-Lun Hsieh proved the following lemma in \cite[Lemma 3.11]{Hsi17}.
\begin{Lemma}
We have
\begin{description}
\item[(1)] $I_{q}(\phi^{*}_{q}\otimes\tilde{\phi}^{*}_{q})=1$, ${}^{\forall}q\nmid N$,
\item[(2)] $I_{\infty}(\phi^{*}_{\infty}\otimes\tilde{\phi}^{*}_{\infty})=2^{k_{2}+k_{3}-k_{1}+1}$.
\end{description}
\end{Lemma}
Next, we normalize the local zeta integrals $I_{p}^{\mathrm{ord}}(\phi^{*}_{p}\otimes\tilde{\phi}^{*}_{p},\mathbf{t}_{n})$ and $I_{q}(\phi^{*}_{q}\otimes\tilde{\phi}^{*}_{q})$ for each finite primes $q\vert N$. For $i=1,2,3$, let $B_{\pi^{\prime}_{i,q}}$ be as in defined in (2.4.4).
\begin{Definition}
We define normalized local zeta integrals to be
\begin{align*}
\mathfrak{I}^{\mathrm{unb}}_{\Pi_{\underline{Q},p}}&=\frac{I^{\mathrm{ord}}_{p}(\phi^{*}_{p}\otimes\tilde{\phi}^{*}_{p},\mathbf{t}_{n})}{L(1,\Pi_{\underline{Q},p},\mathrm{Ad})}\cdot\frac{\omega^{\prime}_{1,p}(-p^{2n})}{\alpha_{f_{1},p}^{2}\vert\cdot\vert_{p}(-p^{2n})}\cdot\frac{\zeta_{p}(2)}{\zeta_{p}(1)}\cdot\displaystyle{\prod_{i=1}^{3}}\langle \rho(t_{n})W^{?}_{\pi^{\prime}_{i,p}}.W^{?}_{\pi^{\prime}_{i,p}}\otimes{\omega^{\prime}_{i,p}}^{-1}\rangle_{p},\\
\mathfrak{I}_{\Pi_{\underline{Q},q}}&=I_{q}(\phi^{*}_{q}\otimes\tilde{\phi}^{*}_{q})\cdot\frac{\zeta_{q}(1)}{\vert N\vert_{q}^{2}\zeta_{q}(2)}\cdot(\chi_{\underline{Q}})^{2}_{\mathbb{A},q}(d_{1})\vert d_{\gamma}^{2}\vert_{q}\prod_{i=1}^{3}B_{\pi_{i,q}^{\prime}}\ \mathrm{for\ each}\ q\vert N,
\end{align*}
where $W^{?}_{\pi^{\prime}_{i,p}}$ is $W^{\mathrm{ord}}_{\pi^{\prime}_{i,p}}$ if $i=1$ or $g_{i}$ has a finite slope at $p$ and if not, $W^{?}_{\pi_{i,p}}$ is $W^{\mathrm{new}}_{\pi_{i,p}}$.
\end{Definition}
Next, we define a canonical period. Let $\eta_{f_{1}}$ be the specialization of the congruence number $\eta_{F}$ at $Q_{1}$.
\begin{Definition}
We define a canonical period $\Omega_{f_{1}}$ to be
$$\Omega_{f_{1}}:=(-2\sqrt{-1})^{k+1}\cdot \|f_{1}^{\mathrm{prim}}\|^{2}_{\mathrm{\Gamma}_{0}(m_{f_{1}})}\cdot\frac{\mathcal{E}(f_{1},\mathrm{Ad})}{\iota_{p}(\eta_{f_{1}})},$$
where $f_{1}^{\mathrm{prime}}$ is the primitive cusp form attached to $f_{1}$. 
\end{Definition}
By Proposition 3.2.2, Proposition 3.3.1, Lemma 3.3.2 and (2.5.4), we have the following corollary.
\begin{Corollary}
We put $h:=e\mathcal{H}(g_{2}\delta_{k_{3}}^{r}g_{3}\vert[\kappa])$. We have
$$(a(1,\eta_{f_{1}}1_{\breve{f_{1}}}\mathrm{Tr}_{N\slash N_{1}}h))^{2}=\frac{\psi_{1,(p)}(-1)(-1)^{k_{1}+1}L(1\slash2,\Pi_{\underline{Q}})}{\Omega_{f_{1}}^{2}}\cdot\mathfrak{I}^{\mathrm{unb}}_{\Pi_{\underline{Q},p}}\cdot\displaystyle{\prod_{q\vert N}}\mathfrak{I}_{\Pi_{\underline{Q},q}}.$$
\end{Corollary}
\subsection{The calculation of the local zeta integrals.}
Let us keep the notation as \S3.3. In this subsection, we calculate the local zeta integrals $\mathfrak{I}^{\mathrm{unb}}_{\Pi_{p}}$ and $\mathfrak{I}_{\Pi,q}$ for each finite prime $q\vert N$.
\begin{Definition}
 We say that an irreducible representation of $\pi_{q}$ of $\mathrm{GL}_{2}(\mathbb{Q}_{q})$ is minimal if $\pi_{q}$ satisfies $c(\pi_{q})\leq c(\pi_{q}\otimes\chi)$ for any character $\chi$ of $\mathbb{Q}_{q}$.
\end{Definition}
Ming-Lun Hsieh compute the integral $I_{q}(\phi_{q}^{*}\otimes\tilde{\phi}_{q}^{*})$ in the following hypothesis.
\begin{Hypothesis}
For any $q\vert N$, there exists an element $\sigma\in S_{3}$ such that
\begin{description}
\item[(1)] $c(\pi^{\prime}_{\sigma(1),q})\leq\mathrm{min}\{c(\pi^{\prime}_{\sigma(2),q}),c(\pi^{\prime}_{\sigma(3),q})\}$,
\item[(2)] $\pi^{\prime}_{\sigma(1),q}$ and $\pi^{\prime}_{\sigma(3),q}$ are minimal,
\item[(3)] $\pi^{\prime}_{\sigma(3),q}$ is a principal series or both of $\pi^{\prime}_{\sigma(2),q}$ and $\pi^{\prime}_{\sigma(3),q}$ are discrete series.
\end{description}
\end{Hypothesis}
We fix a prime $q\mid N$. In this subsection, we assume that the triple $(f_{1},g_{2},g_{3})$ satisfies Hypothesis 3.4.2. Further by symmetry we assume that $c(\pi^{\prime}_{1,q})\leq\mathrm{min}\{c(\pi^{\prime}_{2,q}),c(\pi^{\prime}_{3,q})\}$ and $\pi^{\prime}_{3,q}$ are minimal.
\begin{Remark}
In fact, by Hypothesis (6), the triple $(f_{1},g_{2},g_{3})$ satisfies Hypothesis 3.4.2. In Remark 5.1.4, we check that the triple $(f_{1},g_{2},g_{3})$ satisfies Hypothesis 3.4.2.
\end{Remark}
By Hypothesis 3.4.2 and Hypothesis (3), $\{\pi_{1,q}^{\prime},\pi_{2,q}^{\prime},\pi_{3,q}^{\prime}\}$ satisfies one of the following conditions.
\begin{description}
\item[Case (Ia)] The representation $\pi^{\prime}_{3,q}$ is a principal series.
\item[Case (Ib)] The representation $\pi^{\prime}_{i,q}$ is a discrete series for all $i=1,2,3$.
\item[Case (I\hspace{-.1em}Ia)] The representation $\pi^{\prime}_{1,q}$ is a principal series and both of $\pi^{\prime}_{2,q}$ and $\pi^{\prime}_{3,q}$ are discrete series such that $L(s,\pi^{\prime}_{2,q}\otimes\pi^{\prime}_{3,q})\neq1$.
\item[Case (I\hspace{-.1em}Ib)] The representation $\pi^{\prime}_{1,q}$ is a principal series and both of $\pi^{\prime}_{2,q}$ and $\pi^{\prime}_{3,q}$ are discrete series such that $L(s,\pi^{\prime}_{2,q}\otimes\pi^{\prime}_{3,q})=1$.
\end{description}
Ming-Lun Hsieh calculated the local integral $I_{q}(\phi_{q}^{*}\otimes\tilde{\phi}_{q}^{*})$ in each of the above cases in \cite[\S 6]{Hsi17}. However, we only need the calculation result of $I_{q}(\phi_{q}^{*}\otimes\tilde{\phi}_{q}^{*})$ in the conditions Case (Ia), Case (Ib) and Case (I\hspace{-.1em}Ib). Further, when the condition Case (Ib) holds, we need the result only if $L(s,\pi_{2^{\prime},q}\otimes\pi_{3^{\prime},q})=1$ holds. Hence, in this paper we restrict ourselves to the cases which are needed. We put $c^{*}=\mathrm{max}\{c(\pi^{\prime}_{2,q}),c(\pi^{\prime}_{3,q})\}$.
\begin{Proposition}
\begin{enumerate}
\item[(1)] In the case $\mathrm{[Case (Ia)]}$, we have
$$I_{q}(\phi_{q}^{*}\otimes\tilde{\phi}_{q}^{*})=\epsilon(1\slash2,\pi^{\prime}_{1,q}\otimes\pi^{\prime}_{2,q}\otimes\chi_{3})\cdot\chi_{3}^{-2}\vert\cdot\vert_{q}(q^{c^{*}})\omega^{\prime}_{3,q}(-1)\epsilon(1\slash2,\pi^{\prime}_{3,q})^{2}\cdot\frac{1}{B_{\Pi_{\underline{Q},q}}}\cdot\frac{\zeta_{q}(2)^{2}}{\zeta_{q}(1)^{2}},$$
where $\pi^{\prime}_{3,q}=\pi(\chi_{3},v_{3})$ is a principal series where $\chi_{3}$ is unramified.\\
\item[(2)] In the case $\mathrm{[Case (Ib)]}$, if we have $L(s,\pi^{\prime}_{2,q}\otimes\pi^{\prime}_{3,q})=1$, then we have 
$$I_{q}(\phi_{q}^{*},\tilde{\phi}_{q}^{*})=\chi_{1}^{2}\vert\cdot\vert_{q}(q^{c^{*}})\epsilon(1\slash2,\pi^{\prime}_{2,q}\otimes\pi^{\prime}_{3,q}\otimes v_{1})\epsilon(1\slash2,\pi^{\prime}_{2,q})^{2}\epsilon(1\slash2,\pi^{\prime}_{3,q})^{2}\cdot\frac{1}{B_{\Pi_{\underline{Q},q}}}\cdot\frac{\zeta_{q}(2)^{2}}{\zeta_{q}(1)^{2}},$$
where $\pi^{\prime}_{1,q}=\chi_{1}\vert\cdot\vert_{q}^{-\frac{1}{2}}\mathrm{St}$ is a special representation and $v_{1}=\chi_{1}\vert\cdot\vert_{q}^{-1}$.\\
\item[(3)] In the case $\mathrm{[Case (I\hspace{-.1em}Ib)]}$, we have
$$I_{q}(\phi_{q}^{*}\otimes\tilde{\phi}_{q}^{*})=\omega^{\prime}_{1,q}(-1)\chi_{1}^{-2}\vert\cdot\vert_{q}(q^{c^{*}})\epsilon(1\slash2,\pi^{\prime}_{1,q})^{2}\cdot\epsilon(1\slash2,\pi^{\prime}_{2,q}\otimes\pi^{\prime}_{3,q}\otimes\chi_{1})\cdot\frac{1}{B_{\Pi_{\underline{Q},q}}}\cdot\frac{\zeta_{q}(2)^{2}}{\zeta_{q}(1)^{2}},$$
where $\pi^{\prime}_{1,q}=\pi(\chi_{1},v_{1})$ is a principal series where $\chi_{1}$ is unramified.
\end{enumerate}
Here, $B_{\Pi_{\underline{Q},q}}=\prod_{i=1}^{3}B_{\pi^{\prime}_{i,q}}$.
\end{Proposition}
Next, we calculate $\mathfrak{I}^{\mathrm{unb}}_{\Pi_{\underline{Q},p}}$ by the way based on \cite[Proposition 5.4]{Hsi17}. However, before we calculate $\mathfrak{I}^{\mathrm{unb}}_{\Pi_{\underline{Q},p}}$, we prepare Lemma 3.4.5 and Proposition 3.4.6 to check that the proof of \cite[Proposition 5.4]{Hsi17} holds even if $g_{i}$ is non-ordinary at $p$.
\begin{Lemma}
Let $L$ be a positive integer such that $(L,p)=1$. Let $\pi_{h}=\otimes^{\prime}\pi_{h,v}$ be an irreducible automorphic representation of $\mathrm{GL}_{2}(\mathbb{A})$ with the central character $\omega_{h}$ attached to a normalized Hecke eigenform $h\in S_{k}(p^{t}L,\chi)$. Here, $k,t\in \mathbb{Z}_{\geq1}$ and $\chi$ is a Dirichlet character. We assume that $h$ is primitive outside of $p$ and $a(p,h)\neq 0$. Let $\alpha_{h,p}:\mathbb{Q}_{p}^{\times}\rightarrow \mathbb{C}^{\times}$ be an unramified character with $\alpha_{h,p}(p)=p^{\frac{1}{2}(1-k)}a(p,h)$. We put $v_{p}=\omega_{h,p}\alpha_{h,p}^{-1}$. Then, we have $\alpha_{h,p}v_{p}^{-1}\neq\vert\cdot\vert_{p}^{-1}$ and $\pi_{h,p}$ is a constituent of the induced representation $\pi(v_{p},\alpha_{h,p})$.
\end{Lemma}
\begin{proof}
Because $\mathcal{V}_{\pi_{h,p}}(\alpha_{h,p})\neq\{0\}$, by Proposition 2.3.4,  $\pi_{h,p}$ is a constituent of the representation $\pi(v_{p},\alpha_{h,p})$. Then, we prove $\alpha_{h,p}v_{p}^{-1}\neq\vert\cdot\vert_{p}^{-1}$. We decompose $m_{\chi}=p^{s}L^{\prime}$ where $s\in\mathbb{Z}_{\geq0}$ and $L^{\prime}$ is a positive integer which is prime to $p$. Then, because $h$ is a Hecke eigenform and $a(h,p)\neq 0$, we have (1) $t=s$ or (2) $t=1$ and $s=0$ by \cite[Lemma 4.5.11]{Miy06}.

First, we assume that $h$ is primitive. Then, by \cite[Theorem 4.6.17]{Miy06}, we have
\begin{align*}
&\vert a(p,h)\vert_{\infty}^{2}=p^{k-1}\ &\mathrm{if\ (1)\ holds},\\
&a(p,h)^{2}=\chi_{0}(p)p^{k-2}\ &\mathrm{if\ (2)\ holds},
\end{align*} 
where $\chi_{0}$ is a primitive Dirichlet character associated the Dirichlet character $\chi$. Hence, we have 
$$\vert\alpha_{h,p}v_{p}^{-1}(p)\vert_{\infty}=\vert\alpha_{h,p}(p)\vert_{\infty}^{2}=p^{1-k}\vert a(p,h)\vert_{\infty}^{2}=\begin{cases}
1\ &\mathrm{if\ (1)\ holds},\\
p^{-1}\ &\mathrm{if\ (2)\ holds}.
\end{cases}$$
Then, we have $\alpha_{h,p}v_{p}^{-1}\neq \vert\cdot\vert_{p}^{-1}$ if $h$ is a primitive form.

Next, we assume that $h$ is not a primitive form. Because $h$ is a normalized Hecke eigenform, by \cite[Theorem 4.6.12]{Miy06}, $h$ is an oldform. Because $h$ is an oldform and $h$ is primitive outside of $p$, we have $t\neq s$. Then, we have $t=1$ and $s=0$. Further, there exists a unique primitive form $h_{0}\in S_{k}(L,\chi_{0})$ such that $a(l,h)=a(l,h_{0})$ for all but a finite number of primes $l$. Because $\pi_{h}$ is attached to $h_{0}$, there exists a $\mathbf{K}_{p}$-fixed point in $\pi_{h,p}$. Then, $\pi_{h,p}$ is a principal series and isomorphic to $\pi(v_{p},\alpha_{h,p})$ ($cf$. \cite[Theorem 4.6.4]{Bum97}). Hence, we have $\alpha_{h,p}v_{p}^{-1}\neq\vert\cdot\vert_{p}^{-1}$ ($cf$. \cite[Theorem 4.5.1]{Bum97}). We complete the proof.
\end{proof}
If the conductor $c(\pi_{h,p})$ of $\pi_{h,p}$ defined in (2.3.1) is non-zero, it is well-known that the operator $U_{p}$ induces $U_{p}:\mathcal{V}_{\pi_{h,p}}^{\mathrm{new}}\rightarrow\mathcal{V}_{\pi_{h,p}}^{\mathrm{new}}$.
\begin{Proposition}
Let $h\in S_{k}(L,\chi)$ be a primitive form with $k,L\in\mathbb{Z}_{\geq 1}$. We assume $a(p,h)=0$. Then, the normalized Whittaker newform $W_{\pi_{h,p}}^{\mathrm{new}}\in\mathcal{W}(\pi_{h,p})$ is characterized by

$$W_{\pi_{h,p}}^{\mathrm{new}}(\mathrm{diag}\{a,1\})=\mathbf{I}_{\mathbb{Z}_{p}^{\times}}(a),\ {}^{\forall}a\in\mathbb{Q}_{p}^{\times},$$
where $\mathbf{I}_{\mathbb{Z}_{p}^{\times}}$ is the characteristic function on $\mathbb{Z}_{p}^{\times}$.
\end{Proposition}
\begin{proof}
By \cite[Theorem 4.6.17]{Miy06}, we have $c(\pi_{h,p})\geq2$. Hence, by the results of \cite[\S1.2]{Sch02}, $\pi_{h,p}$ is isomorphic to any of the following representations.
\begin{enumerate}
\item A supercuspidal representation.
\item A principal series $\pi(\chi_{1},\chi_{2})$ for characters $\chi_{1}$ and $\chi_{2}$ such that at least one of the characters is ramified.
\item A special representation $\chi\mathrm{St}$ for a ramified character $\chi$.
\end{enumerate}
By the results of \cite[\S2.4]{Sch02}, it suffices to prove that $\pi$ is not isomorphic to a principal series $\pi(\chi_{1},\chi_{2})$ such that the only one of the characters is unramified. Then, we assume that $\pi_{h,p}$ is isomorphic to a principal series $\pi(\chi_{1},\chi_{2})$ such that the only one of the characters $\chi_{1},\chi_{2}$ is unramified. Because the representation $\pi(\chi_{1},\chi_{2})$ is isomorphic to $\pi(\chi_{2},\chi_{1})$, we can assume that $\chi_{2}$ is unramified. Then, by the result of \cite[\S2.4]{Sch02}, the normalized Whittaker newform $W_{\pi_{h,p}}^{\mathrm{new}}$ is characterized by
$$W_{\pi_{h,p}}^{\mathrm{new}}(\mathrm{diag}\{a,1\})=\vert a\vert_{p}^{\frac{1}{2}}\chi_{2}(a)\mathbf{I}_{\mathbb{Z}_{p}},\ {}^{\forall}a\in\mathbb{Q}_{p}^{\times}.$$
In particular, we have
$U_{p}(W_{\pi_{h,p}}^{\mathrm{new}})(1)=\displaystyle{\sum_{i\in\mathbb{Z}\slash p\mathbb{Z}}}W_{\pi_{h,p}}^{\mathrm{new}}\left(\left(
    \begin{array}{cc}
      p&i \\
      0&1 \\
    \end{array}
  \right)\right)=p^{\frac{1}{2}}\chi_{2}(p)\neq0$. However, because $a(p,h)=0$, we have $U_{p}(W_{\pi_{h,p}}^{\mathrm{new}})=0$. This is contradict. We complete the proof. 
\end{proof}
For each irreducible representation $\pi\otimes\pi^{\prime}$ of $\mathrm{GL}_{2}(\mathbb{Q}_{p})\times\mathrm{GL}_{2}(\mathbb{Q}_{p})$, Let  $\epsilon(s,\pi\otimes\pi^{\prime})$ be the $\epsilon$-factor attached to $\pi\otimes\pi^{\prime}$ defined in \cite{Jac72}. We set 
$$\mathcal{E}_{f_{1}}(\Pi_{\underline{Q},p}):=\dfrac{\epsilon(1\slash2,\pi_{2,p}\otimes\pi_{3,p}\otimes\alpha_{f_{1},p}(\chi_{\underline{Q}})_{\mathbb{A},p}^{-1})^{-1}}{L(1\slash2,\Pi_{\underline{Q},p})}\cdot\dfrac{L(1\slash2,\pi_{2,p}\otimes\pi_{3,p}\otimes\alpha_{f_{1},p}(\chi_{\underline{Q}})_{\mathbb{A},p})}{L(1\slash2,\pi_{2,p}\otimes\pi_{3,p}\otimes \alpha_{f_{1},p}(\chi_{\underline{Q}})_{\mathbb{A},p}^{-1})}.\leqno(3.4.1)$$
Here, $\alpha_{f_{1},p}$ is the unramified character defined in Remark 2.4.2. The following proposition is same as \cite[Proposition 5.4]{Hsi17} except that we do not assume that $g_{i}$ is non-ordinary for $i=2,3$.
\begin{Proposition}
We have
$$\mathfrak{J}^{\mathrm{unb}}_{\Pi_{\underline{Q},p}}=\mathcal{E}_{f_{1}}(\Pi_{\underline{Q},p}).$$
\end{Proposition}
\begin{Remark}
The proof of Proposition 3.4.7 is same as \cite[Proposition 5.4]{Hsi17}. However, we note some points. If the slope of $g_{i}$ is $\infty$, we have to take the newline (resp. the normalized local Whittaker newform) of $\pi^{\prime}_{i,p}$ instead of the oridnary line (resp. the normalized ordinary Whittaker function) of $\pi^{\prime}_{i,p}$. If the slope of $g_{i}$ is not $\infty$, by Lemma 3.4.5, $\pi^{\prime}_{i,p}$ is isomorphic to the unique irreducible quotient $\pi(v_{i,p},\alpha_{g_{i},p})^{0}$ of  $\pi(v_{i,p},\alpha_{g_{i},p})$ even if $g_{i}$ is non-ordinary. Here, $\alpha_{g_{i},p}$ is the unramified character defined in Remark 2.4.2 and $v_{i,p}=\omega^{\prime}_{i,p}\alpha_{g_{i},p}^{-1}$.
\end{Remark}
\section{The triple product $p$-adic $L$-functions in the unbalanced case.}
Let $(F,G^{(2)},G^{(3)})$ be the same as in \S3. In \S4, we construct a three-variable triple product $p$-adic $L$-function attached to the triple $(F,G^{(2)},G^{(3)})$. From now on throughout in this paper, for simplicity, we write $G^{(i)}(m):=G_{Q^{(i)}_{m}}$, $k^{(i)}(m):=k_{Q^{(i)}_{m}}$, $e^{(i)}(m):=e_{Q^{(i)}_{m}}$ and $\epsilon^{(i)}_{m}:=\epsilon_{Q^{(i)}_{m}}$ for each $Q^{(i)}_{m}\in \mathfrak{X}^{(i)}$, where $i=2,3$ and $m\in\mathbb{Z}_{\geq 1}$.
\subsection{Adjustment of levels of $p$-adic families of modular forms.}
In this subsection, for $i=2,3$, we define a good test formal series $G^{(i),*}\in \mathbf{I}_{i}\jump{q}$ attached to $G^{(i)}$ which is $p$-adic families of cusp forms of tame level $N$. The way to construct the series $G^{(i),*}$ is based on \cite[\S3.5]{Hsi17}. However, we have to prove the rigidities of $G^{(i)}$ to construct the series $G^{(i),*}$.
\begin{Proposition}
We fix an element $i\in\{2,3\}$. Let $m$ be a positive integer and $\pi_{m}$ be the automorphic representation with the central character $\omega_{m}$ attached to $G^{(i)}(m)$. We put $m_{i}:=m_{\psi_{i}}.$ Then, at any prime $l\vert N_{i}$, the automorphic type of  $\pi_{m}$ is independent of $m\in\mathbb{Z}_{\geq 1}$. In particular, we have  
\begin{enumerate}
\item[(1)] If $v_{l}(N_{i})=v_{l}(m_{i})$, $\pi_{m,l}$ is isomorphic to a principal series.
\item[(2)] If $v_{l}(N_{i})>v_{l}(m_{i})$, $\pi_{m,l}$ is isomorphic to a special representation.
\end{enumerate}
Further, the representation $\pi_{m,l}$ is minimal.
\end{Proposition}
\begin{proof}
We put $g_{m}:=G^{(i)}(m)$ and we denote by $g_{m}^{\prime}$ the primitive form attached to $g_{m}$. Because $g_{m}$ is new at $l$, we have $a(l,g_{m})=a(l,g_{m}^{\prime})$. Further, by Hypothesis (7), we have $a(l,g_{m}^{\prime})\neq0$ for all but a finite number of $m\in \mathbb{Z}_{\geq 1}$. In particular, by \cite[Theorem 4.6.17]{Miy06}, the pair $(N_{i},m_{i})$ satisfies either of the following two cases. 
\begin{enumerate}
\item[$(1)^{\prime}$] $v_{l}(N_{i})=v_{l}(m_{i})$,
\item[$(2)^{\prime}$] $v_{l}(N_{i})=1$ and $v_{l}(m_{i})=0$.
\end{enumerate}
Then, by \cite[Theorem 4.6.17]{Miy06}, we have $a(l,g_{m}^{\prime})\neq0$ for all $m\in \mathbb{Z}_{\geq 1}$. We fix a prime $l\vert N_{i}$ and a positive integer $m\in\mathbb{Z}_{\geq 1}$. We can define an unramified character $\alpha_{m,l}:\mathbb{Q}_{l}^{\times}\rightarrow\mathbb{C}^{\times}$ with $\alpha_{m,l}(l)=a(l,g_{m}^{\prime})l^{\frac{1}{2}(1-k^{(i)}(m))}$. Then, by Proposition 2.3.4, $\pi_{m,l}$ is isomorphic to either of the principal series $\pi(\alpha_{m,l},\alpha_{m,l}^{-1}\omega_{m,l})$ or the special representation $\alpha_{m,l}\vert\cdot\vert_{l}^{-\frac{1}{2}}\mathrm{St}$. 

First, we prove (1). If the condition (1) holds, by \cite[Theorem 4.6.17]{Miy06}, we have $\vert a(l,g_{m}^{\prime})\vert_{\infty}^{2}=l^{k^{(i)}(m)-1}$ . Then, we have 
$$\vert\alpha_{m,l}^{2}\omega_{m,l}^{-1}(l)\vert_{\infty}=\vert a(l,g_{m}^{\prime})^{2}l^{1-k^{(i)}(m)}\vert_{\infty}=1.$$
Hence, we have $\alpha_{m,l}^{2}\omega_{m,l}^{-1}\neq\vert\cdot\vert_{l}^{\pm}$ and $\pi_{m,l}$ is isomorphic to the principal series $\pi(\alpha_{m,l},\alpha_{m,l}^{-1}\omega_{m,l})$ ($cf$. \cite[Theorem 4.5.1]{Bum97}).

Next, we prove (2). If the condition (2) holds, by \cite[Theorem 4.6.17]{Miy06}, we have $v_{l}(N_{i})=1$ and $v_{l}(m_{i})=0$. Then, $\psi_{i,\mathbb{A}}$ is unramified at $l$ and hence $\omega_{m,l}$ is unramified. Let $\psi_{m}^{\prime}$ be the primitive character attached to the Nebentypus of $g_{m}$. Because $v_{l}(m_{i})=0$, the modulo of $\psi_{m}^{\prime}$ is prime to $l$. Then, we can define $\psi_{m}^{\prime}(l)$ and we have $\omega_{m,l}(l)=\psi_{m}^{\prime}(l)$. Further by \cite[Theorem 4.6.17]{Miy06}, we have $a(l,g_{m}^{\prime})^{2}=\psi_{m}^{\prime}(l)l^{k^{(i)}(m)-2}$. Then, we have
$$\alpha_{m,l}^{2}\omega_{m,l}^{-1}(l)=a(l,g_{m}^{\prime})^{2}l^{1-k^{(i)}(m)}{\psi_{m}^{\prime}}^{-1}(l)=l^{-1}.$$
Because $\omega_{m,l}$ is unramified, we have $\alpha_{m,l}^{2}\omega_{m,l}^{-1}=\vert\cdot\vert_{l}$. Then, $\pi_{m,l}$ is isomorphic to the special representation $\alpha_{m,l}\vert\cdot\vert_{l}^{-\frac{1}{2}}\mathrm{St}$. 

It is not difficult to prove that the minimality of the representation $\pi_{m,l}$. For each character $\xi_{l}:\mathbb{Q}_{l}^{\times}\rightarrow \mathbb{C}^{\times}$,
$$\pi_{m,l}\otimes\xi_{l}\cong\begin{cases}
\pi(\xi_{l}\alpha_{m,l},\xi_{l}\alpha_{m,l}^{-1}\omega_{m,l})\ &\mathrm{case}\ (1),\\
\xi_{l}\alpha_{m,l}\vert\cdot\vert_{l}^{-\frac{1}{2}}\mathrm{St}\ &\mathrm{case}\ (2).
\end{cases}
$$
Then, by the simple calculation, we can prove that $\pi_{m,l}$ is minimal. We complete the proof.
\end{proof}
The rigidity of automorphic types in a primitive Hida family is proved in \cite[Lemma 2.14]{FO12}. By \cite[Lemma 2.14]{FO12} and Proposition 4.1.1, we can adjoint the level of $G^{(2)}$ and $G^{(3)}$. If necessary, we exchange $\mathbf{I}_{i}$ to a normal finite flat extension for $i=2,3$, and we assume that $\mathbf{I}_{i}$ contains the roots of the polynomials 
$$X^{2}-a(l,G^{(i)})X+\psi_{i}(l)l^{-1}\langle l\rangle^{i}\in\mathbf{I}_{i}[X]\leqno(4.1.1)$$
for each $l\vert N$. We fix a root $\beta_{l}(G^{(i)})\in\mathbf{I}_{i}^{\times}$ of the polynomial $X^{2}-a(l,G^{(i)})X+\psi_{i}(l)l^{-1}\langle l\rangle^{i}$. Let $(d_{1},d_{2},d_{3})$ be the triple of integers defined in (3.1.1). Let $\Sigma_{i,0}^{(\mathrm{I\hspace{-.1em}I}\mathrm{b})}$ be the set defined in (3.1.2) for $i=1,2,3$. By Proposition 4.1.1 and \cite[Lemma 2.14]{FO12}, these $(d_{1},d_{2},d_{3})$ and $\Sigma_{i,0}^{(\mathrm{I\hspace{-.1em}I}\mathrm{b})}$ are independent of the specialization $\underline{Q}=(Q_{1},Q_{m_{2}}^{(2)},Q_{m_{3}}^{(3)})\in\mathfrak{X}_{\mathbf{I}_{1}}\times\mathfrak{X}^{(2)}\times\mathfrak{X}^{(3)}$. Let $V^{(i)}_{m}:R^{(i)}\jump{q}\rightarrow R^{(i)}\jump{q}$ be an $R^{(i)}$-module homomorphism defined by $V^{(i)}_{m}(\sum_{n\geq0}a_{n}q^{n})=\sum_{n\geq 0}a_{n}q^{mn}$ for each $m\in\mathbb{Z}_{\geq 1}$. We define a formal series 
$$G^{(i),*}=\displaystyle{\sum_{I\subset \Sigma_{i}^{(\mathrm{I\hspace{-.1em}I}\mathrm{b})}}}(-1)^{\vert I\vert}\beta_{I}(G^{(i)})^{-1}V^{(i)}_{d_{i}\slash n_{I}}G^{(i)},\leqno(4.1.2)$$
where $n_{I}=\prod_{l\in I}l$, $\beta_{I}(G^{(i)})=\prod_{l\in I}\beta_{l}(G^{(i)})$ for $i=2,3$.
\subsection{The construction of the triple product $p$-adic $L$-function.}
We denote by $R:=\mathbf{I}_{1}\widehat{\otimes}_{\mathcal{O}_{K}}\mathbf{I}_{2}\widehat{\otimes}_{\mathcal{O}_{K}}\mathbf{I}_{3}$ the complete tensor product of $\mathbf{I}_{1},\mathbf{I}_{2}$ and $\mathbf{I}_{3}$ over $\mathcal{O}_{K}$. In this subsection, we construct a triple product $p$-adic $L$-function $L^{F}_{G^{(2)},G^{(3)}}\in R$ attached to $F,G^{(2)}$ and $G^{(3)}$. First, we define a set of interpolation points attached to $F,G^{(2)}$ and $G^{(3)}$.
\begin{Definition}
We define a set of interpolation points of $R$ to be
$$\mathfrak{X}_{R}:=\{\underline{Q}=\left(Q_{1},Q^{(2)}_{m_{2}},Q^{(3)}_{m_{3}}\right)\in\mathfrak{X}_{\mathbf{I}_{1}}\times\mathfrak{X}^{(2)}\times\mathfrak{X}^{(3)}\mid k_{Q_{1}}+k^{(2)}(m_{2})+k^{(3)}(m_{3})\equiv 0 \ (\mathrm{mod}\ 2)\}$$
and a subset of $\mathfrak{X}_{R}^{F}$ consisting of unbalanced points to be
$$\mathfrak{X}_{R}^{F}:=\{(Q_{1},Q^{(2)}_{m_{2}},Q^{(3)}_{m_{3}})\in\mathfrak{X}_{R}\mid k_{Q_{1}}\geq k^{(2)}(m_{2})+k^{(3)}(m_{3})\}.$$
\end{Definition}
We define a formal operator $\mathbf{U}_{R,m}\in \mathrm{End}_{R}(R\jump{q})$ for each $m\in\mathbb{Z}_{\geq 1}$ to be
$$\mathbf{U}_{R,m}(f)=\displaystyle{\sum_{n\geq 0}}a(mn,f)q^{n}$$
for each $f=\sum_{n\geq 0}a(n,f)q^{n}\in R\jump{q}$. Let $\Theta:\mathbb{Z}_{p}^{\times}\rightarrow R^{\times}$ be a character defined by
$$\Theta(z)=\psi_{1,(p)}\omega_{p}^{-a}(z)\langle z{\rangle_{\mathbf{I}_{1}}}^{\frac{1}{2}}(\langle z\rangle^{(2)}\langle z\rangle^{(3)})^{-\frac{1}{2}},$$
for each $z\in \mathbb{Z}_{p}^{\times}$. For each $f\in\displaystyle{\sum_{n\geq 0}}a(n,f)q^{n}\in R\jump{q}$, we define a $\Theta$-twisted form $f\vert\lbrack\Theta\rbrack\in R\jump{q}$ to be
$$f\vert\lbrack\Theta\rbrack=\displaystyle{\sum_{p\nmid n}}\Theta(n)\cdot a(n,f)q^{n}.$$
For each $\underline{Q}=(Q_{1},Q^{(2)}_{m_{2}},Q^{(3)}_{m_{3}})\in\mathfrak{X}_{R}^{F}$, we have $f\vert[\Theta](\underline{Q})=d^{r_{\underline{Q}}}(f(\underline{Q})\vert[\Theta_{\underline{Q}}])$ with the Dirichlet character
$$\Theta_{\underline{Q}}=\psi_{1,(p)}\omega_{p}^{-a-r_{\underline{Q}}}\epsilon_{Q_{1}}^{\frac{1}{2}}{\epsilon^{(2)}_{m_{2}}}^{-\frac{1}{2}}{\epsilon^{(3)}_{m_{3}}}^{-\frac{1}{2}},$$ 
where $r_{\underline{Q}}=\frac{1}{2}(k_{Q_{1}}-k^{(2)}(m_{2})-k^{(3)}(m_{3}))$. We regard $G^{(2)}$ and $G^{(3)}$ as elements of $R\jump{q}$ by natural embeddings $\mathbf{I}_{2}\hookrightarrow R$ and $\mathbf{I}_{3}\hookrightarrow R$. Let $G^{(2),*}$ and $G^{(3),*}$ be the formal series attached to $G^{(2)}$ and $G^{(3)}$ defined in (4.1.2). We set $H:=G^{(2),*}\cdot (G^{(3),*}\vert\lbrack\Theta\rbrack)\in R\jump{q}$.
\begin{Lemma}
We fix an element $\underline{Q}=(Q_{1},Q^{(2)}_{m_{2}},Q^{(3)}_{m_{3}})\in\mathfrak{X}_{R}^{F}$. Let $L$ be a finite extension of $K$ such that $\mathcal{O}_{L}$ contains $Q_{1}(\mathbf{I}_{1}),Q^{(2)}_{m_{2}}(\mathbf{I}_{2})$ and $Q^{(3)}_{m_{3}}(\mathbf{I}_{3})$. Then, the sequence $\{U_{\mathbf{R},p}^{n!}H(\underline{Q}) \}_{n\geq1}$ converges in $\mathcal{O}_{L}\jump{q}$ by the $(p)$-adic topology and the limit of the sequence equals to the Fourier expansion of $e\mathcal{H}(G^{(2),*}(m_{2})\delta_{k^{(3)}(m_{3})}^{r_{\underline{Q}}}G^{(3),*}(m_{3})\vert\Theta_{\underline{Q}})\in eS_{k_{Q_{1}}}(Np^{e_{Q_{1}}},\psi_{1,(p)}\overline{\psi}_{1}^{(p)}\omega_{p}^{k_{Q_{1}}}\epsilon_{Q_{1}})$, with $e_{Q_{1}}:=\mathrm{max}\{1,m_{\epsilon_{Q_{1}}}\}$.
\end{Lemma}
\begin{proof}
It is known that $H(\underline{Q})$ is a Fourier expansion of a $p$-adic modular form and by \cite[Lemma 5.2]{Hid88c}, we have
$$H(\underline{Q})=\mathcal{H}(G^{(2),*}(m_{2})\delta_{k^{(3)}(m_{3})}^{r_{\underline{Q}}}G^{(3),*}(m_{3})\vert\Theta_{\underline{Q}})+dg^{\prime}_{\underline{Q}}\in L\jump{q},$$
where $g^{\prime}_{\underline{Q}}\in L\jump{q}$ is a $p$-adic modular form. By \cite[(6.12)]{Hid88c}, $ed=0$ and hence $eH(\underline{Q})=e\mathcal{H}(G^{(2),*}(m_{2})\delta_{k^{(3)}(m_{3})}^{r_{\underline{Q}}}G^{(3),*}(m_{3})\vert\Theta_{\underline{Q}})$. Further, by \cite[(4.3)]{Hid88c}, the sequence $\{U_{R,p}^{n!}H(\underline{Q}) \}_{n\geq1}$ converges in $\mathcal{O}_{L}\jump{q}$ by the $(p)$-adic topology and the limit of the sequence equals to $eH(\underline{Q})$. We complete the proof.
\end{proof}
To construct a triple product $p$-adic $L$-function $L^{F}_{G^{(2)},G^{(3)}}\in R$, we prove the following lemma and proposition. 
\begin{Lemma}
There exists a unique element $H^{\mathrm{ord}}\in R\jump{q}$ such that the specialization of $H^{\mathrm{ord}}$ at each $\underline{Q}=(Q_{1},Q^{(2)}_{m_{2}},Q^{(3)}_{m_{3}})\in\mathfrak{X}_{R}^{F}$ equals to the Fourier expansion of the modular form $e\mathcal{H}(G^{(2),*}(m_{2})\delta_{k^{(3)}(m_{3})}^{r_{\underline{Q}}}G^{(3),*}(m_{3})\vert\Theta_{\underline{Q}})$.
\end{Lemma}
\begin{proof}
Let $\mathfrak{m}_{R}$ be the maximal ideal of $R$. Let $I_{\underline{Q}}$ be the ideal of $R$ generalized by $\mathrm{Ker}Q_{1}$,$\mathrm{Ker}Q^{(2)}_{m_{2}}$ and $\mathrm{Ker}Q^{(3)}_{m_{3}}$ for each $\underline{Q}=(Q_{1},Q^{(2)}_{m_{2}},Q^{(3)}_{m_{3}})\in\mathfrak{X}_{R}^{F}$. We denote by $\mathfrak{B}$ the set of finite intersections of $I_{\underline{Q}}$ for $\underline{Q}\in \mathfrak{X}_{R}^{F}$. Then, we can easily check that $\displaystyle{\cap_{J\in\mathfrak{B}}}J=\{0\}$. Further, we have the natural isomorphism $R\cong\plim[J\in\mathfrak{B}](R\slash J)$. In particular, we have
\begin{align*}
R\jump{q}\cong\plim[J\in\mathfrak{B}]R\jump{q}\otimes_{R}(R\slash J).
\end{align*}
For each $J=\cap_{i=1}^{m}I_{\underline{Q}_{i}}\in\mathfrak{B}$, it suffices to prove that there exists a unique element $H_{J}^{\mathrm{ord}}\in R\jump{q}\otimes_{R}(R\slash J)$ such that the image of $H_{J}^{\mathrm{ord}}$ by the natural embedding $i_{J}: R\jump{q}\otimes_{R}(R\slash J)\hookrightarrow\displaystyle{\prod_{i=1}^{m}}(R\jump{q}\otimes_{R}R\slash I_{\underline{Q}_{i}})$ equals to $\left[e(H(\underline{Q}_{i}))\right]_{i=1}^{m}$. The uniqueness of $H_{J}^{\mathrm{ord}}$ is trivial. We prove the existence of $H_{J}^{\mathrm{ord}}$.

Let $p_{J}:R\jump{q}\rightarrow R\otimes_{R}(R\slash J)$ be the natural projection. If $J=I_{\underline{Q}}$ for $\underline{Q}\in \mathfrak{X}_{R}^{F}$, by Lemma 4.2.2, there exists an element $H^{\mathrm{ord}}_{J}\in R\jump{q}\otimes_{R}(R\slash J)$ such that $H^{\mathrm{ord}}_{J}(\underline{Q})=e(H(\underline{Q}))$. Further, we have $\displaystyle{\lim_{n\to \infty}} p_{J}(U_{R,p}^{n!}H)=H^{\mathrm{ord}}_{J}$ by the $\mathfrak{m}_{R}$-adic topology. We assume that there exist elements $H^{\mathrm{ord}}_{J}=\displaystyle{\lim_{n\to \infty}} p_{J}(U_{R,p}^{n!}H)\in R\jump{q}\otimes(R\slash J)$ and $H^{\mathrm{ord}}_{J^{\prime}}=\displaystyle{\lim_{n\to \infty}} p_{J^{\prime}}(U_{R,p}^{n!}H)\in R\jump{q}\otimes(R\slash J^{\prime})$ for a pair $(J,J^{\prime})\in \mathcal{B}\times \mathcal{B}$. We define the $R$-linear map:
$$
\begin{array}{ccc}
(R\jump{q}\otimes_{R}(R\slash J))\times (R\jump{q}\otimes_{R}(R\slash J^{\prime}))&\stackrel{i_{J,J^{\prime}}}{\longrightarrow} &(R\jump{q}\otimes_{R}(R\slash J+J^{\prime}))\\
\rotatebox{90}{$\in$} & & \rotatebox{90}{$\in$} \\
(a,b) & \longmapsto & a-b
\end{array}.
$$
Then, we have $i_{J,J^{\prime}}(H^{\mathrm{ord}}_{J},H^{\mathrm{ord}}_{J^{\prime}})=\displaystyle{\lim_{n\to \infty}}i_{J,J^{\prime}}(p_{J}(U_{R,p}^{n!}H),p_{J^{\prime}}(U_{R,p}^{n!}H))=0$. Further, because $\mathrm{Ker}\ i_{J,J^{\prime}}\cong R\jump{q}\otimes_{R}(R\slash J\cap J^{\prime})$, there exists a unique element $H^{\mathrm{ord}}_{J\cap J^{\prime}}\in R\jump{q}\otimes_{R}(R\slash J\cap J^{\prime})$ such that the image of $H^{\mathrm{ord}}_{J\cap J^{\prime}}$ in $(R\jump{q}\otimes_{R}(R\slash J))\times (R\jump{q}\otimes_{R}(R\slash J^{\prime}))$ equals to $(H^{\mathrm{ord}}_{J},H^{\mathrm{ord}}_{J^{\prime}})$. In particular, we have $H^{\mathrm{ord}}_{J\cap J^{\prime}}=\displaystyle{\lim_{n\to \infty}}p_{J\cap J^{\prime}}(U_{R,p}^{n!}H)$. Then, for each $J=\cap_{i=1}^{m}I_{\underline{Q}_{i}}\in\mathcal{B}$, there exists a unique element $H^{\mathrm{ord}}_{J}\in R\jump{q}\otimes_{R}(R\slash J)$ such that the image of $H_{J}^{\mathrm{ord}}$ by the natural embedding $i_{J}: R\jump{q}\otimes_{R}(R\slash J)\hookrightarrow\displaystyle{\prod_{i=1}^{m}}(R\jump{q}\otimes_{R}R\slash I_{\underline{Q}_{i}})$ equals to $\left[e(H(\underline{Q}_{i}))\right]_{i=1}^{m}$. We complete the proof.
\end{proof}
\begin{Proposition}
The power series $H^{\mathrm{ord}}$ is an element of $\mathbf{S}^{\mathrm{ord}}(N,\psi_{1,(p)}\overline{\psi}_{1}^{(p)},\mathbf{I}_{1})\widehat{\otimes}_{\mathbf{I}_{1}}R.$ 
\end{Proposition}
\begin{proof}
We identify the Iwasawa algebra $\Lambda$ with $\mathcal{O}_{K}\jump{X}$ by the isomorphism $[1+p] \mapsto 1+X$ and we regard $\mathbf{I}_{i}$ as the normal finite flat extension of $\mathcal{O}_{K}\jump{X_{i}}$ for $i=1,2,3$. Let $\alpha_{1},\alpha_{2},\ldots,\alpha_{n}$ be the bases of $R$ over $R_{0}=\mathcal{O}_{K}\jump{X_{1},X_{2},X_{3}}$. We put 
$$H^{\mathrm{ord}}=\sum_{i=1}^{n}H^{(i)}\alpha_{i},$$
where $H^{(i)}\in R_{0}\jump{q}$ for each $i=1,\ldots,n$. We put $L=\mathrm{Frac}{R}$ and $L_{0}=\mathrm{Frac}{R_{0}}$. Let $\mathrm{Tr}_{L\slash L_{0}}:L\rightarrow L_{0}$ be the trace map and $\alpha_{1}^{*},\alpha_{2}^{*},\ldots,\alpha_{n}^{*}$ be the dual bases of $\alpha_{1},\alpha_{2}\ldots,\alpha_{n}$ with respect to $\mathrm{Tr}_{L\slash L_{0}}$.Then, we have
$$H^{(i)}(\underline{Q})=\mathrm{Tr}(H(\underline{Q})\alpha_{i}^{*}(\underline{Q}))$$
for all but a finite number of $\underline{Q}=(Q_{1},Q^{(2)}_{m_{2}},Q^{(3)}_{m_{3}})\in\mathfrak{X}_{R}^{F}$. Further, $\mathrm{Tr}(H(\underline{Q})\alpha_{i}^{*}(\underline{Q}))$ is a $q$-expansion of an element of $S^{\mathrm{ord}}_{k_{Q_{1}}}(Np^{e_{Q_{1}}},\epsilon_{Q_{1}}\psi_{1,(p)}\overline{\psi}_{1}^{(p)}\omega_{p}^{-k_{Q_{1}}})$. It suffices to prove 
$$H^{(i)}\in \mathbf{S}^{\mathrm{ord}}(N,\psi_{1,(p)}\overline{\psi}_{1}^{(p)},\mathcal{O}_{K}\jump{X_{1}})\widehat{\otimes}_{\mathcal{O}_{K}\jump{X_{1}}}R_{0}$$
for each $i=1,\ldots,n$.

For each $m_{2},m_{3}\in\mathbb{Z}_{\geq 1}$, let $H^{(i)}_{m_{2},m_{3}}\in\mathcal{O}_{K}[b^{(2)}_{m_{2}},b^{(3)}_{m_{3}}]\jump{X_{1}}\jump{q}$ be the specialization of $H^{(i)}$ at $(Q^{(2)}_{m_{2}},Q^{(3)}_{m_{3}})$, where $b^{(2)}_{m_{2}}:=Q^{(2)}_{m_{2}}(X_{2})$ and $b^{(3)}_{m_{3}}:=Q^{(3)}_{m_{3}}(X_{3})$. First, we prove  $H^{(i)}_{m_{2},m_{3}}\in \mathbf{S}^{\mathrm{ord}}(N,\psi_{1,(p)}\overline{\psi}_{1}^{(p)},\mathcal{O}_{K}[b^{(2)}_{m_{2}},b^{(3)}_{m_{3}}]\jump{X_{1}})$. We define a subset $\mathfrak{X}_{m_{2},m_{3}}^{F}$ of arithmetic points of $\mathbf{I}_{1}$ to be
$$\mathfrak{X}_{m_{2},m_{3}}^{F}:=\left\{Q\in\mathfrak{X}_{\mathbf{I}_{1}}\ \middle|\ (Q,Q^{(2)}_{m_{2}},Q^{(3)}_{m_{3}})\in \mathfrak{X}_{R}^{F} \right\}.$$
For any $k\in\mathbb{Z}_{\geq 1}$, there exists an arithmetic point $Q\in\mathfrak{X}_{\mathbf{I}}$ with $k_{Q}=k$. Then, we have $\#\mathfrak{X}_{m_{2},m_{3}}^{F}=\infty$. Let $\mathbf{S}^{\mathrm{ord}}_{m_{2},m_{3}}\subset \mathcal{O}_{K}[b^{(2)}_{m_{2}},b^{(3)}_{m_{3}}]\jump{X_{1}}\jump{q}$ be an $\mathcal{O}_{K}[b^{(2)}_{m_{2}},b^{(3)}_{m_{3}}]\jump{X_{1}}$-module consisting of elements $f\in\mathcal{O}_{K}[b^{(2)}_{m_{2}},b^{(3)}_{m_{3}}]\jump{X_{1}}\jump{q}$ such that, for all but a finite number of $Q\in\mathfrak{X}_{m_{2},m_{3}}^{F}$, $f(Q)$ equals to a specialization of an element of $\mathbf{S}^{\mathrm{ord}}(N,\psi_{1,(p)}\overline{\psi}_{1}^{(p)},\mathcal{O}_{K}[b^{(2)}_{m_{2}},b^{(3)}_{m_{3}}]\jump{X_{1}})$ at $Q$. Then, we have $\mathbf{S}^{\mathrm{ord}}(N,\psi_{1,(p)}\overline{\psi}_{1}^{(p)},\mathcal{O}_{K}[b^{(2)}_{m_{2}},b^{(3)}_{m_{3}}]\jump{X_{1}})\subset \mathbf{S}^{\mathrm{ord}}_{m_{2},m_{3}}$ and $H^{(i)}_{m_{2},m_{3}}\in\mathbf{S}^{\mathrm{ord}}_{m_{2},m_{3}}$. It suffices to prove that we have $\mathbf{S}^{\mathrm{ord}}(N,\psi_{1,(p)}\overline{\psi}_{1}^{(p)},\mathcal{O}_{K}[b^{(2)}_{m_{2}},b^{(3)}_{m_{3}}]\jump{X_{1}})=\mathbf{S}^{\mathrm{ord}}_{m_{2},m_{3}}$. Let $g_{1},\ldots,g_{d}$ be elements of $\mathbf{S}^{\mathrm{ord}}_{m_{2},m_{3}}$ which are $\mathcal{O}_{K}[b^{(2)}_{m_{2}},b^{(3)}_{m_{3}}]\jump{X_{1}}$-linear independent. Then, there are $m_{1},\ldots,m_{d}\in\mathbb{Z}_{\geq1}$ such that 
$$d=\mathrm{det}(a(m_{i},g_{j}))_{1\leq i,j\leq d}\neq0\in\mathcal{O}_{K}[b^{(2)}_{m_{2}},b^{(3)}_{m_{3}}]\jump{X_{1}}.$$
Because $\#\mathfrak{X}_{m_{2},m_{3}}^{F}=\infty$, there exists an element $Q\in \mathfrak{X}_{m_{2},m_{3}}^{F}$ such that $d(Q)\neq0$. Then, we have
$$\mathrm{rank}_{\mathcal{O}_{K}[b^{(2)}_{m_{2}},b^{(3)}_{m_{3}}]\jump{X_{1}}}\mathbf{S}^{\mathrm{ord}}_{m_{2},m_{3}}=\mathrm{rank}_{\mathcal{O}_{K}[b^{(2)}_{m_{2}},b^{(3)}_{m_{3}}]\jump{X_{1}}}\mathbf{S}^{\mathrm{ord}}(N,\psi_{1,(p)}\overline{\psi}_{1}^{(p)},\mathcal{O}_{K}[b^{(2)}_{m_{2}},b^{(3)}_{m_{3}}]\jump{X_{1}}).$$
Then, if we take an element $f\in\mathbf{S}^{\mathrm{ord}}_{m_{2},m_{3}}$, there exists an element $a\in\mathcal{O}_{K}[b^{(2)}_{m_{2}},b^{(3)}_{m_{3}}]\jump{X_{1}}\backslash\{0\}$ such that $af\in\mathbf{S}^{\mathrm{ord}}(N,\psi_{1,(p)}\overline{\psi}_{1}^{(p)},\mathcal{O}_{K}[b^{(2)}_{m_{2}},b^{(3)}_{m_{3}}]\jump{X_{1}}).$ Because $a$ has only finite roots, we have $f\in \mathbf{S}^{\mathrm{ord}}(N,\psi_{1,(p)}\overline{\psi}_{1}^{(p)},\mathcal{O}_{K}[b^{(2)}_{m_{2}},b^{(3)}_{m_{3}}]\jump{X_{1}})$. Then, we have $\mathbf{S}^{\mathrm{ord}}(N,\psi_{1,(p)}\overline{\psi}_{1}^{(p)},\mathcal{O}_{K}[b^{(2)}_{m_{2}},b^{(3)}_{m_{3}}]\jump{X_{1}})=\mathbf{S}^{\mathrm{ord}}_{m_{2},m_{3}}$.

For each $m_{3}\in \mathbb{Z}_{\geq 1}$, let $H^{(i),m_{3}}\in \mathcal{O}_{K}[b^{(3)}_{m_{3}}]\jump{X_{1},X_{2}}$ be the specialization of $H^{(i)}$ at $Q^{(3)}_{m_{3}}$. Next, we prove $H^{(i),m_{3}}\in\mathbf{S}^{\mathrm{ord}}(N,\psi_{1,(p)}\overline{\psi}_{1}^{(p)},\mathcal{O}_{K}[b^{(3)}_{m_{3}}]\jump{X_{1}})\widehat{\otimes}_{\mathcal{O}_{K}[b^{(3)}_{m_{3}}]}\mathcal{O}_{K}[b^{(3)}_{m_{3}}]\jump{X_{2}}$.  We define an $\mathcal{O}_{K}[b^{(3)}_{m_{3}}]\jump{X_{1},X_{2}}$-module $\mathbf{S}^{\mathrm{ord}}_{m_{3}}\subset\mathcal{O}_{K}[b^{(3)}_{m_{3}}]\jump{X_{1},X_{2}}$ consisting of elements $f(X_{1},X_{2})$ such that $f(X_{1},b^{(2)}_{m})\in \mathbf{S}^{\mathrm{ord}}(N,\psi_{1,(p)}\overline{\psi}_{1}^{(p)},\mathcal{O}_{K}\jump{X_{1}})\otimes_{\mathcal{O}_{K}}\mathcal{O}_{\overline{\mathbb{Q}}_{p}}$ for each $m\in\mathbb{Z}_{\geq1}$. We have already proved that $H^{(i),m_{3}}\in \mathbf{S}^{\mathrm{ord}}_{m_{3}}$. It is clear that $\mathbf{S}^{\mathrm{ord}}(N,\psi_{1,(p)}\overline{\psi}_{1}^{(p)},\mathcal{O}_{K}[b^{(3)}_{m_{3}}]\jump{X_{1}})\widehat{\otimes}_{\mathcal{O}_{K}[b^{(3)}_{m_{3}}]}\mathcal{O}_{K}[b^{(3)}_{m_{3}}]\jump{X_{2}}\subset\mathbf{S}^{\mathrm{ord}}_{m_{3}}$. Further, if $g_{1},\ldots,g_{d}\in \mathbf{S}^{\mathrm{ord}}_{m_{3}}$ are linear independent, there exist $m_{1},\ldots,m_{d}\in\mathbb{Z}_{\geq1}$ such that 
$$d=\mathrm{det}(a(m_{i},g_{j}))_{1\leq i,j\leq d}\neq0\in\mathcal{O}_{K}[b^{(3)}_{m_{3}}]\jump{X_{1},X_{2}}.$$
Because there exists a $m_{2}\in\mathbb{Z}_{\geq 1}$ such that $d(X_{1},b^{(2)}_{m_{2}})\neq 0$, we have $\mathrm{rank}_{\mathcal{O}_{K}[b^{(3)}_{m_{3}}]\jump{X_{1},X_{2}}}\mathbf{S}^{\mathrm{ord}}_{m_{3}}=\mathrm{rank}_{\mathcal{O}_{K}[b^{(3)}_{m_{3}}]\jump{X_{1}}}\mathbf{S}^{\mathrm{ord}}(N,\psi_{1,(p)}\overline{\psi}_{1}^{(p)},\mathcal{O}_{K}[b^{(3)}_{m_{3}}]\jump{X_{1}}).$
Then, we take an $a\in \mathcal{O}_{K}[b^{(3)}_{m_{3}}]\jump{X_{1},X_{2}}\backslash \{0\}$ such that $aH^{(i),m_{3}}\in\mathbf{S}^{\mathrm{ord}}(N,\psi_{1,(p)}\overline{\psi}_{1}^{(p)},\mathcal{O}_{K}[b^{(3)}_{m_{3}}]\jump{X_{1}})\widehat{\otimes}_{\mathcal{O}_{K}[b^{(3)}_{m_{3}}]}\mathcal{O}_{K}[b^{(3)}_{m_{3}}]\jump{X_{2}}$. Because we have $a(X_{1},p^{m})\neq 0$ for almost all $m\in\mathbb{Z}_{\geq1}$, there exists a $k_{m_{3}}\in \mathbb{Z}_{\geq1}$ such that $H^{(i),m_{3}}(X_{1},p^{m^{\prime}})\in \mathbf{S}^{\mathrm{ord}}(N,\psi_{1,(p)}\overline{\psi}_{1}^{(p)},\mathcal{O}_{K}[b^{(3)}_{m_{3}}]\jump{X_{1}})$ for each $m^{\prime}\in\mathbb{Z}_{\geq k_{m_{3}}}$. 

We put $H^{(i),m_{3}}_{0}:=H^{(i),m_{3}}$ and $c_{m}=p^{k_{m_{3}}+m}$ for each $m\in\mathbb{Z}_{\geq0}$. We define a power series $H^{(i),m_{3}}_{m}\in\mathcal{O}_{K}[b^{(3)}_{m_{3}}]\jump{X_{1},X_{2}}\jump{q}$ inductively for $m\in\mathbb{Z}_{\geq 1}$ to be
$$H^{(i),m_{3}}_{m}(X_{1},X_{2}):=(H^{(i),m_{3}}_{m-1}(X_{1},X_{2})-H^{(i),m_{3}}_{m-1}(X_{1},c_{m}))(X_{2}-c_{m})^{-1}\in\mathcal{O}_{K}[b^{(3)}_{m_{3}}]\jump{X_{1},X_{2}}\jump{q}.$$
By the induction of $m\in\mathbb{Z}_{\geq 0}$, we have $H^{(i),m_{3}}_{m}(X_{1},c_{l})\in\mathbf{S}^{\mathrm{ord}}(N,\psi_{1,(p)}\overline{\psi}_{1}^{(p)},\mathcal{O}_{K}[b^{(3)}_{m_{3}}]\jump{X_{1}})$ for any $m\in\mathbb{Z}_{\geq 0}$ and $l\in\mathbb{Z}_{\geq m+1}$. In particular, if we put $H^{(i),m_{3}}_{m,m+1}:=H^{(i),m_{3}}_{m}(X_{1},c_{m+1})$, we have
$$H^{(i),m_{3}}=\displaystyle{\sum_{m=1}^{\infty}}H^{(i),m_{3}}_{m,m+1}\displaystyle{\prod_{j=1}^{m}}(X_{2}-c_{j})\in \mathbf{S}^{\mathrm{ord}}(N,\psi_{1,(p)}\overline{\psi}_{1}^{(p)},\mathcal{O}_{K}[b^{(3)}_{m_{3}}]\jump{X_{1}})\widehat{\otimes}_{\mathcal{O}_{K}[b^{(3)}_{m_{3}}]}\mathcal{O}_{K}[b^{(3)}_{m_{3}}]\jump{X_{2}}.$$

Next, we prove $H^{(i)}\in \mathbf{S}^{\mathrm{ord}}(N,\psi_{1,(p)}\overline{\psi}_{1}^{(p)},\mathcal{O}_{K}\jump{X_{1}})\widehat{\otimes}_{\mathcal{O}_{K}}\mathcal{O}_{K}\jump{X_{2},X_{3}}$. By the same way as above, we can take a non-zero element $a\in\mathcal{O}_{K}\jump{X_{1},X_{2},X_{3}}\backslash \{0\}$ such that $aH^{(i)}$ is an element of $\mathbf{S}^{\mathrm{ord}}(N,\psi_{1,(p)}\overline{\psi}_{1}^{(p)},\mathcal{O}_{K}\jump{X_{1}})\widehat{\otimes}_{\mathcal{O}_{K}}\mathcal{O}_{K}\jump{X_{2},X_{3}}$. Further, there exists a $k\in \mathbb{Z}_{\geq1}$ such that $H^{(i)}(X_{1},X_{2},p^{m})\in \mathbf{S}^{\mathrm{ord}}(N,\psi_{1,(p)}\overline{\psi}_{1}^{(p)},\mathcal{O}_{K}\jump{X_{1}})\widehat{\otimes}_{\mathcal{O}_{K}}\mathcal{O}_{K}\jump{X_{2}}$ for each $m\in\mathbb{Z}_{\geq k}$. We put $H^{(i)}_{0}:=H^{(i)}$ and $c^{\prime}_{m}=p^{k+m}$ for each $m\in\mathbb{Z}_{\geq 0}$. We define a power series $H^{(i)}_{m}\in\mathcal{O}_{K}\jump{X_{1},X_{2},X_{3}}\jump{q}$ inductively for $m\in\mathbb{Z}_{\geq 1}$ to be
$$H^{(i)}_{m}:=(H^{(i)}_{m-1}(X_{1},X_{2},X_{3})-H^{(i)}_{m-1}(X_{1},X_{2},c^{\prime}_{m}))(X_{3}-c^{\prime}_{m})^{-1}\in\mathcal{O}_{K}\jump{X_{1},X_{2},X_{3}}\jump{q}.$$
Then, we have
$$H^{(i)}=\displaystyle{\sum_{m=0}^{\infty}}H^{(i)}_{m}(X_{1},X_{2},c^{\prime}_{m+1})\displaystyle{\prod_{j=1}^{m}}(X_{3}-c^{\prime}_{j})\in \mathbf{S}^{\mathrm{ord}}(N,\psi_{1,(p)}\overline{\psi}_{1}^{(p)},\mathcal{O}_{K}\jump{X_{1}})\widehat{\otimes}_{\mathcal{O}_{K}}\mathcal{O}_{K}\jump{X_{2},X_{3}}.$$
We complete the proof.
\end{proof}
If necessary, we exchange $\mathbf{I}_{1}$ to a normal finite flat extension, and we assume that $\mathbf{I}_{1}$ contains the roots of the polynomials 
$$X^{2}-a(l,F)X+\psi_{1}(l)l^{-1}\langle l\rangle_{\mathbf{I}_{1}}\in\mathbf{I}_{1}[X]$$
for each $l\vert N$. We fix a root $\beta_{l}(F)\in\mathbf{I}_{1}^{\times}$ of the polynomial $X^{2}-a(l,F)X+\psi_{1}(l)l^{-1}\langle l\rangle_{\mathbf{I}_{1}}$. Let $\sum_{\Sigma_{1,0}^{\mathrm{I\hspace{-.1em}Ib}}}$ and $d_{1}$ be the subset of primes and the number defined in (3.1.1) and (3.1.2) respectively. As explained in \S4.1, $\sum_{\Sigma_{1,0}^{\mathrm{I\hspace{-.1em}Ib}}}$ and $d_{1}$ are independent of the specialization. We put
$$H^{\mathrm{aux}}:=\displaystyle{\sum_{I\subset \Sigma_{1,0}^{\mathrm{I\hspace{-.1em}Ib}}}}(-1)^{\mid I\mid}\frac{\psi_{1,(p)}(n_{I}\slash d_{1})\langle n_{I}\slash d_{1}\rangle_{\mathbf{I}_{1}}d_{1}}{\beta_{I}(F)n_{I}}\mathbf{U}_{d_{1}\slash n_{I}}(H^{\mathrm{ord}}).$$
Let $\mathrm{Tr}_{N\slash N_{1}}: \mathbf{S}^{\mathrm{ord}}(N,\psi_{1,(p)}\overline{\psi}_{1}^{(p)},\mathbf{I}_{1})\rightarrow \mathbf{S}^{\mathrm{ord}}(N_{1},\psi_{1,(p)}\overline{\psi}_{1}^{(p)},\mathbf{I}_{1})$ be the trace map defined in \cite[page14]{Hid88d}. Let $\breve{F}\in \mathbf{S}^{\mathrm{ord}}(N_{1},\psi_{(p)}\psi^{(p)},\mathbf{I}_{1})$ be the primitive Hida family attached to $F\vert[\overline{\psi^{(p)}}]$. 
\begin{Definition}
We define an element $L_{G^(2),G^{(3)}}^{F}\in R$ to be
$$L_{G^{(2)},G^{(3)}}^{F}:=a(1,\eta_{\breve{F}}1_{\breve{F}}\mathrm{Tr}_{N\slash N_{1}}(H^{\mathrm{aux}})).$$
Here, $1_{\breve{F}}$ is the idempotent element defined in (2.5.2) and $\eta_{\breve{F}}$ is the congruence number defined in Definition 2.5.5. 
\end{Definition}
\section{The interpolation formula.}
In \S5.1 and \S5.2, we prove Main Theorem. However, before we prove Main Theorem, we introduce a unite element $\mathfrak{f}_{F}\in R^{\times}$ which is called the fudge factor and constructed in \cite[Proposition 6.1.2]{Hsi17}.  In \S5.3, we give some examples of the triple ($\mathbf{I}_{i},\mathfrak{X}^{(i)},G^{(i)}$) which satisfy the axioms of Main Theorem. 

Throughout \S5, we denote by $(F,G^{(2)},G^{(3)})$ the same as in \S3. Let $\mathfrak{X}^{R}_{F}$ be the set defined in Definition 4.2.1. In \S5.1 and \S5.2, we assume that Hypotheses (1)`(7) holds.
\subsection{The fudge factor.}
In this subsection, we construct the fudge factor defined in \cite[Proposition 6.1.2]{Hsi17}. Then, we have to check the result of \cite[Proposition 6.1.2]{Hsi17} can be generalized to the case when $G^{(i)}$ is not Hida family for $i=2,3$.
\begin{Lemma}
For $i=2,3$and each $l\vert N_{i}$, we have $a(l,G^{(i)})\in \mathbf{I}_{i}^{\times}$.
\end{Lemma}
\begin{proof}
We fix an element $i\in\{2,3\}$. We put $m_{i}:=m_{\psi_{i}}$. First, we assume $v_{l}(m_{i})=v_{l}(N_{i})$. By \cite[Theorem 4.6.17]{Miy06}, for each $m\in\mathbb{Z}_{\geq 1}$, we have
$$a(l,G^{(i)}(m))\overline{a(l,G^{(i)}(m))}=l^{\frac{1}{2}(k^{(i)}(m)-1)}.$$
Then, we have $\vert a(n,G^{(i)}(m))\vert_{p}=1$ and hence $a(l,G^{(i)})\in\mathbf{I}_{i}^{\times}$.

Next we assume $v_{l}(m_{i})\neq v_{l}(N_{i})$. We checked in the proof of Proposition 4.1.1 that if $v_{l}(m_{i})\neq v_{l}(N_{i})$, we have $v_{l}(m_{i})=0$ and $v_{l}(N_{i})=1$. Then, by \cite[Theorem 4.6.17]{Miy06}, we have $\vert a(l,G^{(i)}(m))\vert_{p}=1$ for each $m\in\mathbb{Z}_{\geq 1}$. Hence, we have $a(l,G^{(i)})\in\mathbf{I}_{i}^{\times}$. We complete the proof.
\end{proof}
\begin{Lemma}
Let $l\vert N$ be a prime. 
\begin{enumerate}
\item[(1)] For $i=2,3$, there exists a unit $\epsilon_{l}^{i}\in\mathbf{I}_{i}^{\times}$ such that
$$\epsilon_{l}^{i}(Q^{(i)}_{m})=\epsilon(1\slash2,\pi_{G^{(i)}(m),l})\vert N_{i}\vert_{l}^{-\frac{{k}^{(i)}(m)}{2}}$$
for each $m\in\mathbb{Z}_{\geq 1}$.\\
\item[(2)] There exists a unit $\epsilon_{l}^{2,3}\in (\mathbf{I}_{2}\widehat{\otimes}_{\mathcal{O}_{K}}\mathbf{I}_{3})^{\times}$ such that
$$\epsilon^{2,3}_{l}(Q^{(2)}_{m_{2}},Q^{(3)}_{m_{3}})=\epsilon(r_{m_{2},m_{3}},\pi_{G^{(2)}(m_{2}),l}\otimes\pi_{G^{(3)}(m_{3}),l}),$$
where $r_{m_{2},m_{3}}=\frac{1}{2}(2-k^{(2)}(m_{2})-k^{(3)}(m_{3}))$ for each $m_{2},m_{3}\in\mathbb{Z}_{\geq 1}$.
\end{enumerate}
\end{Lemma}
\begin{proof}
First, we prove (1). We put $g_{i,m}:=G^{(i)}(m)$, $k_{i,m}:=k^{(i)}(m)$ and $\epsilon_{i,m}:=\epsilon^{(i)}_{m}$ for each $m\in\mathbb{Z}_{\geq 1}$. By Proposition 4.1.1, $\pi_{g_{i,m},l}$ is a principal series or a special representation. Further, the automorphic type of $\pi_{g_{i,m},l}$ is independent for each $m\in\mathbb{Z}_{\geq 1}$. We divide the proof of the construction of $\epsilon_{l}^{i}\in\mathbf{I}_{i}^{\times}$ into the following two cases. 
\begin{enumerate}
\item[$(1)^{\prime}$] The representation $\pi_{g_{i,m},l}$ is a principal series.  
\item[$(2)^{\prime}$] The representation $\pi_{g_{i,m},l}$ is a special representation.
\end{enumerate}
First, we prove in the case $(1)^{\prime}$. By Proposition 4.1.1, the representation $\pi_{g_{i,m},l}$ is isomorphic to the principal series $\pi(\alpha^{(i)}_{m,l},{\alpha^{(i)}_{m,l}}^{-1}\omega^{(i)}_{m,l})$, where $\alpha^{(i)}_{m,l}$ is an unramified character with $\alpha^{(i)}_{m,l}(l)=a(l,g_{i,m})l^{\frac{1}{2}(1-k_{i,m})}$ and $\omega^{(i)}_{m,l}$ is the central character of $\pi_{g_{i,m},l}$. Then, by \cite[Proposition 3.5]{JL70}, it is known that 
$$\epsilon(s,\pi_{g_{i,m},l})=\epsilon(s,\alpha^{(i)}_{m,l})\epsilon(s,{\alpha^{(i)}_{m,l}}^{-1}\omega^{(i)}_{m,l}).$$
In particular, because $\alpha^{(i)}_{m,l}$ is unramified, we have $\epsilon(s,\pi_{g_{i,m},l})=\epsilon(s,{\alpha^{(i)}_{m,l}}^{-1}\omega^{(i)}_{m,l})$. By Proposition 4.1.1 and the equation \cite[(3.2.6.3)]{Tat79}, we have
$$\epsilon(1\slash2,{\alpha^{(i)}_{m,l}}^{-1}\omega^{(i)}_{m,l})=\epsilon(1\slash2,(\psi_{i,(p)})_{\mathbb{A},l}^{-1})\cdot{\alpha^{(i)}_{m,l}}^{-1}(l^{v_{l}(N_{i})})\cdot(\omega_{p}^{-k_{i,m}}\epsilon_{i,m}\psi_{i}^{(p)})_{\mathbb{A},l}^{-1}(l^{v_{l}(N_{i})}).$$
Obviously, the function  $\epsilon(s,(\psi_{i,(p)})_{\mathbb{A},l}^{-1})$ is independent of $m\in\mathbb{Z}_{\geq 1}$ and by the definitions of the characters, we have 
$${\alpha^{(i)}_{m,l}}^{-1}(l^{v_{l}(N_{i})})\cdot(\omega_{p}^{-k_{i,m}}\epsilon_{i,m}\psi_{i}^{(p)})_{\mathbb{A},l}^{-1}(l^{v_{l}(N_{i})})=a(l,g_{i,m})^{-v_{l}(N_{i})}l^{\frac{v_{l}(N_{i})}{2}(k_{i,m}-1)}\omega_{p}^{-k_{i,m}}\epsilon_{m}\psi_{i}^{(p)}(l^{v_{l}(N_{i})}).$$
By Lemma 5.1.1, $a(l,G^{(i)})\in\mathbf{I}_{i}^{\times}$, then we can interpolate $\epsilon(1\slash2, \pi_{g_{m},l})\vert N_{i}\vert_{l}^{-\frac{k_{m}}{2}}$ by
$$a(l,G^{(i)})^{-v_{l}(N_{i})}{\langle l\rangle^{(i)}}^{v_{l}(N_{i})}\epsilon_{l}(1\slash2,(\psi_{i,(p)})_{\mathbb{A},l}^{-1})l^{-\frac{v_{l}(N_{i})}{2}}\psi_{i}^{(p)}(l^{v_{l}(N_{i})})\in \mathbf{I}_{i}^{\times}.$$

Next, we prove in the case $(2)^{\prime}$. By Proposition 4.1.1, the representation $\pi_{g_{i,m},l}$ is isomorphic to the special representation $\alpha_{m,l}^{(i)}\vert\cdot\vert_{l}^{-\frac{1}{2}}\mathrm{St}$ and $v_{l}(N_{i})=1$. By \cite[Proposition 3.6]{JL70}, it is known that
$$\epsilon(s,\pi_{g_{i,m},l})=\epsilon(s,\alpha_{m,l}^{(i)})\epsilon(s,\alpha_{m,l}^{(i)}\vert\cdot\vert_{l}^{-1})\frac{L(1-s,{\alpha_{m,l}^{(i)}}^{-1})}{L(s,\alpha_{m,l}^{(i)}\vert\cdot\vert_{l}^{-1})}.$$
Then, we have
\begin{align*}
\epsilon(1\slash2,\pi_{g_{i,m},l})&=\frac{L(1\slash2,{\alpha^{(i)}_{m,l}}^{-1})}{L(1\slash2,\alpha_{m,l}^{(i)}\vert\cdot\vert_{l}^{-1})}\\
&=-\alpha_{m,l}^{(i)}\vert\cdot\vert_{l}^{-\frac{1}{2}}(l).
\end{align*}
Then, we have
$$\vert N_{i}\vert_{l}^{-\frac{k_{i,m}}{2}}\epsilon(1\slash2,\pi_{g_{i,m},l})=-a(l,g_{i,m})l.$$
We can interpolate $\vert N_{l}\vert_{l}^{-\frac{k_{i,m}}{2}}\epsilon(1\slash2,\pi_{g_{i,m},l})$ by
$$-a(l,G^{(i)})l\in\mathbf{I}_{i}^{\times}.$$ We complete the proof of (1).

Next, we prove (2). We divide the construction into the following two cases.
\begin{enumerate}
\item[$(1)^{\prime\prime}$] At least one of $\pi_{g_{2,m_{2}},l}$ and $\pi_{g_{3,m_{3}},l}$ is a principal series.
\item[$(2)^{\prime\prime}$] Both of $\pi_{g_{2,m_{2}},l}$ and $\pi_{g_{3,m_{3}},l}$ are special representations.
\end{enumerate}
By Proposition 4.1.1, if the pair $(\pi_{g_{2,m_{2}},l},\pi_{g_{3,m_{3}},l})$ satisfies the condition $(1)^{\prime\prime}$ (resp. $(2)^{\prime\prime}$) at $m_{2},m_{3}\in\mathbb{Z}_{\geq 1}$, then the pair $(\pi_{g_{2,m_{2}},l},\pi_{g_{3,m_{3}},l})$ satisfies the condition $(1)^{\prime\prime}$ (resp. $(2)^{\prime\prime}$) at any $m_{2},m_{3}\in\mathbb{Z}_{\geq 1}$.

First, we prove in the case $(1)^{\prime\prime}$. By symmetry, we assume that $\pi_{g_{3,m_{3}},l}$ is a principal series. By \cite[Proposition (1.4)]{GJ78}, we have
$$\epsilon(s,\pi_{g_{2,m_{2}},l}\otimes\pi_{g_{3,m_{3}},l})=\epsilon(s,\pi_{g_{2,m_{2}},l}\otimes\alpha^{(3)}_{m_{3},l})\epsilon(s,\pi_{g_{2,m_{2}},l}\otimes{\alpha^{(3)}_{m_{3},l}}^{-1}\omega^{(3)}_{m_{3},l}).$$
Then, it suffices to interpolate $\epsilon(r_{m_{2},m_{3}},\pi_{g_{2,m_{2}},l}\otimes\alpha^{(3)}_{m_{3},l})$ and $\epsilon(r_{m_{2},m_{3}},\pi_{g_{2,m_{2}},l}\otimes{\alpha^{(3)}_{m_{3},l}}^{-1}\omega^{(3)}_{m_{3},l})$.
Hence, we can construct $\epsilon_{l}^{2,3}\in (\mathbf{I}_{2}\widehat{\otimes}_{\mathcal{O}_{K}}\mathbf{I}_{3})^{\times}$ by the same way as the construction of $\epsilon_{l}^{i}\in\mathbf{I}_{i}^{\times}$. 

Next, we prove in the case $(2)^{\prime\prime}$. By \cite[Proposition (1.4)]{GJ78}, we have
$$\epsilon_{l}(s,\pi_{g_{2,m_{2}},l}\otimes\pi_{g_{3,m_{3}},l})=\epsilon_{l}(s,\pi_{g_{2,m_{2}},l}\otimes\alpha^{(3)}_{m_{3},l})\epsilon_{l}(s,\pi_{g_{2,m_{2}},l}\otimes{\alpha^{(3)}_{m_{3},l}}\vert\cdot\vert_{l}^{-1}).$$
Then, it suffices to interpolate $\epsilon_{l}(r_{m_{2},m_{3}},\pi_{g_{2,m_{2}},l}\otimes\alpha^{(3)}_{m_{3},l})$ and $\epsilon_{l}(r_{m_{2},m_{3}},\pi_{g_{2,m_{2}},l}\otimes{\alpha^{(3)}_{m_{3},l}}\vert\cdot\vert_{l}^{-1})$. Hence, we can construct $\epsilon_{l}^{2,3}\in (\mathbf{I}_{2}\widehat{\otimes}_{\mathcal{O}_{K}}\mathbf{I}_{3})^{\times}$ by the same way as the construction of $\epsilon_{l}^{i}\in\mathbf{I}_{i}^{\times}$.
\end{proof}
\begin{Remark}
Let $l\vert N$ be a finite prime. For $i=2,3$, by the same way as the proof of (2) of Lemma 5.1.2, we can construct a unite $\epsilon_{l}^{1,i}\in (\mathbf{I}_{1}\widehat{\otimes}_{\mathcal{O}_{K}}\mathbf{I}_{i})^{\times}$ such that
$$\epsilon^{1,i}_{l}(Q_{1},Q^{(i)}_{m_{i}})=\epsilon(\frac{1}{2}(2-k_{Q_{1}}-k^{(i)}(m_{i})),\pi_{F_{Q_{1}},l}\otimes\pi_{G^{(i)}(m_{i}),l})$$
for each $(Q_{1},Q_{m_{i}}^{(i)})\in\mathfrak{X}_{\mathbf{I}_{1}}\times\mathfrak{X}^{(i)}$.
\end{Remark}
For each $l\vert N$, let $\mathfrak{I}_{\Pi_{\underline{Q},q}}$ be the normalized zeta integral defined in Definition 3.3.3 and $\chi_{\underline{Q}}$ be the Dirichlet character defined in (3.2.1). We put $(\pi^{\prime}_{1},\pi^{\prime}_{2},\pi^{\prime}_{3})=(\pi_{F_{Q_{1}}}\otimes(\chi_{\underline{Q}})_{\mathbb{A}},\pi_{G^{(2)}(m_{2})},\pi_{G^{(3)}(m_{3})})$. The following proposition is proved in \cite[Proposition 6.12]{Hsi17}.
\begin{Proposition}
For each $q\vert N$, there exists a unique element $\mathfrak{f}_{q}\in R^{\times}$ such that
$$\mathfrak{f}_{q}(\underline{Q})=\mathfrak{I}_{\Pi_{\underline{Q},q}}$$
for any $\underline{Q}\in\mathfrak{X}_{R}^{F}$.
\end{Proposition}

\begin{Remark}
Lemma 5.1.2 is corresponded to \cite[Lemma 6.11]{Hsi17}. Further, we note that the triple $(\pi^{\prime}_{1,q},\pi^{\prime}_{2,q},\pi^{\prime}_{3,q})$ satisfies Hypothesis 3.4.2 because all representations $\pi^{\prime}_{1,q},\pi^{\prime}_{2,q}$ and $\pi^{\prime}_{3,q}$ are minimal and not supercuspidal by Proposition 4.1.1. By Lemma 5.1.2 and Hypothesis 3.4.2, we can adopt the proof of \cite[Proposition 6.12]{Hsi17} in our case. 
\end{Remark}
\subsection{Proof of the main results}
In this subsection, we prove Main Theorem. 
\begin{Theorem}
Let us assume the hypotheses (1)`(7). Then, there exists a unique element $\mathcal{L}^{F}_{G^{(2)},G^{(3)}}\in R$ such that we have the interpolation property :
$$(\mathcal{L}^{F}_{G^{(2)},G^{(3)}}(\underline{Q}))^{2}=\mathcal{E}_{F_{Q_{1}}}(\Pi_{\underline{Q},p})\cdot\frac{L(\frac{1}{2},\Pi_{\underline{Q}})}{(\sqrt{-1})^{2k_{Q_{1}}}\Omega_{F_{Q_{1}}}^{2}}\leqno(\mathrm{MT})$$
for every $\underline{Q}=(Q_{1},Q^{(2)}_{m_{2}},Q^{(3)}_{m_{3}})\in \mathfrak{X}^{F}_{R}.$
\end{Theorem}
\begin{proof}
Let $L^{F}_{G^{(2)},G^{(3)}}$ be the element of $R$ defined in Definition 4.2.5. By Corollary 3.3.5 and Proposition 3.4.7, we have
$$(L^{F}_{G^{(2)},G^{(3)}}(\underline{Q}))^{2}=\frac{\psi_{1,(p)}(-1)(-1)^{k_{1}+1}L(1\slash2,\Pi_{\underline{Q}})}{\Omega_{F_{Q_{1}}}^{2}}\cdot\mathcal{E}_{F_{Q_{1}}}(\Pi_{\underline{Q},p})\cdot\displaystyle{\prod_{q\vert N}}\mathfrak{I}_{\Pi_{\underline{Q},q}}.$$
If we put $\mathcal{L}^{F}_{G^{(2)},G^{(3)}}=(-\psi_{1,(p)}(-1))^{-\frac{1}{2}}L^{F}_{G^{(2)},G^{(3)}}\cdot\displaystyle{\prod_{q\vert N}}\mathfrak{f}_{q}^{-\frac{1}{2}}$, by Proposition 5.1.4, we have
$$(\mathcal{L}^{F}_{G^{(2)},G^{(3)}}(\underline{Q}))^{2}=\mathcal{E}_{F_{Q_{1}}}(\Pi_{\underline{Q},p})\cdot\frac{L(\frac{1}{2},\Pi_{\underline{Q}})}{(\sqrt{-1})^{2k_{Q_{1}}}\Omega_{F_{Q_{1}}}^{2}}.$$
\end{proof}
We complete the proof.
\subsection{Examples.} 
In this subsection, we give examples of the triple ($\mathbf{I}_{i},\mathfrak{X}^{(i)},G^{(i)}$) which satisfy Hypothesis (5), (6) and (7). As a first example, we can take families of CM forms of weight 1. Let $L$ be a quadratic imaginary extension of $\mathbb{Q}$ with a discriminant $D$. We assume that $D$ is square-free and prime to $p$. Let $\mathfrak{f}$ be an integral ideal of $\mathcal{O}_{L}$ such that $\mathfrak{f}$ is prime to $Dp$. We assume that $\mathrm{N}(\mathfrak{f})$ is square-free, where $\mathrm{N}$ is the absolute norm. Let $\mathfrak{C}(\mathbf{f}(p)^{j})$ be the class ray group modulo $\mathbf{f}(p)^{j}$ over $L$ for each $j\geq0$. By the class field theory, $\mathfrak{C}(\mathfrak{f}(p)^{\infty})=\plim[j\geq0]\mathfrak{C}(\mathfrak{f}(p)^{j})$ is a $\mathbb{Z}_{p}$-module of rank 2. Let $\Delta_{\mathfrak{f}}$ be the torsion part of $\mathfrak{C}(\mathfrak{f}(p)^{\infty})$ and $\chi:\Delta_{\mathfrak{f}}\rightarrow \mathbb{C}^{\times}$ be a primitive character.  Here, a primitive character means that it is not induced by any character from $\Delta_{\mathfrak{f}^{\prime}}$ for $\mathfrak{f}\subsetneq\mathfrak{f}^{\prime}$. Let $L^{-}_{\infty}\slash L$ be the anticyclotomic extension of $L$. By the class field theory, the Galois group $\mathrm{Gal}(L^{-}_{\infty}\slash L)$ is attached to a direct summand of the $\mathbb{Z}_{p}$-torsion free part of $\mathfrak{C}(\mathfrak{f}(p)^{\infty})$. Let $\mathrm{pr}_{\mathfrak{f}}:\mathfrak{C}(\mathfrak{f}(p)^{\infty})\rightarrow\Delta_{\mathfrak{f}}$ and $\mathrm{pr}_{-}:\mathfrak{C}(\mathfrak{f}(p)^{\infty})\rightarrow\mathrm{Gal}(L^{-}_{\infty}\slash L)$ be the natural projections to $\Delta_{\mathfrak{f}}$ and $\mathrm{Gal}(L^{-}_{\infty}\slash L)$ respectively. Let $E$ be a finite Galois extension of $\mathbb{Q}_{p}$ such that the image of $\Delta_{\mathfrak{f}}$ by $\chi$ is contained in $E$. We define a group homomorphism
$$\Psi:\mathfrak{C}(\mathfrak{f}(p)^{\infty})\rightarrow \mathcal{O}_{E}\jump{\mathrm{Gal}(L^{-}_{\infty}\slash L)}^{\times}$$
by $\Psi(a)=\chi(\mathrm{pr}_{\mathfrak{f}}(a))[\mathrm{pr}_{-}(a)]$ for $a\in\mathfrak{C}(\mathfrak{f}(p)^{\infty})$. Let $\mathrm{J}_{\mathfrak{f}(p)}$ be the group which consists of fractional ideals $\mathfrak{a}$ of $L$ which is prime to $\mathfrak{f}(p)$. For each finite prime ideal $\mathfrak{l}$, we denote by $L_{\mathfrak{l}}$ the completion of $L$ by $\mathfrak{l}$. Let $\mathcal{O}_{L_{\mathfrak{l}}}$ be the integers of $L_{\mathfrak{l}}$ and $\pi_{\mathfrak{l}}$ be a generator of the maximal ideal of $\mathcal{O}_{L_{\mathfrak{l}}}$. We define a group homomorphism 
$$\Psi^{*}:\mathrm{J}_{\mathfrak{f}(p)}\rightarrow R^{\times}$$
to be $\Psi^{*}(\mathfrak{a})=\displaystyle{\prod_{l\nmid\mathfrak{f}(p)}}\Psi_{\mathfrak{l}}(\pi_{\mathfrak{l}}^{n_{\mathfrak{l}}})$, where $\mathfrak{a}=\displaystyle{\prod_{\mathfrak{l}\nmid\mathfrak{f}(p)}}\mathfrak{l}^{n_{\mathfrak{l}}}$. We put
$$F_{\Psi}=\displaystyle{\sum_{\mathfrak{a}\nmid\mathfrak{f}(p)}}\Psi^{*}(\mathfrak{a})q^{\mathrm{N}(\mathfrak{a})},$$
where $\mathfrak{a}$ runs through integral ideals of $L$ which are prime to $\mathfrak{f}(p)$. Let $\epsilon:\mathrm{Gal}(L^{-}_{\infty}\slash L)\rightarrow\overline{\mathbb{Q}}^{\times}$ be a finite character, and we denote by $P_{\epsilon}:\mathcal{O}_{E}\jump{\mathrm{Gal}(L^{-}_{\infty}\slash L)}\rightarrow\overline{\mathbb{Q}_{p}}$ the $\mathcal{O}_{E}$-algebra homomorphism defined by $P_{\epsilon}([w])=\epsilon(w)$ for $w\in\mathrm{Gal}(L^{-}_{\infty}\slash L)$. It is known that for each finite character $\epsilon: \mathrm{Gal}(L^{-}_{\infty}\slash L)\rightarrow\overline{\mathbb{Q}}^{\times}$, the series $f_{\epsilon}:=P_{\epsilon}(F_{\Psi})\in P_{\epsilon}(\mathcal{O}_{E}\jump{\mathrm{Gal}(L^{-}_{\infty}\slash L)})\jump{q}$ is a Fourier expansion of a classical modular form of weight 1 and level $(-D)\mathrm{N}(\mathbf{f})p^{e_{\epsilon}}$ with $e_{\epsilon}\in \mathbb{Z}_{\geq 1}$ ($cf.$ \cite[Theorem 4.8.2]{Miy06}) and by the definition it is a CM-form. We remark that the $p$-th coefficient $a(p,F_{\Psi})\in \mathcal{O}_{E}\jump{\mathrm{Gal}(L^{-}_{\infty}\slash L)}$ of $F_{\Psi}$ is zero by the definition. However, if $\epsilon:\mathrm{Gal}(L^{-}_{\infty}\slash L)\rightarrow\overline{\mathbb{Q}}^{\times}$ is primitive and the conductor is sufficiently large, it is known that $f_{\epsilon}$ is a primitive form ($cf$. \cite[Theorem 4.8.2]{Miy06}). Then, if we put $\mathfrak{X}:=\{\mathrm{Ker}P_{\epsilon}\mid f_{\epsilon}\ \mathrm{is\ primitive}\}$, the cardinality of $\mathfrak{X}$ is not finite, and the triple $(\mathcal{O}_{E}\jump{\mathrm{Gal}(L^{-}_{\infty}\slash L)},\mathfrak{X},F_{\Psi})$ satisfies the condition (6). Further, it is not difficult to prove that the triple $(\mathcal{O}_{E}\jump{\mathrm{Gal}(L^{-}_{\infty}\slash L)},\mathfrak{X},F_{\Psi})$ satisfies the condition (5). Let $\mathrm{pr}_{\mathbb{A}^{\times}}:\mathbb{A}^{\times}\rightarrow \mathfrak{C}(\mathfrak{f}(p)^{\infty})$ be the natural projection defined by the class field theory. If we put $\langle n\rangle=n\omega_{p}(n)^{-1}\Psi([\mathrm{pr}_{\mathbb{A}^{\times}}\circ j_{p}(n\omega_{p}(n)^{-1})])^{-1}\in R$ for each positive integer $n\in\mathbb{Z}_{\geq 1}$ which is prime to $p$, $\langle n\rangle$ satisfies the condition of (5). Because $D\mathrm{N}(\mathbf{f})$ is square-free, by \cite[Theorem 4.6.17]{Miy06}, $F_{\Psi}$ satisfies Hypothesis (7).

As a second example of ($\mathbf{I}_{i},\mathfrak{X}^{(i)},G^{(i)}$), we give Coleman families. For an element $x\in K$ and $\epsilon\in p^{\mathbb{Q}}$, we denote by $\mathcal{B}[x,\epsilon]_{K}$ the closed ball of radius $\epsilon$ and center $x$, seen as a $K$-affinoid space. We denote by $\mathcal{A}_{\mathcal{B}[x,\epsilon]_{K}}$ the ring of  analytic functions on $\mathcal{B}[x,\epsilon]_{K}$ and by $\mathcal{A}_{\mathcal{B}[x,\epsilon]_{K}}^{0}$ the subring of power bounded elements of $\mathcal{A}_{\mathcal{B}[x,\epsilon]_{K}}$. We remark that if $\epsilon\in K$, the ring $\mathcal{A}_{\mathcal{B}[x,\epsilon]_{K}}^{0}$ is isomorphic to the ring
$$\mathcal{O}_{K}\langle \epsilon^{-1}(T-x)\rangle=\left\{\sum_{n\geq0} a_{n}\left(\epsilon^{-1}(T-x)\right)^{n}\in \mathcal{O}_{K}\jump{\epsilon^{-1}(T-x)}\middle|\lim_{n\to\infty}\vert a_{n}\vert_{p}=0\right\}.$$
Let $M$ be a positive integer which is prime to $p$ and square-free. Let $\epsilon_{M}$ be a Dirichlet character mod $M$. Let $f$ be a $p$-stabilized newform of weight $k_{0}$, level $Mp$ , slope $\alpha<k_{0}-1$ and Nebentypus $\epsilon_{M}\omega_{p}^{i-k_{0}}$ where $0\leq i\leq p-1$. Further, we assume $a(p,f)^{2}\neq\epsilon_{M}(p)p^{k_{0}-1}$ if $i=0$. Then, by Coleman in \cite{Col97}, there exists a radius $\epsilon\in p^{\mathbb{Q}}\cap K$ and a series
$$G\in\mathcal{A}_{\mathcal{B}[k_{0},\epsilon]_{K}}^{0}\jump{q}$$
such that for any $k\in \mathcal{B}[k_{0},\epsilon]_{K}(K)\cap\mathbb{Z}_{>\alpha+1}$ the specialization $G(k)$ of $G$ at $k$ is a Fourier expansion of a normalized Hecke eigenform of weight $k$, level $Mp$, slope $\alpha$ and Nebentypus $\epsilon_{M}\omega_{p}^{i-k}$. Further, we prove in A.2.7 that we can take a sufficiently small $\epsilon$ such that $G(k)$ is a $p$-stabilized newform for any $k\in \mathcal{B}[k_{0},\epsilon]_{K}(K)\cap\mathbb{Z}_{>\alpha+1}$.  If we put $X=\epsilon^{-1}(T-k_{0})$, then by the definition, we can regard the Coleman series $G$ as a series $G(X)$ in $\mathcal{O}_{K}\jump{X}$. Further, if we put $b_{k}=\epsilon^{-1}(k-k_{0})$ for any $k\in\mathbb{Z}_{>\alpha+1}\cap\mathcal{B}[k_{0},\epsilon]_{K}(K)$, $G(b_{k})$ is a Fourier expansion of a $p$-stabilized newform of weight $k$, level $Mp$, slope $\alpha$ and Nebentypus $\epsilon_{M}\omega_{p}^{i-k}$. We denote by $P_{k}:\mathcal{O}_{K}\jump{X}\rightarrow K$ the continuous $\mathcal{O}_{K}$-algebra homomorphism defined by $P_{k}(X)=b_{k}$ for all $k\in\mathbb{Z}_{>\alpha+1}\cap\mathcal{B}[k_{0},\epsilon]_{K}(K)$. Then, the triple $(\mathcal{O}_{K}\jump{X},\mathfrak{X},G(X))$ satisfies Hypothesis (6). We check that the triple $(R,\mathfrak{X},G(X))$ satisfies Hypothesis (5). Let $\mathrm{exp}(x)$ and $\mathrm{log}(x)$ be the formal exponential series and log series in $K\jump{x}$ defined by
\begin{align*}
\mathrm{exp}(x)=&\displaystyle{\sum_{n\geq 0}}\frac{1}{n!}x^{n},\\
\mathrm{log}(x)=&\displaystyle{\sum_{n\geq1}}\frac{(-1)^{n-1}}{n!}x^{n}.
\end{align*}
Let $\langle\ \rangle_{\Lambda}:\mathbb{Z}_{p}^{\times}\rightarrow\Lambda^{\times}$ be the group homomorphism defined in (2.5.1). We fix an isomorphism $\Lambda\cong\mathcal{O}_{K}\jump{X}$ defined by $[1+p]\mapsto X+1$ and we define a formal series
$$\langle n\rangle^{\prime}:=\langle n\rangle_{\Lambda}((1+p)^{k_{0}}\mathrm{exp}(\epsilon X\mathrm{log}(1+p))-1)$$
for each $n\in \mathbb{Z}_{\geq 1}$ which is prime to $p$. We remark that because we have $\vert p^{m}\vert_{p}\leq\vert m!\vert_{p}$ for each $m\in\mathbb{Z}_{\geq 1}$, the series $\langle n\rangle^{\prime}$ is contained in $\mathcal{O}_{K}\jump{X}$. Further, for each $n\in\mathbb{Z}_{\geq 1}$ which is prime to $p$, the series $\langle n\rangle^{\prime}$ satisfies the condition of Hypothesis (5). Because $M$ is square-free, by \cite[Theorem 4.6.17]{Miy06}, $G(X)$ satisfies Hypothesis (7).
\appendix
\section{Rigidity of conductors in a Coleman family.}
In this appendix, we prove the rigidity of conductors in a Coleman family. We assume $p\geq 5$. For any $K$-affinoid space $\mathfrak{X}$, we denote by $A(\mathfrak{X})$ the $K$-affinoid algebra attached to $\mathfrak{X}$.
\subsection{Notation on overconvergent modular forms.}
Let $M$ be a positive integer which is prime to $p$. Further, we assume $M\geq 5$. In this subsection, we recall some of notation on overconvergent modular forms and Coleman families of level $Mp$. We denote by $X(\Gamma_{1}(Mp))$ (resp. $X(M;p)$) the modular curve over $K$ which represents the equivalent classes of generalized elliptic curves with $\Gamma_{1}(Mp)$-structure (resp. $\Gamma_{1}(M)\cap \Gamma_{0}(p)$-structure). We denote by $f: E_{1}(Mp)\rightarrow X(\Gamma_{1}(Mp))$ the universal generalized elliptic curve. We put $\omega:=f_{*}(\Omega_{E_{1}(Mp)\slash X(\Gamma_{1}(Mp))})$. We define the normalized Eisenstein series of weight $p-1$ and level 1 to be
$$E_{p-1}=1-\frac{2(p-1)}{B_{p-1}}\displaystyle{\sum_{n\geq 1}}\sigma_{p-2}(n)q^{n},\leqno(\mathrm{A}.1.1)$$
where $B_{p-1}$ is the $p-1$-th Bernoulli number and $\sigma_{p-2}(n):=\sum_{d\vert n}d^{p-2}$. It is known that $E_{p-1}$ is a lifting of the Hasse invariant and the $q$-expansion of $E_{p-1}$ at $\infty$ is invertible in $\mathbb{Z}_{p}\jump{q}$. Because $E_{p-1}$ is a section of $H^{0}(X(\Gamma_{1}(Mp)),\omega^{\otimes p-1})$, the $p$-adic norm $\vert E_{p-1}(x)\vert_{p}$ of $E_{p-1}$ at each closed point $x\in X(\Gamma_{1}(Mp))$ makes sense by \cite[\S2]{Col96}. Let $\mathfrak{X}(\Gamma_{1}(Mp))$ be the $K$-rigid space associated to the scheme $X(\Gamma_{1}(Mp))$. For each $v\in \mathbb{Q}\cap [0,p^{p(p+1)^{-1}})$, we define a sub-affinoid space of $\mathfrak{X}(\Gamma_{1}(Mp))$ to be
$$\mathfrak{X}(\Gamma_{1}(Mp))(v):=\{x\in \mathfrak{X}(\Gamma_{1}(Mp))\mid \vert E_{p-1}(x)\vert_{p}\geq \vert p\vert_{p}^{v} \}.$$
We also define the $K$-affinoid spaces $\mathfrak{X}(\Gamma_{1}(M))(v)$ and $\mathfrak{X}(\Gamma_{1}(M)\cap\Gamma_{0}(p))(v)$ in a similar way. As explained in \cite[\S6]{Col96}, we can regard $\mathfrak{X}(\Gamma_{1}(M))(v)$ as a subspace of $\mathfrak{X}(\Gamma_{1}(M)\cap\Gamma_{0}(p))(v)$. Let $i_{v}: \mathfrak{X}(\Gamma_{1}(Mp))(v)\rightarrow \mathfrak{X}(\Gamma_{1}(M)\cap \Gamma_{0}(p))(v)$ be the natural forgetful map and we denote by $\tilde{\mathfrak{X}}(\Gamma_{1}(Mp))(v)$ the inverse image of $\mathfrak{X}(\Gamma_{1}(M))(v)$ under $i_{v}$. Let $f_{v}:\tilde{E}_{1}(Mp,v)\rightarrow \tilde{\mathfrak{X}}(\Gamma_{1}(Mp))(v)$ be the pullback of $E_{1}(Mp)$ to $\tilde{\mathfrak{X}}(\Gamma_{1}(Mp))(v)$. We denote by $\tilde{\omega}_{v}$ the direct image of $\Omega_{\tilde{E}_{1}(Mp,v)\slash \tilde{\mathfrak{X}}(\Gamma_{1}(Mp))(v)}$ by $f_{v}$. 
\begin{Definition}
For each $v\in \mathbb{Q}\cap [0,p^{p(p+1)^{-1}})$ and $k\in\mathbb{Z}$, we call a global section $f$ of $\tilde{\omega}_{v}^{\otimes k}$ an overconvergent modular form over $\tilde{\mathfrak{X}}(\Gamma_{1}(Mp))(v)$.
\end{Definition}
For each overconvergent modular form $f$, we can define the $q$-expansion at each cusp in the same way as classical modular forms. We say that an overconvergent modular form $f$ is cuspidal, if $f$ has no constant term at each cusp. We denote by $S_{k}(Mp,v;K)$ the $K$-Banach space of overconvergent cusp forms of weight $k$ level $Mp$ over $\tilde{\mathfrak{X}}(\Gamma_{1}(Mp))(v)$. Further, for each $d\in (\mathbb{Z}\slash Mp\mathbb{Z})^{\times}$, we can define the diamond operator $\langle d\rangle:S_{k}(Mp,v;K)\rightarrow S_{k}(Mp,v;K)$. Then, we have the characteristic decomposition 
$$S_{k}(Mp,v;K):=\displaystyle{\bigoplus_{\chi}}S_{k}(Mp,\chi,v;K),$$
where $\chi$ runs through Dirichlet characters modulo $Mp$.  As an important example, there exists an overconvergent modular form
$$E:=1+\frac{2}{L_{p}(0,\mathbf{1})}\displaystyle{\sum_{n\geq1}}\left(\displaystyle{\sum_{d\vert n}}\omega_{p}^{-1}(d) \right)q^{n}\in \mathbb{Z}_{p}\jump{q}\leqno(\mathrm{A}.1.2)$$
of weight 1 and Nebentypus $\omega_{p}^{-1}$ which is defined in \cite[p.447]{Col97}. Here, $L_{p}(s,\mathbf{1})$ is the $p$-adic $L$-function defined in \cite[p.88]{Hid93}. By \cite[Lemma B1.2]{Col97}, there exists an overconvergent modular form $E^{-1}$ of weight -1 whose $q$-expansion is the inverse of the $q$-expansion of $E$. We fix a Dirichlet character $\epsilon_{M}$ modulo $M$. Because $E$ is invertible, we can define an isomorphism
$$S_{0}(Mp,\epsilon_{M}\omega_{p}^{i},v;K)\xrightarrow{\times E^{k}} S_{k}(Mp,\epsilon_{M}\omega_{p}^{i-k},v;K).\leqno(\mathrm{A}.1.3)$$

Next, we recall the definition of families of overconvergent modular forms. For each 1-dimensional $K$-affinoid closed ball $\mathcal{B}$, because $E-1\in q\mathbb{Z}_{p}\jump{q}$, we can define a formal series
$$E^{s}:=\displaystyle{\sum_{n\geq 0}}\left(
    \begin{array}{c}
      s \\
      n \\
    \end{array}
  \right)(E-1)^{n}\in A(\mathcal{B})\jump{q},\leqno(\mathrm{A}.1.4)$$
where
$$\left(
    \begin{array}{c}
      s \\
      n \\
    \end{array}
  \right):=\begin{cases}
\frac{s(s-1)\cdots (s-n+1)}{n!}\ &\mathrm{if}\ n\geq 0,\\
1\ &\mathrm{if}\ n=0.
\end{cases}
$$
\begin{Definition}
We call a formal series $G=\displaystyle{\sum_{n\geq1}}a_{n}(x)q^{n}\in A(\mathcal{B})\jump{q}$ a family of overconvergent cusp forms of type $i$ and Nebentypus $\epsilon_{M}$ over $\tilde{\mathfrak{X}}(\Gamma_{1}(Mp))(v)$ if the series
$$G(s)\slash E^{s}:=(\displaystyle{\sum_{n\geq 1}}a_{n}(s)q^{n})E^{-s}\in A(\mathcal{B})\jump{q}$$
is the $q$-expansion of an element of $S_{0}(Mp,\epsilon_{M}\omega_{p}^{i},v;K)\widehat{\otimes}_{K}A(\mathcal{B})$. 
\end{Definition}
We denote by $S(Mp,\epsilon_{M},i,v;A(\mathcal{B}))$ the  $A(\mathcal{B})$-module consisting of families of overconvergent cusp forms of type $i$ and Nebentypus $\epsilon_{M}$ over $\tilde{\mathfrak{X}}(\Gamma_{1}(Mp))(v)$. By the definition, the $A(\mathcal{B})$-module $S(Mp,\epsilon_{M},i,v;A(\mathcal{B}))$ is isomorphic to $S_{0}(Mp,\epsilon_{M}\omega_{p}^{i},v;K)\widehat{\otimes}_{K}A(\mathcal{B})$. Then, we naturally regard $S(Mp,\epsilon_{M},i,v;A(\mathcal{B}))$ as an $A(\mathcal{B})$-Banach module.

Let $\alpha\in \mathbb{Q}_{>0}$ be a rational number. Next, we recall the construction of the projection of $S(Mp,\epsilon_{M},i,v;A(\mathcal{B}))$ to the slope $\alpha$ part. By \cite[p.220]{Col96}, we can define the Hecke operator at $p$ 
$$T_{p}:S_{k}(Mp,\epsilon_{M}\omega_{p}^{i-k},v;K)\rightarrow S_{k}(Mp,\epsilon_{M}\omega_{p}^{i-k},v;K)\leqno(\mathrm{A}.1.5)$$ 
and this operator is known to be completely continuous by \cite[Proposition I\hspace{-.1em}I 3.15]{Go88}. Hence, by \cite[\S5]{Ser62}, we can define the Fredholm determinant of $T_{p}$ by
$$P_{k}(X):=\mathrm{det}(I-XT_{p})\in K\jump{X}.\leqno(\mathrm{A}.1.6)$$
By the isomorphism of $S_{0}(Mp,\epsilon_{M}\omega_{p}^{i},v;K)\xrightarrow{\times E^{k}} S_{k}(Mp,\epsilon_{M}\omega_{p}^{i-k},v;K)$, we can naturally define a complete continuous operator $\tilde{T}_{k,p}:=E^{-k}T_{p}E^{k}\in \mathrm{End}_{K}(S_{0}(Mp,\epsilon_{M}\omega_{p}^{i},v;K))$. If we take a sufficiently small $v\in \mathbb{Q}\cap [0,p^{p(p+1)^{-1}})$, we can interpolate the family $\{\tilde{T}_{k,p}\}_{k\in \mathcal{B}(K)\cap \mathbb{Z}}$ to a complete continuous $A(\mathcal{B})$-module  homomorphism
$$\tilde{T}_{s,p}: S_{0}(Mp,\epsilon_{M}\omega_{p}^{i},v;K)\widehat{\otimes}_{K}A(\mathcal{B})\rightarrow S_{0}(Mp,\epsilon_{M}\omega_{p}^{i},v;K)\widehat{\otimes}_{K}A(\mathcal{B})\leqno(\mathrm{A}.1.7)$$
by \cite[Lemma B3.1]{Col97}. Because $\tilde{T}_{s,p}$ is complete continuous, we can define the Fredholm determinant of $\tilde{T}_{s,p}$
$$P_{s}=\mathrm{det}(1-X\tilde{T}_{s,p})\in A(\mathcal{B})\jump{X}\leqno(\mathrm{A}.1.8)$$
by \cite[p.428]{Col97}. This series $P_{s}$ is the interpolation of $\{P_{k}\}_{k\in \mathcal{B}(K)\cap \mathbb{Z}}$. We fix an integer $k_{0}\in \mathbb{Z}_{> \alpha+1}$. By \cite[Corollary A5.5.1]{Col97}, if we take a sufficiently small $\epsilon \in \mathbb{Q}_{>0}$, there exists a monic polynomial $Q_{s}(X):=1+a_{0}(s)X+\cdots a_{d}(s)X^{d}\in A(\mathcal{B}_{[k_{0},\epsilon]})[X]$ of degree $d$ which satisfies the following conditions.
\begin{enumerate}
\item[1] The coefficient $a_{d}(s)$ is in $A(\mathcal{B}_{[k_{0},\epsilon]})^{\times}$.\\
\item[2] If a pair $(t,z)\in \mathcal{B}_{[k_{0},\epsilon]}(\overline{\mathbb{Q}}_{p})\times \mathrm{Spec}(K[X])$ is a root of $Q_{s}(X)$, we have $\vert z\vert_{p}=\vert p\vert_{p}^{\alpha}$.\\
\item[3] There exists a series $S_{s}(X)\in A(\mathcal{B}_{[k_{0},\epsilon]})\jump{X}$ such that $P_{s}(X)=Q_{s}(X)S_{s}(X)$.
\end{enumerate}
We put $Q^{*}_{s}(X):=X^{d}Q_{s}(X^{-d})\in A(\mathcal{B}_{[k_{0},\epsilon]})[X]$. We define the complete continuous $A(\mathcal{B}_{[k_{0},\epsilon]})$-morphism $T_{s,p}:=E^{s}\tilde{T}_{s,p}E^{-s}\in\mathrm{End}_{A(\mathcal{B}_{[k_{0},\epsilon]})}(S(Mp,\epsilon_{M},i,v;A(\mathcal{B}_{[k_{0},\epsilon]})))$. We put 
$$S^{\alpha}(Mp,\epsilon_{M},i,v;A(\mathcal{B}_{[k_{0},\epsilon]})):=\mathrm{Ker}(Q^{*}(T_{s,p})).$$
By \cite[Proposition A5.3]{Col97}, the $A(\mathcal{B}_{[k_{0},\epsilon]})$-module $S^{\alpha}(Mp,\epsilon_{M},i,v;A(\mathcal{B}_{[k_{0},\epsilon]}))$ is projective of rank $d$. Because $A(\mathcal{B}_{[k_{0},\epsilon]})$ is a principal ideal domain,  $S^{\alpha}(Mp,\epsilon_{M},i,v;A(\mathcal{B}_{[k_{0},\epsilon]}))$ is free of rank $d$ and independent of $v$. Then, we omit $v$ in the notation and we write
$$S^{\alpha}(Mp,\epsilon_{M},i;A(\mathcal{B}_{[k_{0},\epsilon]})):=S^{\alpha}(Mp,\epsilon_{M},i,v;A(\mathcal{B}_{[k_{0},\epsilon]})).\leqno(\mathrm{A}.1.9)$$
Further, by \cite[Theorem A4.3]{Col97}, there exists a projection 
$$p_{\alpha}: S(Mp,\epsilon_{M},i,v;A(\mathcal{B}_{[k_{0},\epsilon]}))\rightarrow S^{\alpha}(Mp,\epsilon_{M},i;A(\mathcal{B}_{[k_{0},\epsilon]})).\leqno(\mathrm{A}.1.10)$$

Next, we recall some results on the Hecke algebra of $S^{\alpha}(Mp,\epsilon_{M},i;A(\mathcal{B}_{[k_{0},\epsilon]}))$. Let $\mathbf{T}$ be the Hecke algebra of $S(Mp,\epsilon_{M},i,v;A(\mathcal{B}_{[k_{0},\epsilon]}))$ defined in \cite[p.464]{Col97}. We define the Hecke algebra $\mathbf{T}^{\alpha}$ of $S^{\alpha}(Mp,\epsilon_{M},i;A(\mathcal{B}_{[k_{0},\epsilon]}))$ to be the image of $\mathbf{T}$ in $\mathrm{End}_{A(\mathcal{B}_{[k_{0},\epsilon]})}(S^{\alpha}(Mp,\epsilon_{M},i;A(\mathcal{B}_{[k_{0},\epsilon]})))$. Because $\mathrm{End}_{A(\mathcal{B}_{[k_{0},\epsilon]})}(S^{\alpha}(Mp,\epsilon_{M},i;A(\mathcal{B}_{[k_{0},\epsilon]})))$ is $A(\mathcal{B}_{[k_{0},\epsilon]})$-free of rank $d^{2}$ and $A(\mathcal{B}_{[k_{0},\epsilon]})$ is a principal ideal domain, $\mathbf{T}^{\alpha}$ is also $A(\mathcal{B}_{[k_{0},\epsilon]})$-free of finite rank. By \cite[9.4.3 Theorem 3]{BGR84} and \cite[3.7.3 Proposition 3]{BGR84}, $\mathbf{T}^{\alpha}$ becomes an $A(\mathcal{B}_{[k_{0},\epsilon]})$-Banach algebra and a $K$-affinoid algebra. We denote by $X(\mathbf{T}^{\alpha})$ the $K$-affinoid space attached to $T^{\alpha}$. We denote by $(i,i^{\#}):X(\mathbf{T}^{\alpha})\rightarrow \mathcal{B}_{[k_{0},\epsilon]}$ the morphism of the $K$-affinoid space attached to the homomorphism $A(\mathcal{B}_{[k_{0},\epsilon]})\rightarrow \mathbf{T}^{\alpha}$. Here, $i:X(\mathbf{T}^{\alpha})\rightarrow \mathcal{B}_{[k_{0},\epsilon]}$ is the topological map and $i^{\#}:A(\mathcal{B}_{[k_{0},\epsilon]})\rightarrow \mathbf{T}^{\alpha}$ is the $K$-affinoid algebra map. We define a series
$$F_{\mathbf{T}^{\alpha}}:=\displaystyle{\sum_{n\geq 1}}T_{n}q^{n}\in \mathbf{T}^{\alpha}\jump{q},\leqno(\mathrm{A}.1.11)$$
where $T_{n}\in \mathbf{T}^{\alpha}$ is the Hecke operator at $n$. We fix a finite extension $L$ of $K$. We denote by $X(\mathbf{T}^{\alpha})(L)$ the set of points of $X(\mathbf{T}^{\alpha})$ over  $L$. For each $x\in X(\mathbf{T}^{\alpha})(L)$, we denote by $\eta_{x}:\mathbf{T}^{\alpha}\rightarrow L$ the continuous $K$-algebra homomorphism attached to $x$. For each $k\in \mathcal{B}_{[k_{0},\epsilon]}(K)$, we put $X(\mathbf{T}^{\alpha})_{k}(L):=i^{-1}(k)\cap X(\mathbf{T}^{\alpha})(L)$. The following theorem is proved in \cite[Theorem B5.7]{Col97}.
\begin{Theorem}
Let $k\in \mathcal{B}_{[k_{0},\epsilon]}(K)\cap \mathbb{Z}_{>\alpha+1}$. The mapping from $X(\mathbf{T}^{\alpha})_{k}(L)$ to $L\jump{q}$, $x\mapsto F_{\mathbf{T}^{\alpha},x}:=\displaystyle{\sum_{n\geq1}}\eta_{x}(T_{n})q^{n}$, is a bijection onto the set of $q$-expansions of classical cuspidal eigenforms of weight $k$, level Mp and Nebentypus $\epsilon_{M}\omega_{p}^{i-k}$ of slope $\alpha$ over $L$.
\end{Theorem}
Next, we recall the construction of Coleman families.  Let $f_{0}$ be a $p$-stabilized newform of weight $k_{0}$, level $Mp$ , slope $\alpha$ and Nebentypus $\epsilon_{M}\omega_{p}^{i-k_{0}}$ over $K$. Further, we assume $a(p,f_{0})^{2}\neq\epsilon_{M}(p)p^{k_{0}-1}$ if $i=0$. We denote by $x_{0}\in X(\mathbf{T}^{\alpha})_{k_{0}}(K)$ the maximal ideal attached to $f_{0}$ by Theorem A.1.3. The following corollary is proved in \cite[Corollary B.5.7.1]{Col97}.
\begin{Corollary}
If we take a sufficiently small $\epsilon\in \mathbb{Q}_{>0}$, there exists a unique section
$$(s,s^{\#}):\mathcal{B}_{[k_{0},\epsilon]}\rightarrow X(\mathbf{T}^{\alpha})$$
such that $s(k_{0})=x_{0}$.
\end{Corollary}
Let $(s,s^{\#}):\mathcal{B}_{[k_{0},\epsilon]}\rightarrow X(\mathbf{T}^{\alpha})$ be the section in Corollary A.1.4.  Then, if we put 
$$F=\displaystyle{\sum_{n\geq 1}}s^{\#}(T_{n})q^{n}\in A(\mathcal{B}_{[k_{0},\epsilon]})\jump{q},\leqno(\mathrm{A}.1.12)$$
by Theorem A.1.3, $F$ is a family of cuspidal automorphic  forms of slope $\alpha$ such that $F(k_{0})=f_{0}$.
\subsection{Rigidity of conductors in a Coleman family.}
In this subsection, we prove the rigidity of conductors in a Coleman family. Let us keep the notation as A.1. First, we define old overconvergent cusp forms.
\begin{Definition}
\begin{enumerate}
\item[(1)]Let $k$ be an integer. We fix an element $v\in \mathbb{Q}_{>0}\cap[0,p^{p(p+1)^{-1}})$. We say that an element $f\in S_{k}(Mp,\epsilon_M\omega_{p}^{i-k};K)(v)$ is old outside of $p$ if it is contained in the subspace
$$\sum_{\substack{t\vert M \\ t\neq M \\ m_{\chi}\vert t}}\sum_{m\vert\frac{M}{t}}S_{k}\left(Mp\slash t,\epsilon_{M}\omega_{p}^{i-k};K\right)(v)[m],$$
where 
$$S_{k}\left(Mp\slash t,\epsilon_{M}\omega_{p}^{i-k};K\right)(v)[m]:=\left\{g(q^{m})\middle| g\in S_{k}\left(Mp\slash t,\epsilon_{M}\omega_{p}^{i-k};K\right)(v)\right\}.$$\\
\item[(2)]We say that an element $G\in S(Mp,\epsilon_{M},i,v;A(\mathcal{B}_{[k_{0},\epsilon]}))$ is old outside of $p$ if for any $k\in \mathbb{Z}_{>\alpha+1}\cap \mathcal{B}(K)$, $G\slash E^{-k}$ is old outside of $p$.
\end{enumerate}
\end{Definition}
We denote by $S^{\mathrm{old}}(Mp,\epsilon_{M},i,v;A(\mathcal{B}_{[k_{0},\epsilon]})$ the submodule of $S(Mp,\epsilon_{M},i,v;A(\mathcal{B}_{[k_{0},\epsilon]}))$ consisting of elements which are old outside of $p$.
\begin{Lemma}
We fix an element $k\in \mathbb{Z}_{>\alpha+1}\cap \mathcal{B}(K)$. For any $g\in S_{k}(Mp,\epsilon_M\omega_{p}^{i-k};K)(v)$ which is old outside of $p$, there exists an element $G\in S^{\mathrm{old}}(Mp,\epsilon_{M},i,v;A(\mathcal{B}_{[k_{0},\epsilon]})$ such that the specialization of $G$ at $k$ equals to $g$.
\end{Lemma} 
\begin{proof}
We put $E^{s-k}g$. Because $g$ is old outside of $p$, $E^{s-k}g\in S^{\mathrm{old}}(Mp,\epsilon_{M},i,v;A(\mathcal{B}_{[k_{0},\epsilon]})$. We complete the proof.
\end{proof}
We set
$$S^{\alpha,\mathrm{old}}(Mp,\epsilon_{M},i;A(\mathcal{B}_{[k_{0},\epsilon]}):=p_{\alpha}(S^{\mathrm{old}}(Mp,\epsilon_{M},i,v;A(\mathcal{B}_{[k_{0},\epsilon]})),$$
where $p_{\alpha}$ is the projection defined in (A.1.10).
We define the ideal $\mathbf{t}$ of $\mathbf{T}^{\alpha}$ to be
$$\mathbf{t}:=\mathrm{Ann}_{\mathbf{T}^{\alpha}}(S^{\alpha,\mathrm{old}}(Mp,\epsilon_{M},i;A(\mathcal{B}_{[k_{0},\epsilon]})).$$
For each $k\in \mathbb{Z}_{>\alpha+1}\cap \mathcal{B}(K)$, let $\eta_{k}:A(\mathcal{B}_{[k_{0},\epsilon]})\rightarrow K$ be the $K$-algebra homomorphism attached to $k$. By $\eta_{k}$, we define the $K$-Banach module
$$S_{k}^{\alpha}(Mp,\epsilon_{M},i;K):=S^{\alpha}(Mp,\epsilon_{M},i;A(\mathcal{B}_{[k_{0},\epsilon]}))\widehat{\otimes}_{A(\mathcal{B}_{[k_{0},\epsilon]})}K$$
and the $K$-affinoid algebra
$$\mathbf{T}^{\alpha}_{k}:=\mathbf{T}^{\alpha}\widehat{\otimes}_{A(\mathcal{B}_{[k_{0},\epsilon]}))}K.$$
By \cite[Theorem 8.1]{Col96}, $S_{k}^{\alpha}(Mp,\epsilon_{M},i;K)$ is a $K$- subvector space of $q$-expansions of classical cusp forms of weight $k$, level $Mp^{i-k}$ and Nebentypus $\epsilon_{M}\omega_{p}^{i-k}$ over $K$. Further, $\mathbf{T}^{\alpha}_{k}$ is the Hecke algebra of $S_{k}^{\alpha}(Mp,\epsilon_{M},i;K)$. We define a $K$-subvector space $S_{k}^{\alpha,\mathrm{old}}(Mp,\epsilon_{M},i;K)$ of $S_{k}^{\alpha}(Mp,\epsilon_{M},i;K)$ to be the image of the natural map
$$S^{\alpha,\mathrm{old}}(Mp,\epsilon_{M},i;A(\mathcal{B}_{[k_{0},\epsilon]}))\widehat{\otimes}_{A(\mathcal{B}_{[k_{0},\epsilon]})}K\rightarrow S_{k}^{\alpha}(Mp,\epsilon_{M},i;K).$$
By Lemma A.2.2, the space $S_{k}^{\alpha,\mathrm{old}}(Mp,\epsilon_{M},i;K)$ equals to the subspace of $S_{k}^{\alpha}(Mp,\epsilon_{M},i;K)$ consisting of elements which are old outside of $p$. We define an ideal $\mathbf{t}_{k}$ to be the image of the natural map
$$\mathbf{t}\widehat{\otimes}_{A(\mathcal{B}_{[k_{0},\epsilon]})}K\rightarrow \mathbf{T}^{\alpha}_{k}.\leqno(\mathrm{A}.2.1)$$
\begin{Lemma}
The map (A.2.1) is injective.
\end{Lemma}
\begin{proof}
Let $t_{1},\cdots,t_{n}$ be a basis of $\mathbf{T}^{\alpha}$. Let $t\in \mathbf{t}$ be an element such that the image of $t$ by the map $\mathbf{t}\widehat{\otimes}_{A(\mathcal{B}_{[k_{0},\epsilon]})}K\rightarrow \mathbf{T}^{\alpha}_{k}$ equals to zero. It suffices to prove that there exists an element $b\in \mathrm{Ker}(\eta_{k})$ and $t^{\prime}\in \mathbf{t}$ such that $t=bt^{\prime}$. If we put $t=\displaystyle{\sum_{i=1}^{n}b_{i}t_{i}}$, where $b_{i}\in A(\mathcal{B}_{[k_{0},\epsilon]})$, by the assumption, we have $b_{i}\in \mathrm{Ker}(\eta_{k})$. Because $A(\mathcal{B}_{[k_{0},\epsilon]})$ is a principal ideal domain, there exists a generator $c$ of $\mathrm{Ker}(\eta_{k})$. Then, there exists an element $t^{*}\in \mathbf{T}^{\alpha}$ such that $t=ct^{*}$. Because $A(\mathcal{B}_{[k_{0},\epsilon]})$ is integral, we have $t^{*}\in \mathbf{t}$. We complete the proof.
\end{proof}
\begin{Lemma}
For each $k\in\mathbb{Z}_{>\alpha+1}\cap \mathcal{B}(K)$, we have
$$\mathbf{t}_{k}:=\mathrm{Ann}_{\mathbf{T}^{\alpha}_{k}}(S_{k}^{\alpha,\mathrm{old}}(Mp,\epsilon_{M},i;K)).$$
\end{Lemma}
\begin{proof}
Let $\langle\ ,\ \rangle: S^{\alpha}(Mp,\epsilon_{M},i;A(\mathcal{B}_{[k_{0},\epsilon]})\times \mathbf{T}^{\alpha}\rightarrow A(\mathcal{B}_{[k_{0},\epsilon]})$ be the bilinear map defined by 
$$\langle G,T\rangle:=a(1,T(G))$$
for each $G\in S^{\alpha}(Mp,\epsilon_{M},i;A(\mathcal{B}_{[k_{0},\epsilon]})$ and $T\in \mathbf{T}^{\alpha}$. By \cite[Proposition B5.6]{Col97}, the bilinear map $\langle\ ,\ \rangle$ is perfect. Then, we have two isomorphisms 
\begin{align*}
&\mathbf{T}^{\alpha}\cong \mathrm{Hom}_{A(\mathcal{B}_{[k_{0},\epsilon]})}(S^{\alpha}(Mp,\epsilon_{M},i;A(\mathcal{B}_{[k_{0},\epsilon]})), A(\mathcal{B}_{[k_{0},\epsilon]}),\\
&\mathbf{t}\cong \mathrm{Hom}_{A(\mathcal{B}_{[k_{0},\epsilon]})}(S^{\alpha}(Mp,\epsilon_{M},i;A(\mathcal{B}_{[k_{0},\epsilon]}))\slash S^{\alpha,\mathrm{old}}(Mp,\epsilon_{M},i;A(\mathcal{B}_{[k_{0},\epsilon]})), A(\mathcal{B}_{[k_{0},\epsilon]})).\end{align*}
Then, we have the following exact sequence:
$$0\rightarrow \mathbf{t}\rightarrow \mathbf{T}^{\alpha}\rightarrow \mathrm{Hom}_{A(\mathcal{B}_{[k_{0},\epsilon]})}(S^{\alpha,\mathrm{old}}(Mp,\epsilon_{M},i;A(\mathcal{B}_{[k_{0},\epsilon]})), A(\mathcal{B}_{[k_{0},\epsilon]})).$$
We put $D:=\mathrm{Hom}_{A(\mathcal{B}_{[k_{0},\epsilon]})}(S^{\alpha,\mathrm{old}}(Mp,\epsilon_{M},i;A(\mathcal{B}_{[k_{0},\epsilon]})), A(\mathcal{B}_{[k_{0},\epsilon]}))$. By Lemma A.2.3, we have the following exact sequence:
$$0\rightarrow \mathbf{t}\widehat{\otimes}_{A(\mathcal{B}_{[k_{0},\epsilon]})}K\rightarrow \mathbf{T}^{\alpha}_{k}\rightarrow D\widehat{\otimes}_{A(\mathcal{B}_{[k_{0},\epsilon]})}K.$$
On the other hand, by the definition of $S_{k}^{\alpha,\mathrm{old}}(Mp,\epsilon_{M},i;K)$, we have the following exact sequence:
\begin{align*}
S^{\alpha,\mathrm{old}}(Mp,\epsilon_{M},i;A(\mathcal{B}_{[k_{0},\epsilon]}))\widehat{\otimes}_{A(\mathcal{B}_{[k_{0},\epsilon]})}K&\rightarrow S_{k}^{\alpha}(Mp,\epsilon_{M},i;K)\rightarrow\\
 &S_{k}^{\alpha}(Mp,\epsilon_{M},i;K)\slash S_{k}^{\alpha,\mathrm{old}}(Mp,\epsilon_{M},i;K)\rightarrow 0.
\end{align*}
Then, by the duality, we have the following exact sequence:
$$0\rightarrow\mathrm{Ann}_{\mathbf{T}^{\alpha}_{k}}(S_{k}^{\alpha,\mathrm{old}}(Mp,\epsilon_{M},i;K))\rightarrow \mathbf{T}^{\alpha}_{k}\rightarrow \mathrm{Hom}_{K}(S^{\alpha,\mathrm{old}}(Mp,\epsilon_{M},i;A(\mathcal{B}_{[k_{0},\epsilon]}))\widehat{\otimes}_{A(\mathcal{B}_{[k_{0},\epsilon]})}K,K).$$
We put $E=\mathrm{Hom}_{K}(S^{\alpha,\mathrm{old}}(Mp,\epsilon_{M},i;A(\mathcal{B}_{[k_{0},\epsilon]}))\widehat{\otimes}_{A(\mathcal{B}_{[k_{0},\epsilon]})}K,K)$. Then, we have the following commutative diagram:
$$\xymatrix{ 0 \ar[r]& \mathbf{t}\widehat{\otimes}_{A(\mathcal{B}_{[k_{0},\epsilon]})}K \ar[r] \ar[d] &\mathbf{T}^{\alpha}_{k} \ar{[r]}\ar@{=}[d]& D\widehat{\otimes}_{A(\mathcal{B}_{[k_{0},\epsilon]})}K\ar[d] 
\\ 0 \ar[r] & \mathrm{Ann}_{\mathbf{T}^{\alpha}_{k}}(S_{k}^{\alpha,\mathrm{old}}(Mp,\epsilon_{M},i;K)) \ar[r]  & \mathbf{T}^{\alpha}_{k} \ar[r] & E}.$$
Because $S^{\alpha,\mathrm{old}}(Mp,\epsilon_{M},i;A(\mathcal{B}_{[k_{0},\epsilon]}))$ is a finite $A(\mathcal{B}_{[k_{0},\epsilon]})$-free module, we have 
$$D\widehat{\otimes}_{A(\mathcal{B}_{[k_{0},\epsilon]})}K\cong E.$$
Then, we have 
$$\mathbf{t}\widehat{\otimes}_{A(\mathcal{B}_{[k_{0},\epsilon]})}K\cong \mathrm{Ann}_{\mathbf{T}^{\alpha}_{k}}(S_{k}^{\alpha,\mathrm{old}}(Mp,\epsilon_{M},i;K)).$$
We complete the proof.
\end{proof}
\begin{Lemma}
Let $k$ be an integer such that $k\in \mathbb{Z}_{>\alpha+1}\cap \mathcal{B}_{[k_{0},\epsilon]}(K)$ and $L$ is a finite extension of $K$. We fix a point $x\in X(\mathbf{T}^{\alpha})_{k}(L)$. The following conditions are equivalent.
\begin{enumerate}
\item[(1)] The specialization $F_{\mathbf{T}^{\alpha},x}$ is a $p$-stabilized newform.
\item[(2)] We have $x\not\supset \mathbf{t}$.
\item[(3)] We have $x+\mathbf{t}=\mathbf{T}^{\alpha}$.
\end{enumerate} 
\end{Lemma}
\begin{proof}
Because $x$ is a maximal ideal, the equivalence of (2) and (3) is trivial. We prove the equivalence of (1) and (2). we define an ideal $x_{k}$ of $\mathbf{T}^{\alpha}_{k}$ to be the image of $x$ via the natural map
$$\mathbf{x}\widehat{\otimes}_{A(\mathcal{B}_{[k_{0},\epsilon]})}K\rightarrow \mathbf{T}^{\alpha}_{k}.$$
By the definition of $x$, $x_{k}$ is a maximal ideal of $\mathbf{T}^{\alpha}_{k}$. The condition (2) is equivalent to the following condition:
\begin{enumerate}
\item[$(\mathrm{2})^{\prime}$] $x_{k}\not\supset \mathbf{t}_{k}$.
\end{enumerate}
In fact, if the condition (2) holds, we have $x+\mathbf{t}=\mathbf{T}^{\alpha}$. Then, we have $x_{k}+{t}_{k}=\mathbf{T}^{\alpha}_{k}$. If the condition  $(\mathrm{2})^{\prime}$ holds, it is trivial that $x\not\supset \mathbf{t}$. Then, it suffices to prove the equivalence of (1) and $(\mathrm{2})^{\prime}$. Let $S^{\alpha,\mathrm{new}}_{k}(Mp,\epsilon_{M},i;K)$ be the subspace of $S^{\alpha}_{k}(Mp,\epsilon_{M},i;K)$ of consisting of cusp forms whose images are zero under any of degeneracy trace map for any proper divisor $d$ for $M$. Then, we have the decomposition 
$$S^{\alpha}_{k}(Mp,\epsilon_{M},i;K):=S^{\alpha,\mathrm{new}}_{k}(Mp,\epsilon_{M},i;K)\oplus S_{k}^{\alpha,\mathrm{old}}(Mp,\epsilon_{M},i;K).$$
Because $S^{\alpha,\mathrm{new}}_{k}(Mp,\epsilon_{M},i;K)$ and $S_{k}^{\alpha,\mathrm{old}}(Mp,\epsilon_{M},i;K)$ are stable under the action of $\mathbf{T}^{\alpha}_{k}$, we can define the Hecke algebra $\mathbf{T}^{\alpha,\mathrm{new}}_{k}$ and $\mathbf{T}^{\alpha,\mathrm{old}}_{k}$ of $S^{\alpha,\mathrm{new}}_{k}(Mp,\epsilon_{M},i;K)$ and $S_{k}^{\alpha,\mathrm{old}}(Mp,\epsilon_{M},i;K)$ respectively. Then, we have the decomposition
$$\mathbf{T}^{\alpha}_{k}:=\mathbf{T}^{\alpha,\mathrm{new}}_{k}\oplus\mathbf{T}^{\alpha,\mathrm{old}}_{k}.$$
By Lemma A.2.4, we have
$$\mathbf{t}_{k}=\mathbf{T}^{\alpha,\mathrm{new}}_{k}.$$
Let $\eta_{x,k}:\mathbf{T}^{\alpha}_{k}\rightarrow L$ be the $K$-algebra homomorphism attached to $x_{k}$. This $K$-algebra homomorphism $\eta_{x,k}$ corresponds to the cusp form $F_{\mathbf{T}^{\alpha},x}$ by the duality. We assume the condition (1). Then, the $K$-algebra homomorphism $\eta_{x,k}$ factors through $\mathbf{T}^{\alpha,\mathrm{new}}_{k}$. Then, $(1,0)\in \mathbf{T}^{\alpha}_{k}:=\mathbf{T}^{\alpha,\mathrm{new}}_{k}\oplus\mathbf{T}^{\alpha,\mathrm{old}}_{k}$ is not contained in $x$. Then, we have $(2)^{\prime}$. In order to prove the opposite implication, we assume the condition $(2)^{\prime}$. Then, $(1,0)\in \mathbf{T}^{\alpha}_{k}:=\mathbf{T}^{\alpha,\mathrm{new}}_{k}\oplus\mathbf{T}^{\alpha,\mathrm{old}}_{k}$ is not contained in $x$. Then, $\eta_{x,k}$ factors through  $\mathbf{T}^{\alpha,\mathrm{new}}_{k}$. Then, $F_{\mathbf{T}^{\alpha},x}$ is primitive outside of $p$. We complete the proof.
\end{proof}
Let $x_{0}\in \mathfrak{X}(\mathbf{T}^{\alpha})_{k_{0}}(K)$ be the point attached to $f_{0}$ by Theorem A.1.3. Let $(s,s^{\#}):\mathcal{B}_{[k_{0},\epsilon]}\rightarrow X(\mathbf{T}^{\alpha})$ be the section in Corollary A.1.4. We prove the rigidity of conductors in a Coleman family.
\begin{Proposition}
There exists a closed affinoid  sub-disc $\mathcal{B}_{[k_{0},\epsilon^{\prime}]}\subset \mathcal{B}_{[k_{0},\epsilon]}$ with $\epsilon^{\prime}\in \mathbb{Q}_{>0}$ such that for any $k\in \mathcal{B}_{[k_{0},\epsilon^{\prime}]}(K)\cap \mathbb{Z}_{>\alpha+1}$, we have $s(k)\not\supset \mathbf{t}$.
\end{Proposition}
\begin{proof}
Let $t_{1},\cdots,t_{n}$ be a generator of the ideal $\mathbf{t}$. By Lemma A.2.5, there exists an element $t_{i}$ for $1\leq i\leq n$ such that $t_{i}(x_{0}): =t_{i}\ \mathrm{mod}\ x$ is non-zero. Then, there exists an element $b\in K^{\times}$ such that 
$$\vert t_{i}(x_{0})\vert_{p}>\vert b\vert_{p}.$$
We define a sub-affinoid space of $\mathbf{T}^{\alpha}$ 
$$D:=\{x\in X(\mathbf{T}^{\alpha})\mid \vert t_{i}(x)\vert_{p}\geq \vert b\vert_{p}\}.$$
By the definition of $D$, we have $x_{0}\in D$. Further,  the $K$-algebra homomorphism attached to the inclusion $D\hookrightarrow X(\mathbf{T}^{\alpha})$ is obtained by the natural homomorphism 
$$\mathbf{T}^{\alpha}\rightarrow \mathbf{T}^{\alpha}\langle X\rangle\slash (t_{i}X-b).$$
Then, the elements $x$ of $D$ corresponds to the maximal ideals $\tilde{x}$ of $\mathbf{T}^{\alpha}\langle X\rangle\slash (t_{i}X-b)$. Let $x\in D$ be an element. We assume that $t_{i}\in x$. Then, we have $t_{i}\in\tilde{x}$. Because $Xt_{i}$ is invertible in $\mathbf{T}^{\alpha}\langle X\rangle\slash (t_{i}X-b)$, it is a contradiction. Hence, we have $x\not\supset \mathbf{t}$ for any element $x\in D$. We define a sub-rigid space of $D$
$$U:=\{x\in D\mid \vert t_{i}(x)\vert_{p}> \vert b\vert_{p}\}.$$
We have $k_{0}\in s^{-1}(U)\subset \mathcal{B}_{[k_{0},\epsilon]}$ and we can take a small closed sub-disc $\mathcal{B}_{[k_{0},\epsilon^{\prime}]}\subset s^{-1}(U)$ with $\epsilon^{\prime}\in\mathbb{Q}_{>0}$. By the definition of $U$, for any $k\in \mathcal{B}_{[k_{0},\epsilon^{\prime}]}$, we have $s(k)\in D$. Hence, $s(k)\not\supset \mathbf{t}$. We complete the proof.
\end{proof}
Let $\mathcal{B}_{[k_{0},\epsilon^{\prime}]}\subset \mathcal{B}_{[k_{0},\epsilon]}$ be a closed sub-affinoid disc with $\epsilon^{\prime}\in \mathbb{Q}_{>0}$ such that we have $s(k)\not\supset \mathbf{t}$ for any $k\in \mathcal{B}_{[k_{0},\epsilon^{\prime}]}(K)\cap \mathbb{Z}_{>\alpha+1}$. We denote by $j^{\#}: A(\mathcal{B}_{[k_{0},\epsilon]})\rightarrow A(\mathcal{B}_{[k_{0},\epsilon^{\prime}]})$ the $K$-algebra homomorphism attached to the inclusion $\mathcal{B}_{[k_{0},\epsilon^{\prime}]}\hookrightarrow \mathcal{B}_{[k_{0},\epsilon]}$. Let $F:=\displaystyle{\sum_{n\geq 1}}s^{\#}(T_{n})\in A(\mathcal{B}_{[k_{0},\epsilon]})$ be the family defined in (A.1.12).
\begin{Corollary}
We put $F^{\prime}:=\displaystyle{\sum_{n\geq 1}}(j^{\#}\circ s^{\#})(T_{n})\in A(\mathcal{B}_{[k_{0},\epsilon^{\prime}]})$. Then, for any $k\in \mathcal{B}_{[k_{0},\epsilon^{\prime}]}(K)\cap \mathbb{Z}_{>\alpha+1}$, the specialization $F_{k}$ of $F$ at $k$ is a $p$-stabilized newform.
\end{Corollary}
\begin{proof}
The conclusion of the corollary is clear since $F_{k}$ is a $p$-stabilized newform for any $k\in \mathcal{B}_{[k_{0},\epsilon^{\prime}]}(K)\cap \mathbb{Z}_{>\alpha+1}$ by Lemma A.2.5.
\end{proof}

\footnote[0]{Department of Mathematics, Graduate School of Science, Osaka University, Toyonaka, Osaka 560-0043, JAPAN, E-mail address: k-fukunaga@cr.math.sci.osaka-u.ac.jp}
\end{document}